\documentclass[a4paper,11pt]{article}
\usepackage{amsmath}
\usepackage{amsthm}
\usepackage{amssymb}
\usepackage{graphicx}
\usepackage{adjustbox}
\usepackage{mathtools}
\usepackage{bbm}
\usepackage[margin=2.5cm]{geometry}
\usepackage{setspace}
\usepackage{epstopdf}
\usepackage{subfig}
\usepackage{empheq}
\usepackage{algorithm}
\usepackage{algorithmicx}
\usepackage{algpseudocode}
\usepackage{multirow}
\usepackage{url}
\usepackage[title]{appendix}
\usepackage{enumitem}
\usepackage[svgnames]{xcolor}

\makeatletter

\theoremstyle{plain}
\newtheorem{thm}{\protect\theoremname}
  \theoremstyle{definition}
  \newtheorem{defn}[thm]{\protect\definitionname}
  \newtheorem{cond}[thm]{\protect\condname}
  \theoremstyle{plain}
  \newtheorem{lem}[thm]{\protect\lemmaname}
  \newtheorem{prop}[thm]{\protect\propname}

  \providecommand{\propname}{Proposition}
  \providecommand{\condname}{Condition}
  \providecommand{\probname}{Problem}
  
  \providecommand{\keywords}[1]
{
  \small	
  \textbf{\textit{Keywords---}} #1
}

\makeatother

  \providecommand{\definitionname}{Definition}
  \providecommand{\lemmaname}{Lemma}
  \providecommand{\corollaryname}{Corollary}
\providecommand{\theoremname}{Theorem}

\usepackage{endnotes}
\usepackage{color}
\usepackage{xparse}

\title{Biwhitening Reveals the Rank of a Count Matrix}

\author{ Boris Landa${^{1,*}}$~~~~Thomas T.C.K.\ Zhang${^{2}}$~~~~Yuval Kluger${^{1,3,4}}$\\
\small{${^1}$Program in Applied Mathematics, Yale University}\\
\small{${^2}$Department of Electrical and Systems Engineering, University of Pennsylvania}\\
\small{${^3}$Interdepartmental Program in Computational Biology and Bioinformatics, Yale University}\\
\small{${^4}$Department of Pathology, Yale University School of Medicine}\\
\small{${^*}$Corresponding author. Email: boris.landa@yale.edu}
}

\begin{document}

\maketitle

\begin{abstract}
Estimating the rank of a corrupted data matrix is an important task in data analysis, most notably for choosing the number of components in PCA. Significant progress on this task was achieved using random matrix theory by characterizing the spectral properties of large noise matrices. However, utilizing such tools is not straightforward when the data matrix consists of count random variables, e.g., Poisson, in which case the noise can be heteroskedastic with an unknown variance in each entry. In this work, we consider a Poisson random matrix with independent entries, and propose a simple procedure termed \textit{biwhitening} for estimating the rank of the underlying signal matrix (i.e., the Poisson parameter matrix) without any prior knowledge. Our approach is based on the key observation that one can scale the rows and columns of the data matrix simultaneously so that 
the spectrum of the corresponding noise agrees with the standard Marchenko-Pastur (MP) law, justifying the use of the MP upper edge as a threshold for rank selection. Importantly, the required scaling factors can be estimated directly from the observations by solving a matrix scaling problem via the Sinkhorn-Knopp algorithm. 
Aside from the Poisson, our approach is extended to families of distributions that satisfy a quadratic relation between the mean and the variance, such as the generalized Poisson, binomial, negative binomial, gamma, and many others. This quadratic relation can also account for missing entries in the data.
We conduct numerical experiments that corroborate our theoretical findings, and showcase the advantage of our approach for rank estimation in challenging regimes. Furthermore, we demonstrate the favorable performance of our approach on several real datasets of single-cell RNA sequencing (scRNA-seq), High-Throughput Chromosome Conformation Capture (Hi-C), and document topic modeling.
\end{abstract}


\section{Introduction}
Principal Component Analysis (PCA) is a ubiquitous tool for processing and analyzing multivariate data~\cite{yao2015sample,jackson2005user}, and is widely used across multiple scientific fields for visualization, compression, denoising, and imputation. Yet, when applying PCA, one always faces the nontrivial task of setting the number of principal components that are retained for subsequent use.
To address this challenge, a popular approach is to assume the signal-plus-noise model, which serves as a guidance for selecting the number of components in PCA. Specifically, let $Y\in\mathbb{R}^{m\times n}$ be a data matrix to be analyzed by PCA, and suppose that
\begin{equation}
    Y = X + \mathcal{E}, \label{eq:Y model def}
\end{equation}
where $X$ is a signal matrix with $\operatorname{rank}\{X\} = r$, and $\mathcal{E}$ is a noise matrix with $\mathbb{E}[\mathcal{E}_{i,j}] = 0$ for all $i\in[m]$ and $j\in[n]$. 
For simplicity of presentation we also assume that $m\leq n$, noting that one can always replace $Y$ with $Y^T$ otherwise.
Given the model above, we consider the task of estimating the rank $r$ from the matrix of observations $Y$.

We mention that the literature on rank selection for PCA is vast -- spanning several decades of research across multiple disciplines; see for instance~\cite{cattell1966scree,wold1978cross,hoff2007model,owen2009bi,gavish2014optimal,chatterjee2015matrix,fan2014principal,johnstone2018pca,fan2020estimating,hong2020selecting,choi2017selecting,kritchman2008determining,johnstone2017roy} and references therein. In what follows, we only discuss lines of work that are relevant to our setting and approach.

\subsection{Homoskedastic noise} \label{sec:intorduction - homoskedastic noise}
It is well known that if the noise variables \sloppy $\{\mathcal{E}_{i,j}\}_{i\in[m],\;j\in[n]}$ are i.i.d with variance $\sigma^2$, namely, the noise is $\textit{homoskedastic}$, then in the asymptotic regime of $m,n\rightarrow \infty$ and $m/n\rightarrow \gamma \in (0,1]$, the spectrum of the noise matrix $\mathcal{E}$ is described by the well-known Marchenko-Pastur law~\cite{marvcenko1967distribution}. More precisely, letting $\Sigma = \frac{1}{n} \mathcal{E} \mathcal{E}^T$, we consider the \textit{Empirical Spectral Distribution} (ESD) of the eigenvalues of $\Sigma$, defined by 
\begin{equation}
    F_{\Sigma} (\tau)= \frac{1}{m} \sum_{i=1}^m \mathbbm{1}\left(\lambda_i \{ \Sigma \} \leq \tau \right), \label{eq:empirical spectral distribution}
\end{equation}
where $\mathbbm{1}(\cdot)$ is the indicator function, and $\lambda_i \{ \Sigma \}$ is the $i$'th largest eigenvalue of $\Sigma$. Then, as $m,n\rightarrow \infty$ and $m/n\rightarrow \gamma \in (0,1]$, the empirical spectral distribution $F_{\Sigma}(\tau)$ converges almost surely to the Marchenko-Pastur (MP) distribution $F_{\gamma,\sigma}(\tau)$~\cite{marvcenko1967distribution}, which is the cumulative distribution function of the MP density
\begin{equation}
    dF_{\gamma,\sigma}(\tau) = \frac{\sqrt{(\beta_+ - \tau)(\tau - \beta_-)}}{2\pi \sigma^2 \gamma \tau} \mathbbm{1}\left( \beta_- \leq \tau \leq \beta_+\right), \label{eq:MP density}
\end{equation}
where $\beta_{\pm} = \sigma^2 (1\pm \sqrt{\gamma})^2$. 
A particular quantity of interest is the upper edge of the support of the MP density $\beta_+$ (also known as the upper bulk edge), primarily because the density of the eigenvalues beyond that point is $0$.  
In fact, for many standard distributions for the noise entries $\mathcal{E}_{i,j}$~\cite{geman1980limit,yin1988limit}, the spectrum of $\Sigma$ satisfies the stronger property
\begin{equation}
    \lambda_1\{\Sigma\} \overset{p}{\longrightarrow} \beta_+, \label{eq:largest eigenvalue of noise limit in homoskedastic noise}
\end{equation}
where $\overset{p}{\longrightarrow}$ stands for convergence in probability (in the asymptotic regime $m,n\rightarrow \infty$, $m/n\rightarrow \gamma \in (0,1]$).

We note that the spectral properties of $\Sigma$ mentioned above are not restricted to the case of i.i.d noise variables~\cite{bai2010spectral}, and also hold for noise matrices $\mathcal{E}$ whose rows or columns were sampled independently from a random vector with mean zero and covariance $\sigma^2 I$ (where $I$ is the identity matrix). 

In the case of homoskedastic noise, 
a natural approach for estimating the rank $r$ is by the number of eigenvalues of $n^{-1} Y Y^T$ that exceed the threshold $\sigma^2 (1+\sqrt{m/n})^2$, or equivalently, the number of singular values of $Y$ that exceed $\sigma(\sqrt{m} + \sqrt{n})$. 
This approach has been extensively studied in the context of the spiked-model~\cite{baik2005phase,baik2006eigenvalues,benaych2011eigenvalues,nadler2008finite,paul2007asymptotics}, which states that in the asymptotic regime of $m,n\rightarrow \infty$, $m/n\rightarrow \gamma \in (0,1]$, and if the rank $r$ and the singular values of $X$ are fixed, each eigenvalue of $n^{-1} Y Y^T$ that is greater than $\beta_+$ corresponds to a nonzero eigenvalue of $n^{-1} X X^T$ (through a deterministic mapping), and the respective eigenvectors admit a nonzero correlation.

\subsection{Heteroskedastic noise and count data}
In many real-world applications the noise is not homoskedastic, and the noise variances can change arbitrarily across rows and columns. One notable example is count data, common in domains such as network traffic analysis~\cite{shen2005analysis}, photon imaging~\cite{salmon2014poisson}, document topic modeling~\cite{wallach2006topic}, Single-Cell RNA Sequencing (scRNA-seq)~\cite{hafemeister2019normalization}, and High-Throughput Chromosome Conformation Capture (Hi-C)~\cite{johanson2018genome}, among many others (typically in the biological sciences). Specifically, let us consider a prototypical model where $\{Y_{i,j}\}_{i\in[m],\;j\in[n]}$ are independent with
\begin{equation}
    Y_{i,j} {\sim}\operatorname{Poisson}(X_{i,j}), \label{eq:Y poisson model}
\end{equation}
and $\{X_{i,j}\}$ are Poisson parameters (rates) satisfying $X_{i,j}>0$.
Since $\mathbb{E}[Y_{i,j}] = X_{i,j}$, the noise entries $\mathcal{E}_{i,j}$ from~\eqref{eq:Y model def} are centered Poisson with variances $\mathbb{E}[\mathcal{E}_{i,j}^2] = X_{i,j}$. 

Clearly, in the Poisson model~\eqref{eq:Y poisson model} the noise variances $X_{i,j}$ can differ substantially, making the noise heteroskedastic in the most general sense. 
In this case, the MP law is not expected to hold, and the spectral distribution of the noise is determined by nonlinear equations known as the Dyson equations (see eq. (1.3) in~\cite{alt2017local}). These equations depend on the unknown variance profile of the noise. Hence, the limiting spectral distribution of the noise is nontrivial and is unavailable in advance, posing a major challenge for rank estimation. Naturally, the same challenge arises for other distributions of $Y_{i,j}$ that admit a relation between the mean and the variance, such as for the binomial, negative binomial, gamma, and others (not necessarily count random variables).

A large portion of existing literature on PCA in heteroskedastic noise is dedicated to the task of estimating $X$ or the subspaces spanning its rows and columns~\cite{hong2021heppcat,zhang2018heteroskedastic,hong2018optimally,leeb2018optimal,leeb2019matrix,liu2018pca,cai2019subspace,donoho2020screenot}. 
However, in this line of work it is typically assumed that at least some information on the rank of $X$ or on the noise variance profile is available (such as heteroskedasticity only across rows or across columns). 
Several recent works also considered the task of estimating $X$ in the particular setting of count data~\cite{bigot2017generalized,robin2019low,mcrae2019low,cao2015poisson,bigot2020low}, typically by solving a regularized optimization problem utilizing a low-rank model for $X$.  
Most recently, rank selection under heteroskedastic noise was considered in~\cite{hong2020selecting,leeb2019rapid,ke2020estimation}.
In particular,~\cite{leeb2019rapid} provided an algorithm for computing the upper edge of the noise's spectral distribution assuming that the noise variance profile is of rank one,~\cite{hong2020selecting} proposed a variant of parallel analysis that preserves the variance profile of the noise by random signflips, and~\cite{ke2020estimation} described a rank estimation procedure assuming a prior gamma distribution on the noise variances.

\subsection{Our approach and contributions} \label{sec:introduction - contributions}
In this work, we propose a new approach termed \textit{biwhitening} for estimating the rank $r$ in the Poisson model~\eqref{eq:Y poisson model} (and other models) without requiring any prior knowledge on $X$ or its spectrum.  Our main idea is to guarantee that the standard MP law holds by applying appropriate diagonal scaling to $Y$, namely, multiplying its rows and columns by judiciously chosen scaling factors. These scaling factors can be estimated directly from the observation matrix $Y$, and their purpose is to make the average noise variance in each row and each column of the scaled noise matrix precisely $1$. Then, we estimate $r$ from the spectrum of the scaled version of $Y$ as if the noise was homoskedastic with variance $1$, that is, via the upper edge $\beta_+$ of the MP density $dF_{\gamma,1}$, taking $\gamma = m/n$. We derive our approach and justify it theoretically in the standard Poisson model~\eqref{eq:Y poisson model}, and further extend it to (almost) any distribution that satisfies a quadratic relation between the mean and the variance, i.e., $\operatorname{Var}[Y_{i,j}] = a + b X_{i,j} + c X_{i,j}^2$, where $\operatorname{Var}[Y_{i,j}]$ is the variance of $Y_{i,j}$ (the Poisson is a special case with $a=c=0$, $b=1$). 

Our proposed biwhitening procedure for the Poisson model is described in Algorithm~\ref{alg:noise standardization}, where $\mathbf{x}$ and $\mathbf{y}$ are vectors of length $m$ and $n$, respectively, $D(\mathbf{x})$ is a diagonal matrix with $\mathbf{x}$ on its main diagonal, and $\mathbf{1}_m$ is a vector of $m$ ones.

\begin{algorithm}
\caption{Biwhitening and rank estimation for Poisson data}\label{alg:noise standardization}
\begin{algorithmic}[1]
\Statex{\textbf{Input:} Nonnegative $m\times n$ Poisson count matrix $Y$}.
\State \label{alg: step 1}Find positive vectors $\mathbf{x}$ and $\mathbf{y}$ so that $D(\mathbf{x}) Y D(\mathbf{y})$ has row sums $n \cdot \mathbf{1}_{m}$ and column sums $m \cdot \mathbf{1}_{n}$ by, e.g., the Sinkhorn-Knopp algorithm (see Algorithm~\ref{alg:SK} in Section~\ref{sec:Standardization of Poisson noise by diagonal scaling}).  
\State Form the biwhitened matrix $\hat{Y} = \sqrt{D(\mathbf{x})} {Y} \sqrt{D(\mathbf{y})}$. \label{alg:step 2}
\State Estimate the rank $r$ by the number of singular values of $\hat{Y}$ that exceed $\sqrt{n} + \sqrt{m}$. \label{alg:step 3}
\end{algorithmic}
\end{algorithm}

Figure~\ref{fig:introduction example} exemplifies the advantage of biwhitening on a simulated Poisson count matrix $Y$ with $m=300$, $n=1000$, and $r=10$. More details for reproducibility can be found in Appendix~\ref{reproducibility details introduction example}. 
Notably, it is difficult to visually determine the true rank of $X$ from the spectrum of $Y$, as the sorted eigenvalues of $n^{-1} Y Y^T$ admit a gradual decay and do not exhibit any significant gaps. Also, the histogram of these eigenvalues does not agree with the MP law (nor do we expect it to), and the MP upper edge $\beta_+$ does not provide an accurate threshold for determining the rank of $X$. On the other hand, once we apply biwhitening to $Y$, the rank $r=10$ is ``revealed'' in the sense that the first $r$ eigenvalues of $n^{-1} \hat{Y} \hat{Y}^T$ (where $\hat{Y}$ is from Algorithm~\ref{alg:noise standardization}) are significantly separated from the rest of the spectrum -- whose density agrees with the MP law, and whose upper edge is captured precisely by $\beta_+ = (1+\sqrt{m/n})^2$. 

\begin{figure} 
  \centering
  	{
  	\subfloat[][Sorted eigenvalues, original data]
  	{
    \includegraphics[width=0.4\textwidth]{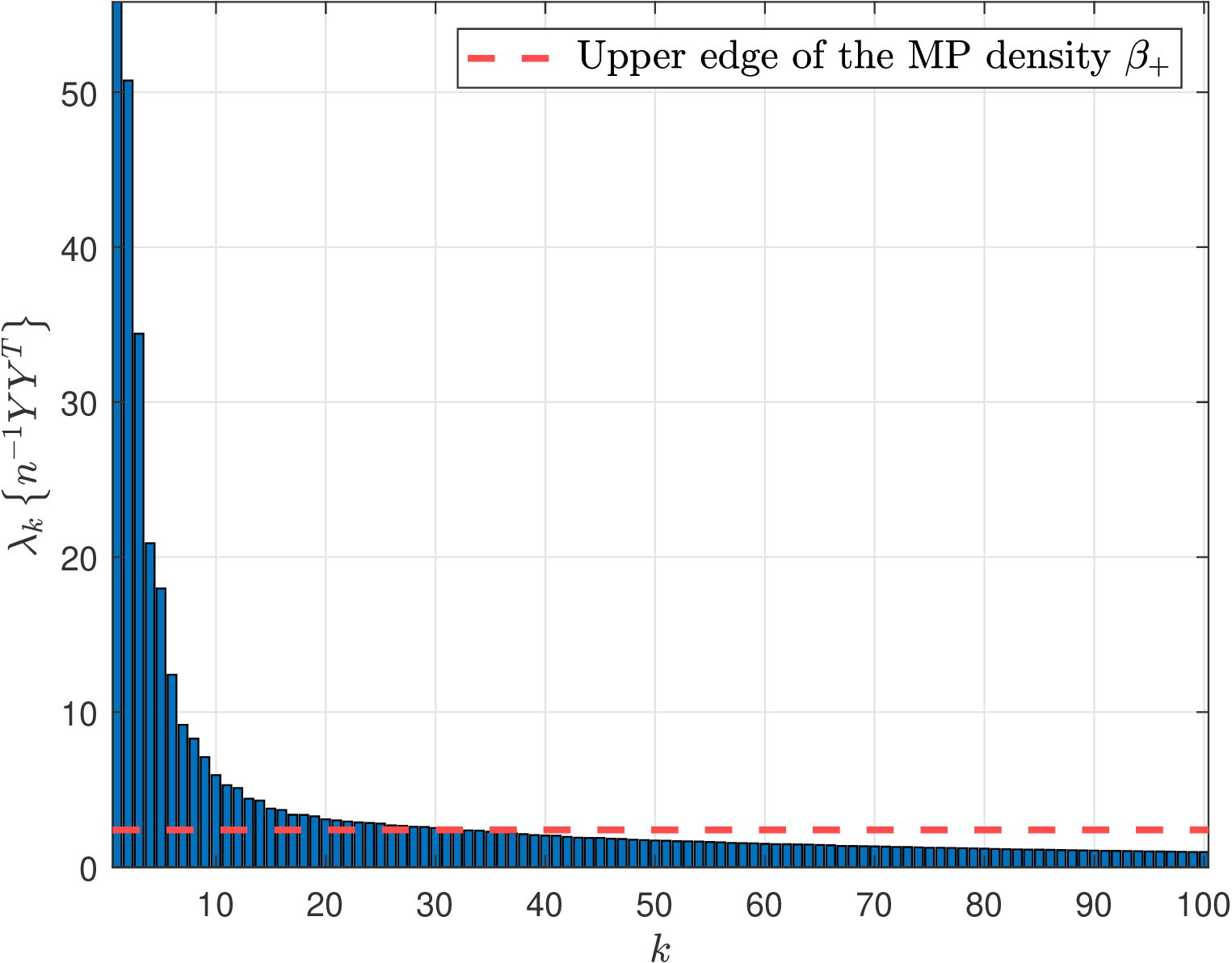} 
    }
    \subfloat[][Sorted eigenvalues, {after} biwhitening] 
  	{
    \includegraphics[width=0.4\textwidth]{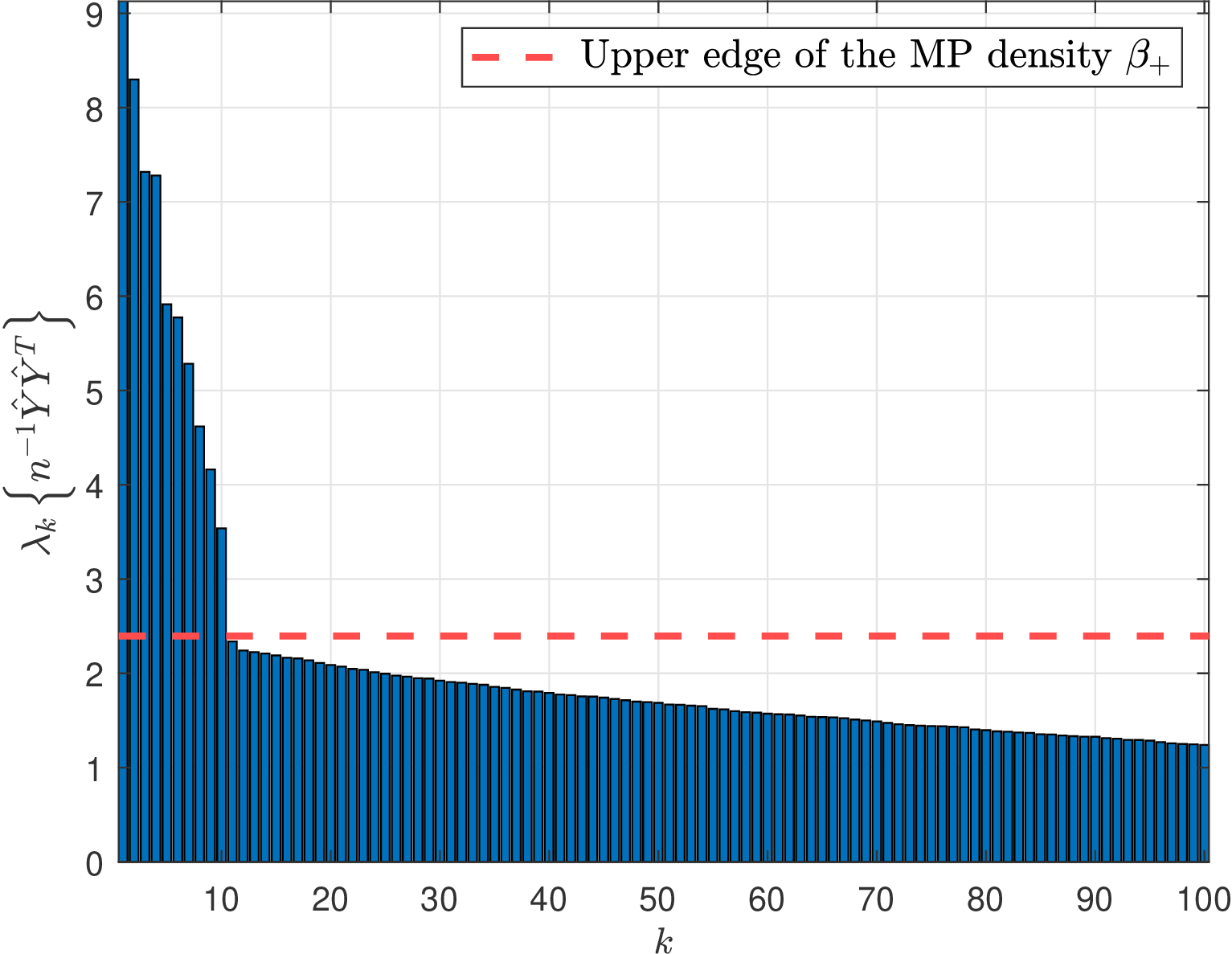} 
    }
    \\
    \subfloat[][Eigenvalue density, original data]  
  	{
    \includegraphics[width=0.4\textwidth]{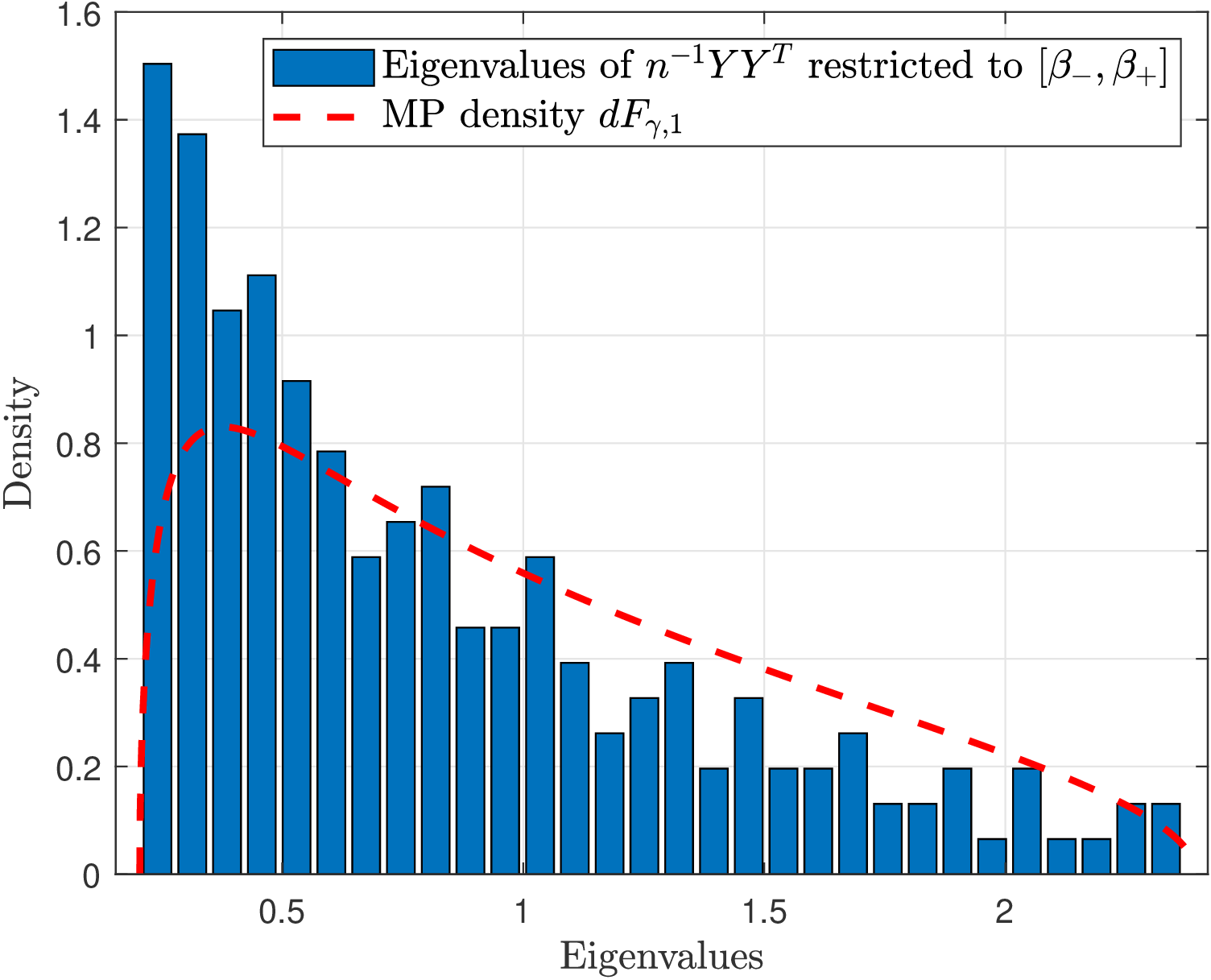}  
    }
    \subfloat[][Eigenvalue density, after biwhitening] 
  	{
    \includegraphics[width=0.4\textwidth]{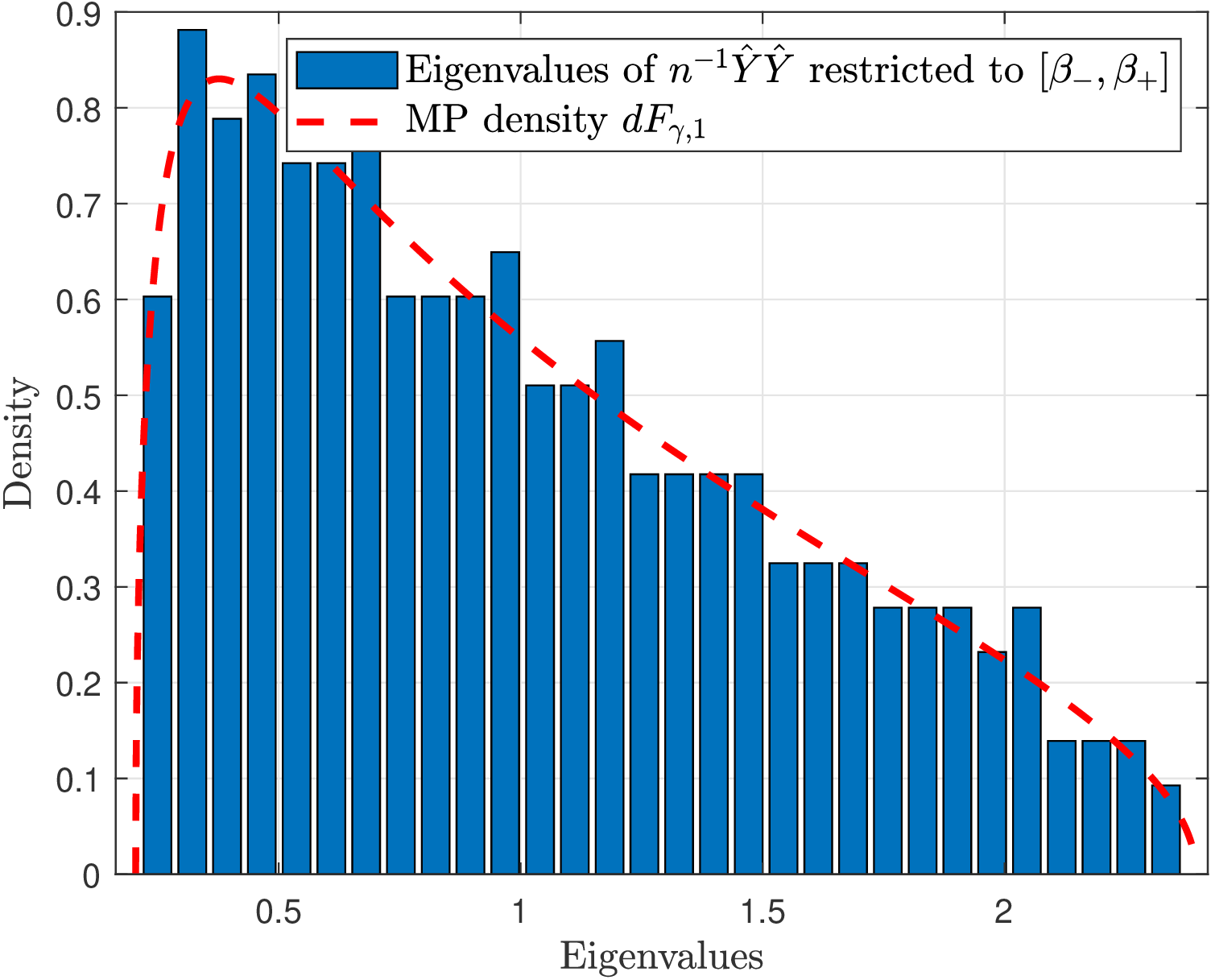} 
    } 

    }
    \caption
    {The spectrum of the original observation matrix and after applying biwhitening, for $m=300$, $n=1000$, and $r=10$. The top two panels depict the sorted eigenvalues of $n^{-1} Y Y^T$ (top left) and $n^{-1} \hat{Y} \hat{Y}^T$ (top right) versus the upper edge $\beta_+$ of the MP density  (dashed red line), where $\hat{Y}$ is the biwhitened matrix from Algorithm~\ref{alg:noise standardization}. The bottom two panels depict the empirical density (i.e., normalized histogram) of the eigenvalues of $n^{-1} Y Y^T$ (bottom left) and $n^{-1} \hat{Y} \hat{Y}^T$ (bottom right) restricted to the support of the MP density, i.e., $[\beta_-,\beta_+]$, versus the MP density $dF_{\gamma,1}$ for $\gamma = m/n$ (dashed red line). } \label{fig:introduction example}
    \end{figure} 
    
The organization of this paper is as follows. In Section~\ref{sec:method and results} we derive and analyze our approach in the Poisson model~\eqref{eq:Y poisson model}. In Section~\ref{sec:Poisson simulations} we conduct numerical experiments that validate our theoretical results and demonstrate that our approach is robust to challenging regimes such as strong heteroskedasticity. In Section~\ref{sec:beyond Poisson noise} we extend our approach to general families of distributions with quadratic variance functions and show that it can also account for missing entries in the data. In Section~\ref{sec:practical considerations and real data} we address important practical considerations such as adapting our approach to unknown data types, and exemplify our approach on several real datasets.

\section{Method derivation and main results} \label{sec:method and results}
\subsection{Standardization of Poisson noise by diagonal scaling} \label{sec:Standardization of Poisson noise by diagonal scaling}
The main idea underlying our approach is to appropriately scale the rows and columns of the Poisson matrix $Y$ from~\eqref{eq:Y poisson model}, such that the noise component in the resulting scaled matrix satisfies the Marchenko-Pastur (MP) law~\eqref{eq:MP density}. Towards that end, let $\mathbf{u} = [u_1,\ldots,u_m]$ and $\mathbf{v} = [v_1,\ldots,v_n]$ be positive vectors, and define
\begin{equation}
    \widetilde{Y} = D(\mathbf{u}) Y D(\mathbf{v}) = \widetilde{X} + \widetilde{\mathcal{E}}, \label{eq:Y_tilde def}
\end{equation}
where $Y$ is from~\eqref{eq:Y poisson model}, $D(\mathbf{u})$ and $D(\mathbf{v})$ are diagonal matrices with $\mathbf{u}$ and $\mathbf{v}$ on their main diagonals, respectively, and
\begin{equation}
    \widetilde{X} = D(\mathbf{u}) X D(\mathbf{v}), \qquad\qquad \widetilde{\mathcal{E}} = D(\mathbf{u}) \mathcal{E} D(\mathbf{v}).
\end{equation}
Notably, $\widetilde{\mathcal{E}}$ and $\widetilde{X}$ preserve several important properties of the noise matrix $\mathcal{E}$ and the signal matrix $X$. In particular, the scaled noise random variables $\{\widetilde{\mathcal{E}}_{i,j}\} = \{u_i \mathcal{E}_{i,j} v_j\}$ are independent with zero means, and 
\begin{equation}
    \operatorname{Rank}\{\widetilde{X}\} = \operatorname{Rank}\{ D(\mathbf{u}) X D(\mathbf{v}) \} = \operatorname{Rank}\{{X}\} = r, \label{eq:rank(Lambda_tilde) = rank(Lambda)}
\end{equation}
since diagonal scaling (with nonzero scaling factors) preserves the row and column spaces of a matrix.
Hence, we can translate the task of estimating the rank of $X$ to the analogous task of estimating the rank of $\widetilde{X}$. 
Crucially, the diagonal scaling in~\eqref{eq:Y_tilde def} allows us to control certain aspects of the variance profile of the scaled noise matrix $\widetilde{\mathcal{E}}$. While we cannot use this diagonal scaling to make all the (entrywise) variances $\mathbb{E}[\widetilde{\mathcal{E}}_{i,j}^2]$ equal (unless $X$ is of rank $1$), we use $\mathbf{u}$ and $\mathbf{v}$ that enforce the average variance in each row and each column of $\widetilde{\mathcal{E}}$ to be $1$. Specifically, we consider $\mathbf{u}$ and $\mathbf{v}$ that satisfy
\begin{equation}
    1 = \frac{1}{m} \sum_{i=1}^m \mathbb{E}[\widetilde{\mathcal{E}}_{i,j}^2] = \frac{1}{m} \sum_{i=1}^m u_i^2 X_{i,j} v_j^2, \qquad \text{and} \qquad 1 = \frac{1}{n} \sum_{j=1}^n \mathbb{E}[\widetilde{\mathcal{E}}_{i,j}^2] = \frac{1}{n} \sum_{i=1}^n u_i^2 X_{i,j} v_j^2, \label{eq:scaling equations from true variances}
\end{equation}
for all $i\in [m]$ and $j\in [n]$, using the fact that $\mathbb{E}[\mathcal{E}_{i,j}^2] = X_{i,j}$ in the Poisson model~\eqref{eq:Y model def}. 
It is worthwhile to point out that since $\{\widetilde{\mathcal{E}}_{i,j}\}$ are independent and have zero means, the equations in~\eqref{eq:scaling equations from true variances} are equivalent to
\begin{equation} \label{eq: biwhitening}
    \mathbb{E}[\frac{1}{n} \widetilde{\mathcal{E}} \widetilde{\mathcal{E}}^T] = I_m, \qquad \text{and} \qquad \mathbb{E}[\frac{1}{m} \widetilde{\mathcal{E}}^T \widetilde{\mathcal{E}}] = I_n, 
\end{equation}
where $I_m$ and $I_n$ are the $m\times m$ and $n \times n$ identity matrices, respectively. Observe that~\eqref{eq: biwhitening} is satisfied by typical homoskedastic noise models. In particular,~\eqref{eq: biwhitening} holds if we replace $\widetilde{\mathcal{E}}$ 
by any matrix whose entries are independent with mean zero and variance one, and more generally, if we replace $\widetilde{\mathcal{E}}$ 
by a matrix whose either rows or columns are sampled independently from an isotropic random vector (i.e., a random vector whose mean is zero and covariance is the identity matrix). Indeed, Equation~\eqref{eq: biwhitening} is the motivation for the name \textit{biwhitening}. 

Since $\mathbf{u}$ and $\mathbf{v}$ from~\eqref{eq:scaling equations from true variances} are the solution to a nonlinear system of equations, it may not be immediately obvious that such a solution exists, whether or not it is unique, and how to find it. These questions can be settled by observing that~\eqref{eq:scaling equations from true variances} is in fact an instance of a problem known as \textit{matrix scaling} (also \textit{matrix balancing} or \textit{bi-proportional scaling}); see~\cite{idel2016review} for an extensive review of the topic. Specifically, given a nonnegative matrix $A\in\mathbb{R}^{m\times n}$ and positive vectors $\mathbf{r} = [r_1,\ldots,r_m]$ and $\mathbf{c} = [c_1,\ldots,c_n]$, the goal in matrix scaling is to find positive vectors $\mathbf{x} = [x_1,\ldots,x_m]$ and $\mathbf{y} = [y_1,\ldots,y_n]$ such that the matrix $D(\mathbf{x}) A D(\mathbf{y})$ has prescribed row sums $\mathbf{r}$ and column sums $\mathbf{c}$, i.e.,
\begin{equation}
    c_j = \sum_{i=1}^m x_i A_{i,j} y_j, \qquad \text{and} \qquad r_i = \sum_{j=1}^n x_i A_{i,j} y_j, \label{eq:matrix scaling general equations}
\end{equation}
for all $i\in [m]$ and $j\in [n]$. We will refer to positive $\mathbf{x}$ and $\mathbf{y}$ that solve~\eqref{eq:matrix scaling general equations} as \textit{scaling factors} of $A$, and say that $\mathbf{x}$ and $\mathbf{y}$ scale $A$ to row sums $\mathbf{r}$ and column sums $\mathbf{c}$. Clearly, the equations in~\eqref{eq:scaling equations from true variances} are equivalent to those in~\eqref{eq:matrix scaling general equations} if we take $A = X$, $r_i = n$, $c_j = m$, $x_i = u_i^2$, and $y_j = v_j^2$.  

When the scaling factors of $A$ exist, they can be found by the Sinkhorn-Knopp algorithm~\cite{sinkhorn1967concerning,pukelsheim2014biproportional}; see Algorithm~\ref{alg:SK}. 
For the convergence rate of the Sinkhorn-Knopp algorithm, see~\cite{knight2008sinkhorn,altschuler2017near,chakrabarty2020better} and references therein.
As for the existence and uniqueness of the scaling factors, the subject has been extensively studied~\cite{sinkhorn1967diagonal,sinkhorn1974diagonal, brualdi1966diagonal,brualdi1974dad,csima1972dad}. 
In particular, 
we have the following proposition for the system of equations in~\eqref{eq:scaling equations from true variances}.
\begin{prop}[Existence and uniqueness of $\mathbf{u}$ and $\mathbf{v}$] \label{prop:existcne and uniquness for Lambda}
There exists a pair $(\mathbf{u},\mathbf{v})$ of positive vectors that satisfies~\eqref{eq:scaling equations from true variances} for all $i\in [m]$ and $j\in [n]$. Furthermore, $(\mathbf{u},\mathbf{v})$ is unique up to a positive scalar, namely it can only be replaced with $(a\mathbf{u},a^{-1}\mathbf{v})$ for any $a>0$.
\end{prop}
The proof follows immediately from Theorem 1 in~\cite{sinkhorn1967diagonal} when taking $A = X$, $r_i = n$, $c_j = m$, $x_i = u_i^2$, and $y_j = v_j^2$, using the fact that $X$ is strictly positive and $\sum_{i=1}^m r_i = m n = \sum_{j=1}^n c_j$.

\begin{algorithm}
\caption{The Sinkhorn-Knopp algorithm}\label{alg:SK}
\begin{algorithmic}[1]
\Statex{\textbf{Input:} Nonnegative $m\times n$ matrix $A$, prescribed row sums $\mathbf{r}$ and column sums $\mathbf{c}$, tolerance $\delta>0$}.
\State Initialize: $\mathbf{x}^{(0)} = \mathbf{1}_m$, $\mathbf{y}^{(0)} = \mathbf{1}_n$, $\tau = 0$.
\State While $\max_{i \in [m]}\vert \sum_{j=1}^n x_{i}^{(\tau)} A_{i,j} y_j^{(\tau)}  - r_i \vert > \delta$ or $\max_{ j \in [n]}\vert \sum_{i=1}^m x_{i}^{(\tau)} A_{i,j} y_j^{(\tau)} - c_j \vert > \delta$, do:
\begin{itemize}
    \item $\mathbf{y}_{j}^{{(\tau+1)}} = c_j/({\sum_{i=1}^m A_{i,j} x_{i}^{(\tau)}})$, for $j=1,\ldots,n$.
    \item $\mathbf{x}_{i}^{{(\tau+1)}} = r_i/({\sum_{j=1}^n A_{i,j} y_{j}^{(\tau+1)}})$, for $i=1,\ldots,m$.
    \item Update $\tau \leftarrow \tau + 1$.
\end{itemize}
\State Return $\mathbf{x}^{(\tau)}$ and $\mathbf{y}^{(\tau)}$.
\end{algorithmic}
\end{algorithm}

According to Proposition~\ref{prop:existcne and uniquness for Lambda}, the products $\{u_i v_j\}_{i\in[m], \; j\in[n]}$ are determined uniquely by $X$, which means that the scaled noise matrix $\widetilde{\mathcal{E}} = (\mathcal{E}_{i,j} u_i v_j)_{i\in[m], \; j\in[n]}$ is a random matrix also uniquely determined by $X$. It is then of interest to characterize the spectral properties of $\widetilde{\mathcal{E}}$, particularly in the regime of large $m$ and $n$, noting that the singular values of $\widetilde{\mathcal{E}}$ can be directly obtained from the eigenvalues of the matrix $\widetilde{\Sigma} = n^{-1}\widetilde{\mathcal{E}} (\widetilde{\mathcal{E}})^T $.
To characterize the spectral behavior of $\widetilde{\mathcal{E}}$, let $\{m_n\}_{n\geq 1}$ be a sequence of positive integers such that $m_n\underset{n\rightarrow\infty}{\longrightarrow}\infty$ and $m_n/n \underset{n\rightarrow\infty}{\longrightarrow} \gamma \in (0,1]$. In addition, let $\{X^{(n)}\}_{n\geq 1}$ and $\{\widetilde{\mathcal{E}}^{(n)}\}_{n\geq 1}$ be sequence of $m_n\times n$ matrices defined equivalently to $X$ and $\widetilde{\mathcal{E}}$, respectively, and define $\widetilde{\Sigma}^{(n)} = n^{-1}\widetilde{\mathcal{E}}^{(n)} (\widetilde{\mathcal{E}}^{(n)})^T $. We then have the following result.
\begin{thm}[Marchenko-Pastur law and the limit of the largest eigenvalue of $\widetilde{\Sigma}^{(n)}$]\label{thm:Marchenko-Pastur for biwhitened noise}
Suppose that there exist universal constants $C,c>0$ such that $ c \leq \max_{i,j} X^{(n)}_{i,j} \leq C \min_{i,j} X^{(n)}_{i,j}$ for all $i\in[m]$, $j\in[n]$, and $n\geq 1$. Then, as $n\rightarrow \infty$, the empirical spectral distribution of $\widetilde{\Sigma}^{(n)}$, given by $F_{\widetilde{\Sigma}^{(n)}}$ (see~\eqref{eq:empirical spectral distribution}), converges almost surely to the Marchenko-Pastur distribution with parameter $\gamma$ and noise variance $\sigma^2=1$, i.e., $F_{\gamma,1}$. Furthermore, we have that $\lambda_1\{\widetilde{\Sigma}^{(n)} \} \overset{p}{\longrightarrow} \beta_+ = (1 + \sqrt{\gamma})^2$.
\end{thm}

Essentially, under the conditions in Theorem~\ref{thm:Marchenko-Pastur for biwhitened noise} and in the asymptotic regime of $m,n\rightarrow\infty$, $m/n\rightarrow\gamma\in (0,1]$, the spectrum of the noise $\widetilde{\mathcal{E}}$ behaves as if the noise was homoskedastic, namely that the MP law holds and the largest eigenvalue of $\widetilde{\Sigma}$ converges to the upper edge of the MP bulk $\beta_+$ (see Section~\ref{sec:intorduction - homoskedastic noise}).
The condition $ c < \max_{i,j} X^{(n)}_{i,j} \leq C \min_{i,j} X^{(n)}_{i,j}$ in Theorem~\ref{thm:Marchenko-Pastur for biwhitened noise} requires that the Poisson parameters $X_{i,j}^{(n)}$ are always bounded away from zero, and that all of them admit the same growth rate with $n$ (i.e., none of the Poisson parameters can grow unbounded with $n$ relative to others). Even with this restriction, the ratios $X_{i,j}^{(n)} / X_{k,\ell}^{(n)}$ for $(i,j)\neq (k,\ell)$ can vary, and can be very large or very small for certain pairs $(i,j)$ and $(k,\ell)$, allowing for substantial heteroskedasticity of the noise in the model~\eqref{eq:Y model def}.
The proof of Theorem~\ref{thm:Marchenko-Pastur for biwhitened noise} can be found in Appendix~\ref{appendix:proof of theorem on MP law using true variances}, and relies on the results of~\cite{girko2001theory} and~\cite{alt2017local} with certain boundedness properties of the scaling factors $\mathbf{u}$ and $\mathbf{v}$.

Next, using the fact that $\operatorname{rank}\{\widetilde{X}\} = \operatorname{rank}\{{X}\}$ and $\lambda_1 \{\widetilde{\Sigma}^{(n)} \} \overset{p}{\longrightarrow} (1 + \sqrt{\gamma})^2$, it is natural to consider the following estimator for the rank $r$ of $X$:
\begin{equation}
    \widetilde{r}_\varepsilon = \max \left\{k: \; \lambda_k \{n^{-1} \widetilde{Y} \widetilde{Y}^T\} > \left(1+\sqrt{\frac{m}{n}}\right)^2 + \varepsilon \right\}, \label{eq:r_tilde def} 
\end{equation}
where $\varepsilon>0$ and $\lambda_k \{n^{-1} \widetilde{Y} \widetilde{Y}^T\}$ is the $k$'th largest eigenvalue of $n^{-1} \widetilde{Y} \widetilde{Y}^T$. In words, we take $\widetilde{r}_\varepsilon$ to be the number of eigenvalues of $n^{-1} \widetilde{Y} \widetilde{Y}^T$ that are not included in the $\varepsilon$-neighborhood of the Marchenko-Pastur bulk. 
Let us consider again the asymptotic setting of $m_n\rightarrow\infty$, $m_n/n \rightarrow \gamma \in (0,1]$, letting $\operatorname{rank}\{X^{(n)}\} = {r}^{(n)} < m_n$, and $\widetilde{r}^{(n)}_\varepsilon$ be as $\widetilde{r}_\varepsilon$ from~\eqref{eq:r_tilde def} when replacing $\widetilde{Y}$ with $\widetilde{Y}^{(n)}$. We then have the following property of the rank estimator $\widetilde{r}_\varepsilon^{(n)}$ in the asymptotic regime of $n\rightarrow\infty$.
\begin{thm} \label{thm:r_tilde no false detect}
Under the conditions in Theorem~\ref{thm:Marchenko-Pastur for biwhitened noise}, $\operatorname{Pr} \{ r^{(n)} < \widetilde{r}_\varepsilon^{(n)}\} \underset{n\rightarrow\infty}{\longrightarrow} 0$ for any $\varepsilon>0$.
\end{thm}
The proof can be found in Appendix~\ref{appendix:proof of theorem on rank estimator}, and relies on Theorem~\ref{thm:Marchenko-Pastur for biwhitened noise} and the fact that diagonal scaling preserves the rank of $X$ (see~\eqref{eq:rank(Lambda_tilde) = rank(Lambda)}). Fundamentally, Theorem~\ref{thm:r_tilde no false detect} states that in the asymptotic setting of $m,n\rightarrow\infty$, $m/n\rightarrow\gamma\in (0,1]$, the rank estimator $\widetilde{r}_\varepsilon$ does not overestimate the rank of $X$ for any $\varepsilon>0$ (with probability approaching $1$). Taking $\varepsilon\rightarrow 0$, this result implies that for sufficiently large $m$ and $n$, all eigenvalues of $n^{-1} \widetilde{Y} \widetilde{Y}^T$ that exceed $(1+\sqrt{m/n})^2$ by an arbitrarily-small positive value correspond to true signal components, i.e., to nonzero singular values of $X$. The reason that~\eqref{eq:r_tilde def} asymptotically underestimates the rank is that some of the signal components can be too weak for detection. In particular, the existence of small eigenvalues of $n^{-1} \widetilde{X} \widetilde{X}^T$ can be masked by the bulk of the eigenvalues of the noise $n^{-1} \widetilde{\mathcal{E}} \widetilde{\mathcal{E}}^T$ (given by the MP distribution); see~\cite{baik2005phase,baik2006eigenvalues,benaych2011eigenvalues,nadler2008finite,paul2007asymptotics} for more details on this phenomenon in the case of homoskedastic noise and the spiked-model (noting that our model does not require the rank of $X^{(n)}$ to be fixed).

\subsection{Estimating the scaling factors $\mathbf{u}$ and $\mathbf{v}$} \label{sec:estimating the scaling factors}
The immediate obstacle in employing $\mathbf{u}$ and $\mathbf{v}$ that satisfy~\eqref{eq:scaling equations from true variances} is that $\{X_{i,j}\}_{i\in[m],\;j\in[n]}$ are unknown. Nonetheless, while we do not have access to $X_{i,j}$, we do have access to an unbiased estimator of $X_{i,j}$, which is $Y_{i,j}$. Therefore, instead of solving~\eqref{eq:scaling equations from true variances}, we propose to find positive vectors $\hat{\mathbf{u}} = [\hat{u}_1,\ldots,\hat{u}_m]$ and $\hat{\mathbf{v}} = [\hat{v}_1,\ldots,\hat{v}_n]$ that solve the surrogate system of equations obtained by replacing $X_{i,j}$ in~\eqref{eq:scaling equations from true variances} with $Y_{i,j}$, i.e.,
\begin{equation}
    1 = \frac{1}{m} \sum_{i=1}^m \hat{u}_i^2 Y_{i,j} \hat{v}_j^2, \qquad \text{and} \qquad 1 = \frac{1}{n} \sum_{i=1}^n \hat{u}_i^2 Y_{i,j} \hat{v}_j^2, \label{eq:scaling equations from estimated variances}
\end{equation}
for all $i\in [m]$ and $j\in[n]$. 
Analogously to~\eqref{eq:Y_tilde def}, we then define
\begin{equation}
    \hat{Y} = D(\hat{\mathbf{u}}) Y D(\hat{\mathbf{v}}) = \hat{X} + \hat{\mathcal{E}}, \label{eq:Y_hat def}
\end{equation}
where $\hat{X} = D(\hat{\mathbf{u}}) X D(\hat{\mathbf{v}})$, $\hat{\mathcal{E}} = D(\hat{\mathbf{u}}) \mathcal{E} D(\hat{\mathbf{v}})$,
and equivalently to~\eqref{eq:rank(Lambda_tilde) = rank(Lambda)} we have $\operatorname{Rank}\{\hat{X}\} = \operatorname{Rank}\{{X}\} = r$.
Similarly to~\eqref{eq:scaling equations from true variances}, the system of equations in~\eqref{eq:scaling equations from estimated variances} is an instance of~\eqref{eq:matrix scaling general equations} if we set $A  = Y$, $r_i = n$, $c_j = m$, $x_i = \hat{u}_i^2$, and $y_j = \hat{v}_j^2$.

It is important to note that in contrast to the matrix $X$ in~\eqref{eq:scaling equations from true variances}, a realization of the random matrix $Y$ may not be strictly positive and may contain zeros, hence the existence and uniqueness of positive $\hat{\mathbf{u}}$ and $\hat{\mathbf{v}}$ that satisfy~\eqref{eq:scaling equations from estimated variances} is not obvious.
Indeed, in the most general case of matrix scaling, the existence and uniqueness of $\mathbf{x}$ and $\mathbf{y}$ that satisfy~\eqref{eq:matrix scaling general equations} depends on the particular zero pattern of $A$ (i.e., the set of indices $(i,j)$ for which $A_{i,j} = 0$); see the end of this section and Appendix~\ref{appendix: existence and uniqueness} for more details. 

Let us consider again the asymptotic setting where $\{m_n\}_{n\geq 1}$ is a sequence of positive integers such that $m_n\underset{n\rightarrow\infty}{\longrightarrow}\infty$ and $m_n/n \underset{n\rightarrow\infty}{\longrightarrow} \gamma \in (0,1]$, and let $(\mathbf{u}^{(n)},\mathbf{v}^{(n)})$ be the solution to~\eqref{eq:scaling equations from true variances} corresponding to $X^{(n)} \in \mathbb{R}^{m_n \times n}$ (replacing $X$) and satisfying $\Vert \mathbf{u}^{(n)}\Vert_2 = \Vert \mathbf{v}^{(n)} \Vert_2 $ (such a solution exists and is unique according to Proposition~\ref{prop:existcne and uniquness for Lambda}). In addition, let $\{Y^{(n)}\}_{n\geq 1}$ be a sequence of $m_n \times n$ random matrices with independent Poisson entries (analogous to $Y$) such that $\mathbb{E}[Y^{(n)}] = X^{(n)}$, and define $\hat{\mathbf{u}}^{(n)}$ and $\hat{\mathbf{v}}^{(n)}$ analogously to $\hat{\mathbf{u}}$ and $\hat{\mathbf{v}}$, respectively, when replacing $Y$ in~\eqref{eq:scaling equations from estimated variances} with $Y^{(n)}$. We now provide the following result.
\begin{lem}[Convergence of estimated scaling factors] \label{lem:convergence of scaling factors for estimated variances}
Suppose that there exist universal constants $C,c,\epsilon>0$ such that $c (\log n)^{1+\epsilon} \leq \max_{i,j} X_{i,j}^{(n)} \leq C \min_{i,j}X_{i,j}^{(n)}$ for all $i\in[m]$, $j\in [n]$, and $n\geq 1$. Then, there is a constant $\widetilde{C}>0$ such that with probability that tends to $1$ as $n\rightarrow\infty$,  there exists a pair of positive scaling factors $(\hat{u}^{(n)},\hat{v}^{(n)})$ that solves~\eqref{eq:scaling equations from estimated variances} and a scalar $a_n>0$, which satisfy
\begin{equation}
    {\left\vert \frac{a_n \hat{u}^{(n)}_i}{u^{(n)}_i} - 1 \right\vert} \leq  
      \widetilde{C} \sqrt{\frac{\log n}{n} } , \qquad \qquad
     {\left\vert\frac{a_n^{-1} \hat{v}^{(n)}_j}{v^{(n)}_j} - 1 \right\vert} \leq \widetilde{C} \sqrt{\frac{\log n}{n} },
\end{equation}
for all $i\in[m]$ and $j\in[n]$.
\end{lem}
The proof of Lemma~\ref{lem:convergence of scaling factors for estimated variances} can be found in Appendix~\ref{appendix:proof of lemma on convergence of scaling factors}, and relies on the results in~\cite{landa2020scaling}.
Note that the convergence of the estimated scaling factors $(\hat{\mathbf{u}},\hat{\mathbf{v}})$ to the true scaling factors  $({\mathbf{u}},{\mathbf{v}})$ in Lemma~\ref{lem:convergence of scaling factors for estimated variances} is up to an arbitrary sequence of positive scalars $\{a_n\}$, which arise from the fundamental ambiguity in the uniqueness of the scaling factors; see Proposition~\ref{prop:existcne and uniquness for Lambda}. We note that the statement in Lemma~\ref{lem:convergence of scaling factors for estimated variances} is particularly useful for our subsequent analysis and is more informative for our purposes than a statement on the convergence rate of the errors $|\hat{u}_i^{(n)}\hat{v}_j^{(n)} - u_i^{(n)} v_j^{(n)} | $ (which do not involve the arbitrary scalars $\{a_n\}$).

Let us define $\hat{\Sigma}^{(n)} = n^{-1} \hat{\mathcal{E}}^{(n)} (\hat{\mathcal{E}}^{(n)})^T$. Using Lemma~\ref{lem:convergence of scaling factors for estimated variances} we obtain the following result, which establishes the convergence of the spectrum of $\hat{\Sigma}^{(n)}$ to the spectrum of $\widetilde{\Sigma}^{(n)}$ in probability as $n\rightarrow\infty$.
\begin{thm}[Spectral convergence of $\hat{\Sigma}^{(n)}$ to $\widetilde{\Sigma}^{(n)}$ ]\label{thm:Marchenko-Pastur for biwhitened noise from estimated variances}
Suppose that the conditions in Lemma~\ref{lem:convergence of scaling factors for estimated variances} hold. Then, $\max_{i\in [m_n]}\left\vert \lambda_i\{ \hat{\Sigma}^{(n)} \} - \lambda_i\{ \widetilde{\Sigma}^{(n)}  \} \right\vert \overset{p}{\longrightarrow} 0.$
\end{thm}
The proof can be found in Appendix~\ref{appendix:proof of theorem on MP law using estimated variances}. 
We mention that the main difference between the conditions in Theorem~\ref{thm:Marchenko-Pastur for biwhitened noise} and those in Theorem~\ref{thm:Marchenko-Pastur for biwhitened noise from estimated variances} is that we further impose a growth rate -- slightly larger than logarithmic -- on the Poisson parameters $X_{i,j}$, namely the condition $(c/C) (\log n)^{1+\epsilon} \leq \min_{i,j}X_{i,j}^{(n)}$ for some $\epsilon>0$. Notably, this condition guarantees that the observation matrix $Y^{(n)}$ is strictly positive with probability approaching $1$ as $n\rightarrow\infty$ (using the fact that $\operatorname{Pr}\{Y^{(n)}_{i,j} = 0\} = \operatorname{exp}(-X_{i,j}^{(n)})$, which together with the union bound gives $\operatorname{Pr}\{\cup_{i,j} \{Y^{(n)}_{i,j} = 0 \} \} \leq m_n n \operatorname{exp}(-X_{i,j}^{(n)}) \underset{n\rightarrow\infty}{\longrightarrow} 0$, in the asymptotic regime $m_n,n\rightarrow\infty$, ${m_n}/n\rightarrow\gamma$ ). 
Although our analysis here is concerned with an asymptotically positive realization of the random matrix $Y$, the numerical experiments in Section~\ref{sec:experiments - fit to MP law} demonstrate that our results also hold when a realization of $Y$ can contain zeros.

Under the conditions in Lemma~\ref{lem:convergence of scaling factors for estimated variances} and by combining Theorems~\ref{thm:Marchenko-Pastur for biwhitened noise} and~\ref{thm:Marchenko-Pastur for biwhitened noise from estimated variances}, we get that the empirical spectral distribution of $\hat{\Sigma}^{(n)}$ converges to the MP distribution $F_{\gamma,1}$, and furthermore, $\lambda_1 \{ \hat{\Sigma}^{(n)} \} \overset{p}{\longrightarrow} (1 + \sqrt{\gamma})^2$.
Consequently, 
we propose to estimate $r$ from $\hat{Y}$ analogously to~\eqref{eq:r_tilde def} by replacing $Y$ with $\hat{Y}$.
We mention that an analogous version of Theorem~\ref{thm:r_tilde no false detect} can be proved for this estimator  by repeating the proof of Theorem~\ref{thm:r_tilde no false detect} and making use of Theorems~\ref{thm:Marchenko-Pastur for biwhitened noise} and~\ref{thm:Marchenko-Pastur for biwhitened noise from estimated variances}.

Recall that the equations for $\hat{\mathbf{u}}$ and $\hat{\mathbf{v}}$ in~\eqref{eq:scaling equations from estimated variances} are equivalent to the equations in~\eqref{eq:matrix scaling general equations} for $\mathbf{x}$ and $\mathbf{y}$ when taking $A = Y$, $x_i = \hat{u}_i^2$, $y_j = \hat{v}_j^2$, $r_i = n$, and $c_j = m$.
Since a realization of the random matrix $Y$ may contain zeros, it is important to understand under which circumstances
we have existence and uniqueness of positive $\mathbf{x}$ and $\mathbf{y}$ that satisfy~\eqref{eq:matrix scaling general equations} for $r_i=n$ and $c_j = m$, where $A$ is a deterministic nonnegative matrix that represents a realization of $Y$ with zeros. We provide a comprehensive review of this topic in Appendix~\ref{appendix: existence and uniqueness}, and derive the following simple guarantee on the existence and uniqueness of $\mathbf{x}$ and $\mathbf{y}$ in terms of the zeros in $A$.
\begin{prop}[Existence and uniqueness of $\mathbf{x}$ and $\mathbf{y}$ for $r_i=n$, $c_j=m$] \label{prop:existcne and uniquness for A simple cond}
Let $r_i=n$ and $c_j=m$ for all $i\in [m]$ and $j\in [n]$ in~\eqref{eq:matrix scaling general equations}. Suppose that $A$ does not have any zero rows and columns, and that both requirements below are met:
\begin{enumerate}
    \item For each $k = 1, 2,\ldots, \lfloor {n}/{2} \rfloor$, $A$ has less than $\lceil m k/n \rceil$ rows that have at least $n-k$ zeros each.
    \item For each $\ell = 1,2,\ldots \lfloor {m}/{2} \rfloor$, $A$ has less than $\lceil n \ell /m \rceil$ columns that have at least $m-\ell$ zeros each. 
\end{enumerate} 
Then, there exists a pair $(\mathbf{x},\mathbf{y})$ of positive vectors that satisfies~\eqref{eq:matrix scaling general equations}, and it is unique up to a positive scalar, namely it can only be replaced with $(a\mathbf{x},a^{-1}\mathbf{y})$ for any $a>0$.
\end{prop}
Observe that the conditions in Proposition~\ref{prop:existcne and uniquness for A simple cond} are only concerned with rows or columns of $A$ that have at least $\lceil m/2 \rceil$ and $\lceil n/2 \rceil$ zeros, respectively. Therefore, the existence and uniqueness guarantees in Proposition~\ref{prop:existcne and uniquness for A simple cond} hold if the majority of entries in each row and each column of $A$ are positive.
Aside from this simple sufficient condition for existence and uniqueness, Proposition~\ref{prop:existcne and uniquness for A simple cond} also allows $A$ to have a small number of rows and columns in which all but a few entries are zero, and moreover, $A$ can potentially have a large number of rows or columns with more than half their entries being zero. 
See Appendix~\ref{appendix: existence and uniqueness} for more details on the existence and uniqueness of $\mathbf{x}$ and $\mathbf{y}$, and on appropriate preprocessing steps that can be taken to guarantee them.

\section{Experiments on simulated Poisson data} \label{sec:Poisson simulations}
\subsection{Fit against the MP law} \label{sec:experiments - fit to MP law}
We begin by measuring how well our proposed scaling of Poisson noise fits the theoretical MP law with $\sigma = 1$ (see Theorem~\ref{thm:Marchenko-Pastur for biwhitened noise}), and compare it against the empirical eigenvalue density of the original noise matrix (without any scaling). 
In Figure~\ref{fig:MPfit2} we illustrate the eigenvalue histograms (normalized appropriately) for the matrices $\Sigma_n = n^{-1} Y Y^T$, $\hat{\Sigma}_n = n^{-1} \hat{Y} \hat{Y}^T$, and $\widetilde{\Sigma}_n = n^{-1} \widetilde{Y} \widetilde{Y}^T$, using the aspect ratios $\gamma = m/n = 1/2,\,1/5$ and increasing dimension $n=100,\,500,\,5000$. For visualization purposes, we normalized $\Sigma_n$ by a scalar so that its largest eigenvalue is the MP law upper edge $(1+\sqrt{\gamma})^2$ (for $\sigma=1$). More details for reproducibility can be found in Appendix~\ref{reproducibility details Poisson convergence to MP law}.
We observe that while the eigenvalue histogram for the original noise (i.e., without any scaling) is very different from the MP law, there is a good fit between the MP law and the eigenvalue histogram of the noise after scaling (using either the exact or the estimated scaling factor) which clearly improves upon increasing matrix dimensions. 
The analogous results for the aspect ratios $\gamma = m/n = 1/3,\,1/4$ can be found in Figure~\ref{fig:MPfit2_extra_aspect_ratios} in Appendix~\ref{appendix:MP fit Poisson additional figures}

\begin{figure}[]
    \makebox[\textwidth][c]{
    \hspace{-0.5cm}
        \subfloat[][$\gamma = 1/2$, $n = 100,\, 500,\, 5000$]{
        \includegraphics[width=0.6\textwidth]{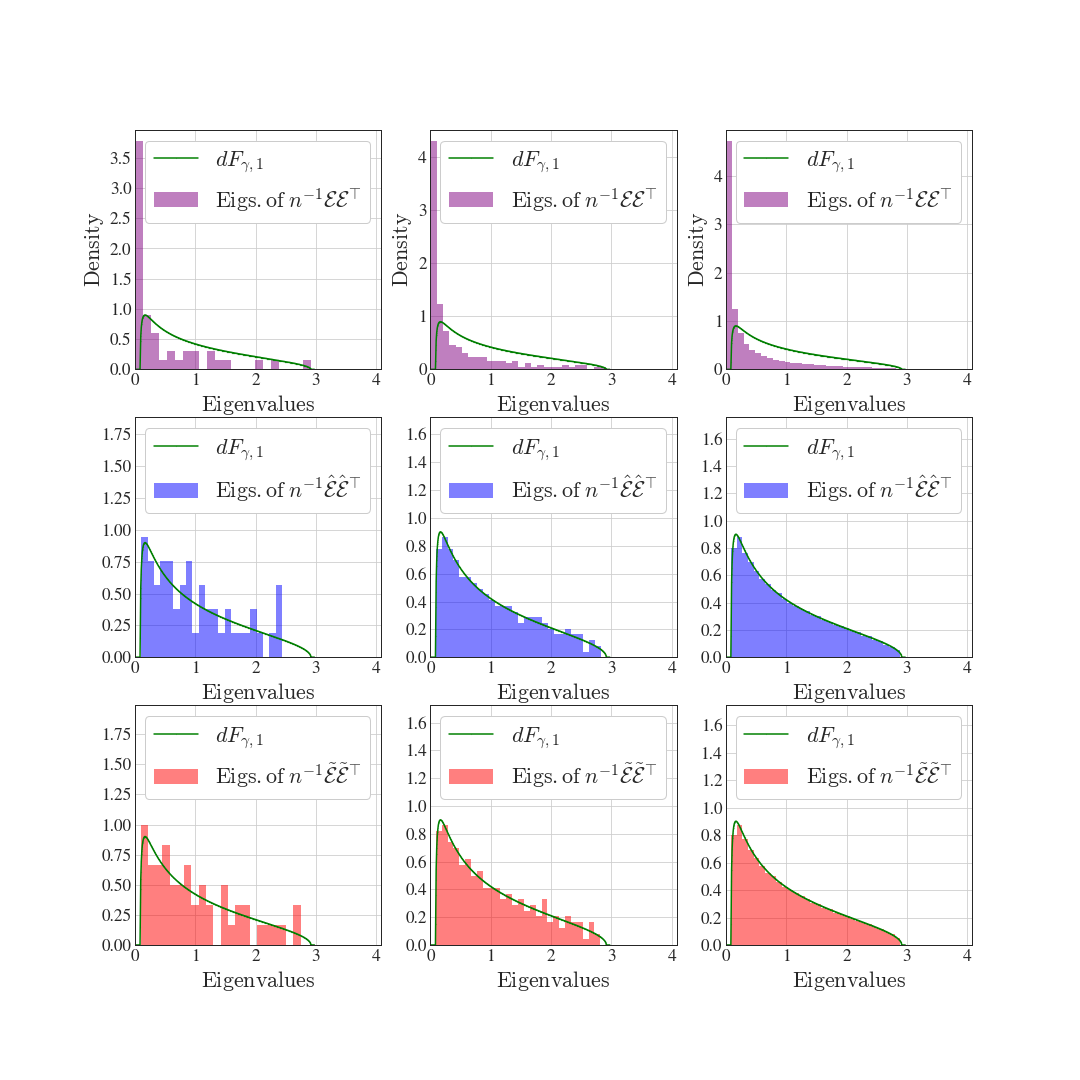}
        }
        \hspace{-1.3cm}
        \subfloat[][$\gamma = 1/5$, $n = 100,\, 500,\, 5000$]{
        \includegraphics[width=0.6\textwidth]{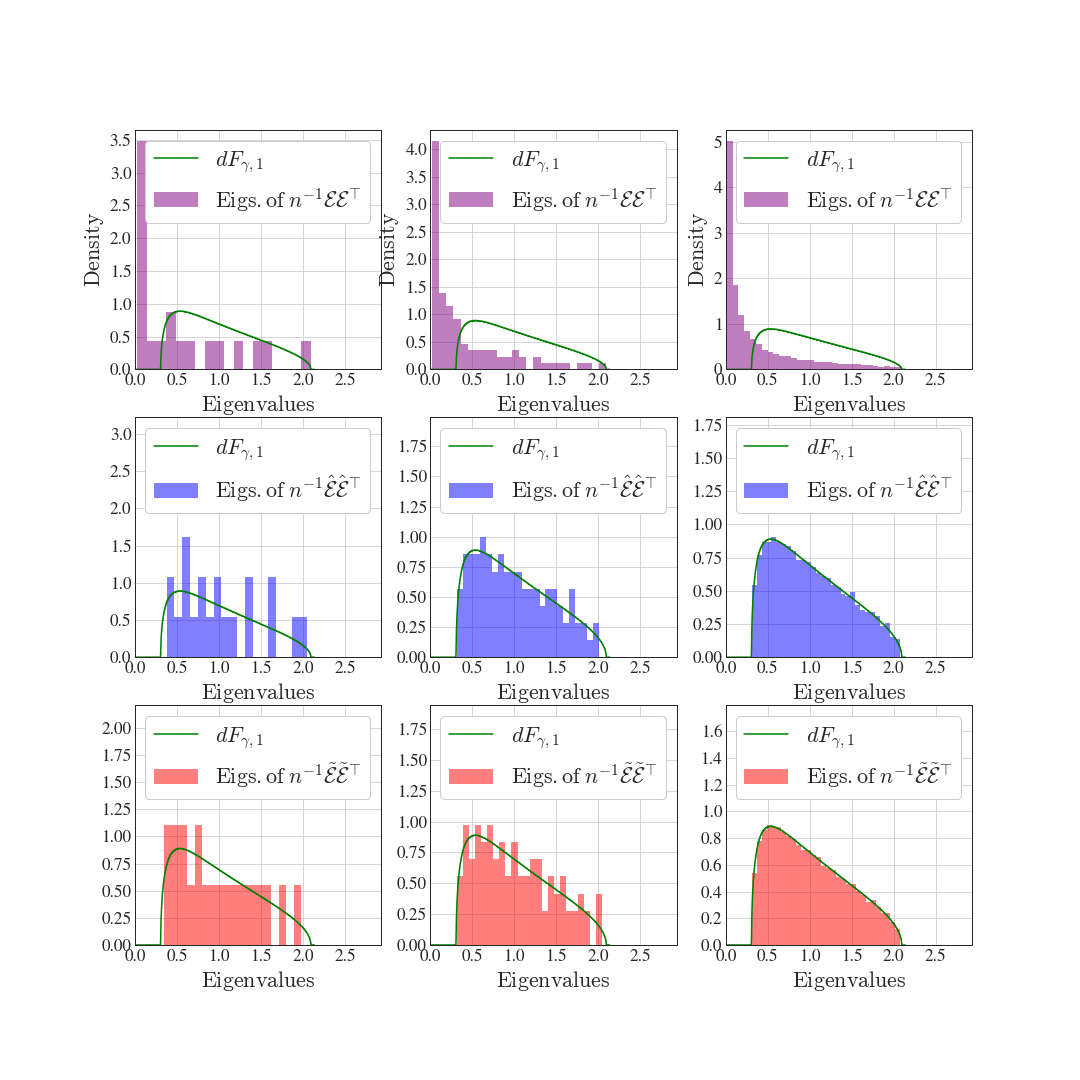}
        }
    }
    \caption{Spectrum of simulated Poisson noise versus the standard ($\sigma = 1$) Marchenko-Pastur density $dF_{\gamma, 1}$, for various aspect ratios $\gamma$ and matrix dimensions $n=100,500,5000$ (from left to right in each panel). The top row in each panel (purple) is the histogram of the eigenvalues of $\Sigma_n$ (i.e., without any row or column scaling). The center row in each panel (blue) is the histogram of the eigenvalues of $\hat{\Sigma}_n$ (i.e., after scaling with the estimated scaling factors). The bottom row in each panel (red) is the histogram of eigenvalues of $\widetilde{\Sigma}_n$ (i.e., after scaling with the exact scaling factors). 
    }
    \label{fig:MPfit2}
\end{figure}

In Figure~\ref{fig:MPfiterror}, we visualize the Kolmogorov-Smirnov (KS) distances (see~\cite{gotze2004rate}) $\sup_x \mathbb{E}\left[ |F_{\widetilde{\Sigma}_n}(x) - F_{\gamma, 1}(x)| \right]$ (red curve) and $\sup_x \mathbb{E}\left[ |F_{\hat{\Sigma}_n}(x) - F_{\gamma, 1}(x)| \right]$ (blue curve) as functions of the dimension $n$, where $F_{\widetilde{\Sigma}_n}$ and $F_{\hat{\Sigma}_n}$ are the empirical spectral distributions (see~\eqref{eq:empirical spectral distribution}) of $\widetilde{\Sigma}_n$ and $\hat{\Sigma}_n$, respectively, and $F_{\gamma, 1}$ is the MP distribution with $\sigma=1$. For this figure, we simulated $Y$ in the same way as we did for Figure~\ref{fig:MPfit2}, and averaged the KS distance over $20$ randomized experiments. As expected from the analysis in Section~\ref{sec:method and results}, we see that the empirical spectral distribution of the scaled noise matrix converges to the MP law, using either the exact or the estimated scaling factors. It is important to mention that even though Theorem~\ref{thm:Marchenko-Pastur for biwhitened noise from estimated variances} requires the Poisson parameters $X_{i,j}$ to increase with (slightly larger than a) logarithmic rate with $n$, the convergence to the MP law in this experiment is achieved without this condition (as the Poisson parameters in this experiment are upper bounded).

\begin{figure}[]
    \centering
    \makebox[\textwidth][c]{
        \subfloat[][$\gamma = 1/2$]{
        \includegraphics[width=0.5\textwidth]{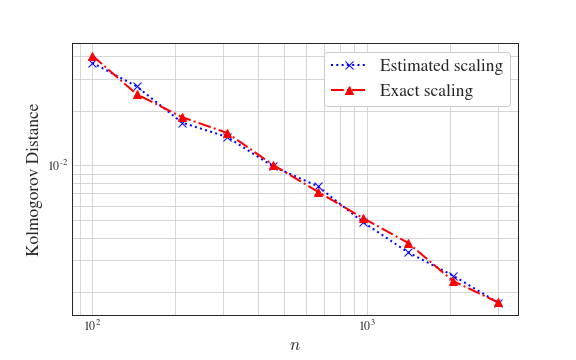}
        }
        
        \subfloat[][$\gamma = 1/5$]{
        \includegraphics[width=0.5\textwidth]{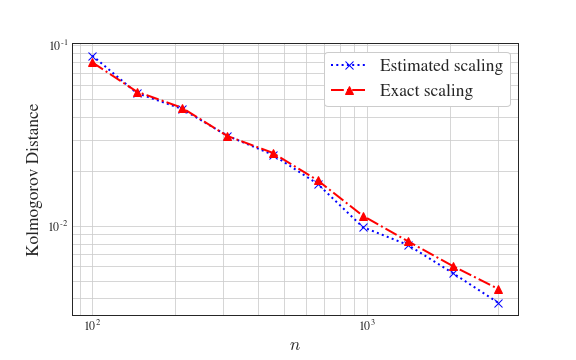}
        }
    }
    \caption{Convergence of the empirical spectral distribution (see~\eqref{eq:empirical spectral distribution}) of $\widetilde{\Sigma}_n$ (true scaling - red curve) and the empirical spectral distribution of $\hat{\Sigma}_n$ (estimated scaling - blue curve) to the MP distribution $F_{\gamma, 1}$, as the column dimension $n$ increases.}
    \label{fig:MPfiterror}
\end{figure}

\subsection{Rank estimation accuracy}\label{section:rank estimation}
Next, we demonstrate the accuracy of Algorithm~\ref{alg:noise standardization} for rank estimation under various degrees of heteroskedasticity, and compare it to several other methods: the Empirical Kaiser Criterion (EKC) \cite{braeken2017empirical}, pairwise Parallel Analysis (PA) (see Algorithm 1 in~\cite{dobriban2019deterministic}), Deflated Deterministic Parallel Analysis+ (DDPA+)~\cite{dobriban2019deterministic}, and Signflip Parallel Analysis (Signflip PA)~\cite{hong2020selecting}. For DDPA+, we input the column-wise centered matrix, where each column has its mean subtracted from it. For Signflip PA, we input the column-wise centered matrix whose rows are further normalized (i.e., each row is divided by its Euclidean norm), as suggested in~\cite{hong2020selecting} in their experiments. We found empirically that these preprocessing steps are essential for these methods to select a nontrivial rank (beyond $1$) in the settings of our experiments. 

We performed three experiments, each one simulating a different scenario. In each one, we generated a matrix $X$ of Poisson parameters in a different way, where the rank is always $20$ and the dimensions are $m=500$, $n=750$. In the first experiment, we generated a matrix $X$ with little variation in the magnitudes of the entries across rows or columns, facilitating a vanilla scenario of mild heteroskedasticity. In the second experiment, we generated a matrix $X$ with substantial variation in the magnitudes of the entries across rows or columns, corresponding to a challenging regime of strong heteroskedasticity. Lastly, in the third experiment we generated a matrix $X$ such that one of the $20$ signal components is much stronger than the rest (the noise is also heteroskedastic due to variation in the entries of the strong factor). This last regime is particularly challenging for methods based on Monte-Carlo simulations (PA, Signflip PA) due to \textit{shadowing}~\cite{dobriban2020permutation} -- a phenomenon where strong signal components introduce bias into the estimation of the spectrum of the noise. More details on the simulation of our different settings can be found in Appendix~\ref{appendix:reproducibility for rank selection simulations}. After generating $X$ for each experiment, we normalized $X$ (excluding the strong factor in the third experiment) so that its average Poisson parameter has a prescribed value that serves as an average signal-to-noise ratio. We then report the rank estimated by the different methods for a wide range of average Poisson parameter values. 
In Figure~\ref{fig: rank select heteroskedasticity} we plot the estimated ranks when averaged over $20$ randomized trials, as well as and their $[0.1, 0.9]$ quantiles, as a function of the average Poisson parameter in the matrix.

In the first scenario (Figure~\ref{fig:mild heteroskedasticity}), where the heteroskedasticity is mild, it is evident that all methods perform similarly. In particular, when the Poisson parameters are small, the signal eigenvalues are not large enough to be detected, hence the ranks estimated by all methods are $1$. Then, as the average Poisson parameter grows, more and more signal eigenvalues become large enough to be detected, and the estimated ranks (by all methods) gradually increase until they stabilize at the correct rank of $20$. Even in this easy setting, our method and the EKC perform slightly better than the other methods, allowing for accurate rank estimation for smaller Poisson parameters (lower signal-to-noise ratios). Intuitively, the reason for this is that both our method and the EKC perform a normalization that stabilizes the noise variances (explicitly in our method and implicitly in the EKC), which controls the spectrum of the noise and restricts its large eigenvalues from being dominated by rows/columns with larger noise variances.

In the second and third scenarios (Figures~\ref{fig:strong heteroskedasticity} and~\ref{fig:strong factor}), which are more challenging, the performance of the other methods substantially degrades. In particular, the methods that were not designed to handle heteroskedastic noise (all methods except ours and Signflip PA) perform poorly in the second scenario, and the methods that are based on simulations (PA, Signflip PA) perform poorly in the third scenario. DDPA+ can occasionally detect the correct rank in the third scenario (as it is designed to cope with strong signal factors by subtracting them), but its behavior is unstable due to the heteroskedastic noise. Evidently, our method is the only one that converges to the correct rank in a stable manner across all scenarios, and is able to detect the correct rank with smaller Poisson parameters (lower signal-to-noise ratios).

\begin{figure}
    \makebox[\textwidth][c]{
        \subfloat[Mild heteroskedasticity]
        {
        \includegraphics[width=0.38\textwidth]{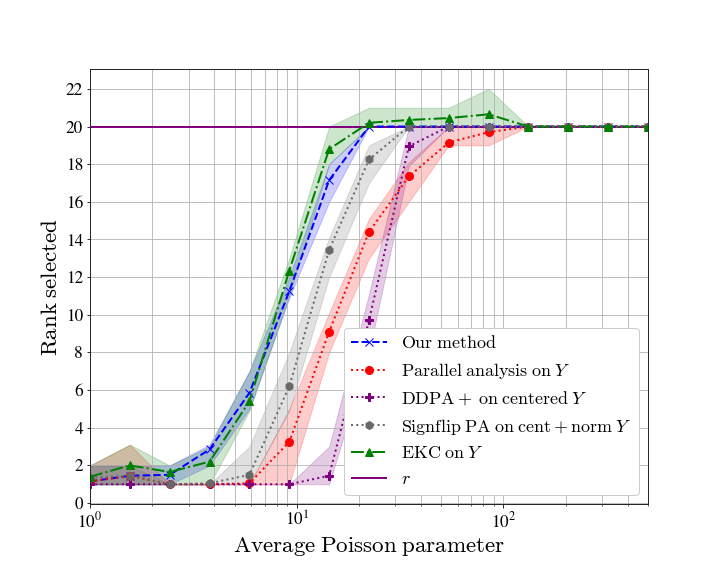}
        \label{fig:mild heteroskedasticity}
        } 
        \hspace{-1cm}
        \subfloat[Strong heteroskedasticity]{
        \includegraphics[width=0.38\textwidth]{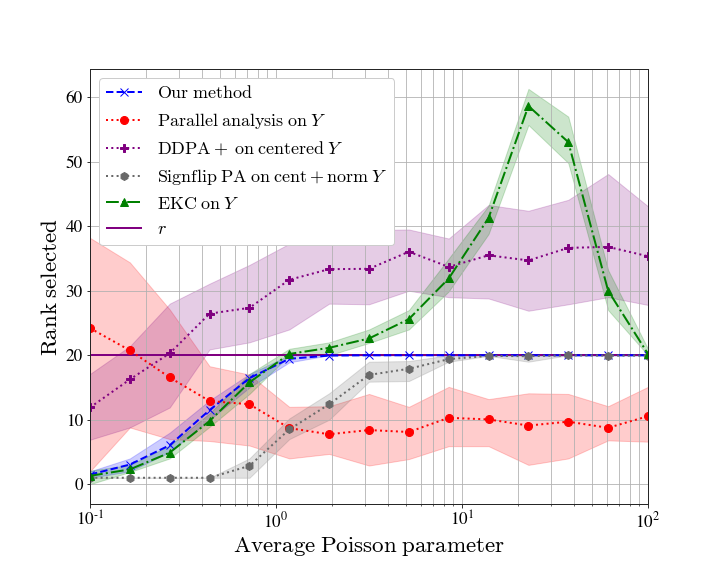} 
        \label{fig:strong heteroskedasticity}
        } 
        \hspace{-1cm}
        \subfloat[Presence of a strong factor]{
        \includegraphics[width=0.38\textwidth]{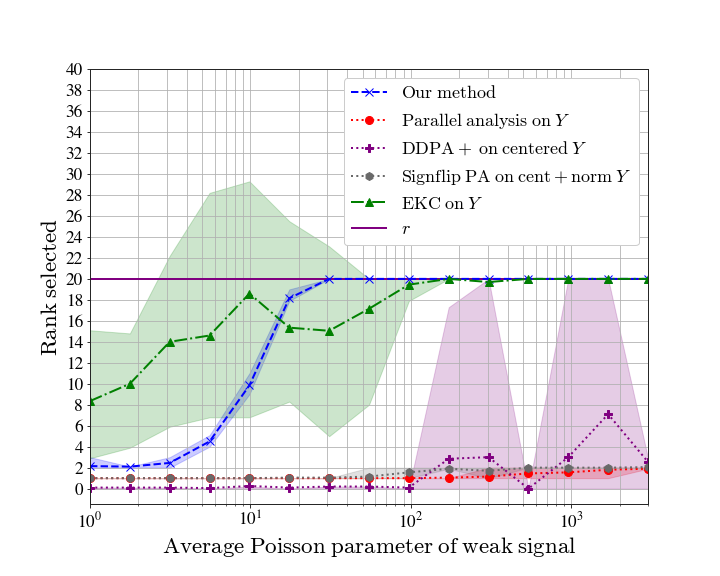} 
        \label{fig:strong factor}
        }
    }
    \caption{Rank selection accuracy of several methods, as well as their variability ($[0.1,0.9]$ quantile range depicted as the shaded areas), for a Poisson count matrix generated from $X$ with rank $r=20$, $m=500$, and $n=750$. On the x-axis, we have the average Poisson parameter in the matrix, while on the y-axis, we have the ranks determined by the algorithms. In the leftmost figure, the Poisson parameter matrix is generated to be relatively uniform, such that the noise is only mildly heteroskedastic. In the center figure, we demonstrate the impact of heteroskedasticity by substantially increasing the variability of the parameter matrix across rows and columns. In the rightmost figure, we consider the impact of having a strong signal component.} \label{fig: rank select heteroskedasticity}
\end{figure} 

\section{Beyond the Poisson distribution} \label{sec:beyond Poisson noise}
Let us consider the model~\eqref{eq:Y model def} where $Y_{i,j}$ are independent but not necessarily Poisson, $X = \mathbb{E}[Y]$ is the matrix whose rank is of interest, and $\mathcal{E} = Y - \mathbb{E}[Y]$ is the corresponding noise matrix. Recall that the main idea underlying our approach is to diagonally scale the data matrix $Y$ to $D(\mathbf{u}) Y D(\mathbf{v})$, so that the average variance in each row and in each column of the scaled noise matrix $D(\mathbf{u}) \mathcal{E} D(\mathbf{v})$ is $1$. 
Therefore, analogously to~\eqref{eq:scaling equations from true variances}, we consider positive $\mathbf{u}$ and $\mathbf{v}$ that satisfy
\begin{equation}
    1 = \frac{1}{m} \sum_{i=1}^m u_i^2 \operatorname{Var}[Y_{i,j}] v_j^2, \qquad \text{and} \qquad 1 = \frac{1}{n} \sum_{j=1}^n u_i^2 \operatorname{Var}[Y_{i,j}] v_j^2, \label{eq:scaling equations from true variances general model}
\end{equation}
for all $i\in[m]$ and $j\in [n]$, where $\operatorname{Var}[Y_{i,j}]$ is the variance of $Y_{i,j}$. Equivalently to Proposition~\ref{prop:existcne and uniquness for Lambda}, if all the variances $\operatorname{Var}[Y_{i,j}]$ are strictly positive, then positive $\mathbf{u}$ and $\mathbf{v}$ that satisfy~\eqref{eq:scaling equations from true variances general model} are guaranteed to exist and are unique up to a positive scalar.

Under appropriate conditions on the variances of $\{\mathcal{E}_{i,j}\}$ and their higher order moments, Theorem~\ref{thm:Marchenko-Pastur for biwhitened noise} can be trivially extended to account for distributions of $Y_{i,j}$ other than the Poisson. In particular, we have the following proposition, which provides general conditions under which the conclusions of Theorem~\ref{thm:Marchenko-Pastur for biwhitened noise} also hold.
\begin{prop} \label{prop:MP law for other distributions}
Suppose that there exist universal constants $C_k,C,c>0$ such that $ c < \max_{i,j} \operatorname{Var}[Y_{i,j}^{(n)}] \leq C \min_{i,j} \operatorname{Var}[Y_{i,j}^{(n)}]$ and $\mathbb{E}\vert\mathcal{E}_{i,j}^{(n)}\vert^{2k} \leq C_k (\operatorname{Var}[Y_{i,j}^{(n)}])^{k}$ for all $i\in[m]$, $j\in[n]$, $n\geq 1$, and $k\geq 1$. Then, the conclusions of Theorem~\ref{thm:Marchenko-Pastur for biwhitened noise} hold for $\widetilde{\Sigma}^{(n)} = n^{-1} \widetilde{\mathcal{E}}^{(n)} (\widetilde{\mathcal{E}}^{(n)})^T$, where $ \widetilde{\mathcal{E}}^{(n)}_{i,j} = u_i^{(n)} \mathcal{E}_{i,j}^{(n)} v_j^{(n)}$ and $(\mathbf{u}^{(n)},\mathbf{v}^{(n)})$ is a solution to~\eqref{eq:scaling equations from true variances general model}.
\end{prop}
The proof of Proposition~\ref{prop:MP law for other distributions} is obtained by repeating the proof outline of Theorem~\ref{thm:Marchenko-Pastur for biwhitened noise} and is omitted for the sake of brevity. 
Note that the proof of Theorem~\ref{thm:r_tilde no false detect} does not directly rely on the particular distribution of $Y_{i,j}$ but only on the results of Theorem~\ref{thm:Marchenko-Pastur for biwhitened noise} and Equation~\eqref{eq:rank(Lambda_tilde) = rank(Lambda)}. Hence, the conclusions of Theorem~\ref{thm:r_tilde no false detect} would also hold for distributions of $Y_{i,j}$ satisfying the conditions in Proposition~\ref{prop:MP law for other distributions}.

Continuing analogously to~\eqref{eq:scaling equations from estimated variances}, we propose to estimate the scaling factors $\mathbf{u}$ and $\mathbf{v}$ by replacing the true variances $\{\operatorname{Var}[Y_{i,j}]\}$ in~\eqref{eq:scaling equations from true variances general model} with corresponding independent unbiased estimators. That is, we propose to find positive $\hat{\mathbf{u}}$ and $\hat{\mathbf{v}}$ that satisfy
\begin{equation}
    1 = \frac{1}{m} \sum_{i=1}^m \hat{u}_i^2 \widehat{\operatorname{Var}}[Y_{i,j}] \hat{v}_j^2, \qquad \text{and} \qquad 1 = \frac{1}{n} \sum_{i=1}^n \hat{u}_i^2 \widehat{\operatorname{Var}}[Y_{i,j}] \hat{v}_j^2, \label{eq:scaling equations from estimated variances general model}
\end{equation}
where $\{\widehat{\operatorname{Var}}[Y_{i,j}]\}_{i\in[m],\;j\in[n]}$ are independent, nonnegative, and satisfy $\mathbb{E}[\widehat{\operatorname{Var}}[Y_{i,j}]] = \operatorname{Var}[Y_{i,j}]$ for all $i\in[m]$ and $j\in [n]$. Lemma~\ref{lem:convergence of scaling factors for estimated variances} and Theorem~\ref{thm:Marchenko-Pastur for biwhitened noise from estimated variances} can then be extended to distributions other than the Poisson by following our proof techniques and utilizing distribution-specific tail bounds.

Overall, Algorithm~\ref{alg:noise standardization} can be adapted to a distribution other than the Poisson by replacing $Y$ in Step~\ref{alg: step 1} with the matrix of noise variance estimators $(\widehat{\operatorname{Var}}[Y_{i,j}])_{i\in[m],\; j\in[n]}$, noting that existence and uniqueness of solutions to the scaling problem~\eqref{eq:scaling equations from estimated variances general model} depend on the zero pattern of this matrix as described at the end of Section~\ref{sec:estimating the scaling factors} and in Appendix~\ref{appendix: existence and uniqueness}. 

\subsection{Quadratic Variance Functions (QVFs)} \label{sec:quadratic variance functions}
We now consider a useful setting where $\widehat{\operatorname{Var}}[Y_{i,j}]$ can be computed directly from $Y_{i,j}$.
Suppose that the entries of $Y$ belong to a family of distributions that satisfies a quadratic relation between the mean and the variance, namely
\begin{equation}
    \operatorname{Var}[Y_{i,j}] = a + b X_{i,j} + cX_{i,j}^2, \label{eq: quadratic variance}
\end{equation}
for all $i\in[m], j\in[n]$, where $a,b,c$ are fixed constants. Evidently, the Poisson model~\eqref{eq:Y poisson model} satisfies~\eqref{eq: quadratic variance} with $a=c=0$ and $b=1$. To clarify the nomenclature, the term `family of distributions' in the context of the Poisson refers to the set of Poisson distributions with all possible Poisson parameters.

Perhaps the most studied families that satisfy the quadratic relation~\eqref{eq: quadratic variance} are the Natural Exponential Families with Quadratic Variance Functions (NEF-QVFs)~\cite{morris1982natural,morris1983natural}. The NEF-QVFs include six fundamental families: the normal with fixed variance ($b=c=0$, $a>0$), the Poisson ($a=c=0$, $b=1$), the binomial with fixed number of trials ($a=0$, $b=1$, $c<0$), the negative binomial with fixed number of failures ($a=0$, $b=1$, $c>0$), the gamma with fixed shape parameter ($a=b=0$, $c>0$), and a family known as the Generalized Hyperbolic Secant ($a>0$, $b=0$, $c>0$). It was shown in~\cite{morris1982natural} that these six fundamental families, together with all of their possible linear transformations (scaling and translation), convolutions (sum of $k$ i.i.d variables), and divisions (the inverse of a convolution), form all possible NEF-QVFs. If the entries of $Y$ belong one of these families, then $Y_{i,j}$ satisfies~\eqref{eq: quadratic variance} with a distinct set of coefficients $a,b,c$ (unique to a specific NEF-QVF family). The NEF-QVFs also admit many favorable properties and satisfy the moment condition required in Proposition~\ref{prop:MP law for other distributions} (see Theorem 3 in~\cite{morris1982natural} and the discussion immediately following the proof).

We now have the following proposition.
\begin{prop}[Unbiased variance estimator for QVFs] \label{prop:existence of variance estimator}
If the entries of $Y$ satisfy the QVF~\eqref{eq: quadratic variance} with $c\neq -1$, then an unbiased variance estimator for $Y_{i,j}$ is given by
\begin{equation}
    \widehat{\operatorname{Var}}[Y_{i,j}] = \frac{a + b Y_{i,j} + c Y_{i,j}^2}{1 + c}. \label{eq:quadratic noise variance estimator}
\end{equation}
Among all of NEF-QVFs, the family of distributions consisting of the Bernoulli and its linear transformations is {the only one} with $c=-1$. For this family there is no unbiased variance estimator that is only a function of $Y_{i,j}$. 
\end{prop}
The proof can be found in Appendix~\ref{appendix:proof of prop on existence of variance estimator}.
Proposition~\ref{prop:existence of variance estimator} shows that all NEF-QVFs except the Bernoulli (and its linear transformations) admit unbiased variance estimators that can be easily computed from $Y_{i,j}$. 
More generally, our methodology is applicable to any family of distributions satisfying a generic QVF~\eqref{eq: quadratic variance} (with $c \neq -1$), not just NEF-QVFs. Such families include the Generalized Poisson~\cite{consul1973generalization} with fixed dispersion, the log-normal with fixed variance (of the natural logarithm), the beta with fixed sample size (the sum of the two shape parameters), and the beta-binomial with fixed number of trials and intra-class correlation parameters, among (infinitely-many) other families of distributions. 
We note that since the quadratic variance relation~\eqref{eq: quadratic variance} involves only the first and second moments, the higher moments of a distribution with a QVF can be arbitrary. Hence, the coefficients $a,b,c$ do not uniquely identify a specific family of distributions in general.  

In Appendix~\ref{appendix:experiments for other distributions}  we conduct numerical experiments analogous to the ones in Section~\ref{sec:experiments - fit to MP law} on three families of distributions: the binomial, negative binomial, and generalized Poisson. For these families we demonstrate the convergence of the spectrum of the noise after scaling to the MP law empirically (using the variance estimator~\eqref{eq:quadratic noise variance estimator} as well as the true variance~\eqref{eq: quadratic variance}).

\subsection{Missing entries and zero-inflation}
It is worthwhile to mention that the QVF~\eqref{eq: quadratic variance} can also accommodate for entries missing at random when imputed with zeros, or equivalently, if some of the entries are randomly assigned a zero value (i.e., zero inflation). In particular, suppose that we observe 
\begin{equation} \label{eq:missing entries}
    \overline{Y}_{i,j} = 
    \begin{dcases}
    {Y}_{i,j}, & \text{with probability} \; p,\\
    0 & \text{with probability} \; 1 - p,
    \end{dcases}
\end{equation}
for some $p\in (0,1]$, where the entries of $Y$ satisfy the QVF~\eqref{eq: quadratic variance}. Then, we have $\overline{X}:= \mathbb{E}[\overline{Y}] = p X$, whose rank is the same as $X$, and a direct calculation shows that 
\begin{equation}
    \operatorname{Var}[\overline{Y}_{i,j}] = ap + bpX_{i,j} + p(c+1-p)X_{i,j}^2 = ap + b \overline{X}_{i,j} + \left(\frac{c+1-p}{p}\right) \overline{X}_{i,j}^2.
\end{equation}
Consequently, the observation model~\eqref{eq:missing entries} satisfies a QVF of the form~\eqref{eq: quadratic variance} with coefficients that depend on the observation probability $p$. Then, the corresponding variance estimator is
\begin{equation}
    \widehat{\operatorname{Var}}[\overline{Y}_{i,j}] = \frac{a p^2 + b p \overline{Y}_{i,j} + (1+c-p)\overline{Y}_{i,j}^2}{1+c}.
\end{equation}
Note that the variance estimator ends up with a quadratic term $\overline{Y}_{i,j}^2$ even if the QVF of $Y_{i,j}$ has $c=0$. Therefore, our biwhitening approach can be useful even in a setting of homoskedastic noise ($a>0$, $b=c=0$) with missing entries, in which case the noise becomes heteroskedastic and satisfies a nontrivial QVF.

\section{Practical considerations and real data} \label{sec:practical considerations and real data}
\subsection{Adapting to the data} \label{sec:adapting to the data}
When analyzing experimental data, it is desirable to have a practical approach for finding the most appropriate coefficients in the QVF~\eqref{eq: quadratic variance} automatically.
To that end, we focus on QVFs with $a=0$, $b\geq 0$, and $c\geq 0$. Such QVFs are naturally found in certain families of nonnegative random variables. In particular, any family that includes distributions with nonnegative variables with zero means must satisfy $a=0$ (since a nonnegative random variable with zero mean must have zero variance). The restriction $b\geq 0$ and $c\geq 0$ is convenient to insure that the variance estimator~\eqref{eq:quadratic noise variance estimator} is always nonnegative if $Y_{i,j}$ is nonnegative. 
Substituting $a=0$, $b = {\alpha (1-\beta)}/(1-\alpha\beta)$, and $c = {\alpha \beta}/(1-\alpha\beta)$ in~\eqref{eq:quadratic noise variance estimator} for any $\beta\in [0,1]$ and $\alpha > 0$, gives
\begin{equation}
    \widehat{\operatorname{Var}}[Y_{i,j}] = \alpha \left[ (1-\beta) Y_{i,j} + \beta Y_{i,j}^2 \right]. \label{eq:variance estimator with alpha and beta}
\end{equation}
Intuitively, the parameter $\beta\in[0,1]$ interpolates between a purely linear variance function and a purely quadratic variance function, whereas $\alpha>0$ controls the global scaling. 

Let us fix $\alpha = 1$ and some $\beta\in [0,1]$. For a given matrix $Y$, the scaling equations~\eqref{eq:scaling equations from estimated variances general model} can be solved when plugging the variance estimator~\eqref{eq:variance estimator with alpha and beta}, providing the pair of scaling factors $(\hat{\mathbf{u}},\hat{\mathbf{v}})$. The scaled data matrix is then given by $\hat{Y} = D(\hat{\mathbf{u}}) Y D(\hat{\mathbf{v}})$, and the corresponding Gram matrix is defined as $\hat{\Sigma} = n^{-1} \hat{Y} \hat{Y}^T$. Importantly, since $\alpha$ controls the global scaling of the matrix $(\widehat{\operatorname{Var}}[Y_{i,j}])_{i\in[m],\; j\in[n]}$, it is clear form the equations~\eqref{eq:scaling equations from estimated variances general model} that the pair of scaling factors that solves~\eqref{eq:scaling equations from estimated variances general model} for any $\alpha>0$ is given by $(\alpha^{-1/4} \hat{\mathbf{u}},\alpha^{-1/4} \hat{\mathbf{v}})$. Consequently, the scaled data matrix for any $\alpha>0$ is given by ${\alpha}^{-1/2} \hat{Y}$, and the corresponding Gram matrix is $\alpha^{-1} \hat{\Sigma}$.
Note that the eigenvalues of $\alpha^{-1} \hat{\Sigma}$ are related to those of $\hat{\Sigma}$ by the scalar multiple $\alpha^{-1}$. 

We now propose a method for choosing $\alpha$ and $\beta$ automatically from $Y$ by minimizing the discrepancy between the spectrum of the resulting scaled data matrix and the MP law (relying on the fact that a low-rank perturbation of the scaled noise matrix $\hat{\mathcal{E}}$ does not change its limiting spectral distribution). 
For any fixed $\beta \in [0,1]$, we choose $\alpha$ by matching the median of the eigenvalues of $\alpha^{-1}\hat{\Sigma}$ to the median of the MP distribution $F_{\gamma,1}$. That is, we take
\begin{equation}
    {\alpha} = \frac{\lambda_{\text{med}}\{\hat{\Sigma}\}}{\mu_{\gamma}}, \label{eq:estimator for alpha}
\end{equation}
where $\gamma = {m}/{n}$, $\lambda_{\text{med}}\{\hat{\Sigma}\}$ is the median eigenvalue of $\hat{\Sigma}$ (which depends on the value of $\beta$), and $\mu_\gamma$ is the median of the MP distribution with parameter $\gamma$ and noise variance $1$, i.e., $\mu_\gamma$ is the unique solution to the equation
\begin{equation}
    F_{\gamma,1}(t) = \int_{\beta_-}^t \frac{\sqrt{(\tau - \beta_-)(\beta_+-\tau)}}{2\pi \gamma \tau} d\tau = \frac{1}{2},
\end{equation}
 in the variable $t\in (\beta_- , \beta_+)$, and $\beta_{\pm} = (1\pm\sqrt{\gamma})^2$. Our approach here for choosing $\alpha$ is equivalent to the method proposed in~\cite{gavish2014optimal} for estimating $\sigma$ in the MP density~\eqref{eq:MP density}. The use of the median in this context is advantageous due its robustness to outlier eigenvalues, e.g., signal components or finite-sample fluctuations of the noise eigenvalues at the edges of the spectrum. 
 Then, we choose $\beta\in[0,1]$ that minimizes the Kolmogorov-Smirnov (KS) distance 
 \begin{equation}
     \sup_x \left\vert F_{[{\alpha}^{-1} \hat{\Sigma}]}(x) - F_{\gamma,1}(x) \right\vert, \label{eq: beta estimator}
 \end{equation}
where $F_{\gamma,1}$ is the MP distribution, $F_{[{\alpha}^{-1} \hat{\Sigma}]}$ is the ESD of ${\alpha}^{-1} \hat{\Sigma}$ (see~\eqref{eq:empirical spectral distribution}), and we emphasize that $\hat{\Sigma}$ and $\alpha$ depend on the value of $\beta$.
Since the minimization over all $\beta\in [0,1]$ is intractable (as it involves solving~\eqref{eq:scaling equations from estimated variances general model} for each $\beta$ separately), we propose to approximate this minimization by scanning over a finite grid of values in $[0,1]$. 

We demonstrate the above approach numerically in Appendix~\ref{appendix:adaptive method numerical experiments} on a simulated negative binomial matrix. We consider both a situation where the number of negative binomial failures is fixed across the entries in the matrix, and also a situation where the number of negative binomial failures is varying -- exemplifying the robustness of our approach to certain violations of the QVF model~\eqref{eq: quadratic variance}.

In order to validate a certain choice of $\alpha$ and $\beta$ in our framework, it is natural to consider the quality of the fit of the spectrum after scaling to the MP law (assuming that the rank of the signal matrix is sufficiently small). 
To account for the bias in the selection procedure of $\alpha$ and $\beta$, we propose to make use of a sample splitting scheme as follows. First, we determine $\alpha$ and $\beta$ from a submatrix of $Y$ (e.g., half the columns of $Y$) as described in this section. Then, we treat $\alpha$ and $\beta$ as known model parameters to solve~\eqref{eq:scaling equations from estimated variances general model} on a disjoint submatrix (e.g., the remaining columns), and measure the fit of the spectrum after scaling to the MP law. For the measure of goodness-of-fit, we use the Kolmogorov-Smirnov (KS) test (whose statistic is the quantity minimized in~\eqref{eq: beta estimator} to choose $\beta$ for a disjoint submatrix). This procedure is repeated several times on randomly-chosen submatrices and the results (KS distances, p-values) are averaged across trials.
We note that the null hypothesis underlying this methodology is that the eigenvalues of the scaled data matrix are sampled independently from the MP distribution. While this assumption does not strictly hold (due to the existence of signal components in the data, and also since the eigenvalues of a random noise matrix are dependent), it serves as a useful surrogate null hypothesis that allows for an interpretable measure of goodness-of-fit. 

\subsection{Fit to the MP law for real data} \label{sec:real data fit to the MP law}
We now exemplify our biwhitening approach on several real datasets from three domains of application: Single-Cell RNA Sequencing (scRNA-seq), High-Throughput Chromosome Conformation Capture (Hi-C), and document topic modeling. For scRNA-seq, we used the well-studied purified Peripheral Blood Mononuclear Cells (PBMCs) dataset from Zheng et al.~\cite{zheng2017massively}, and the dataset by Hrvatin et al.~\cite{hrvatin2018single} that contains mouse visual cortex cells. For Hi-C, we used the dataset by Johanson et al.~\cite{johanson2018genome} from Naïve CD4+ T cells from homo-sapiens, where we extracted the submatrix of interactions between chromosomes one and two (the largest chromosomes), corresponding to 4622$\times$4822 different pairs of loci. For document topic modeling, we used the Associated Press dataset~\cite{harman1993first} (containing 10473 terms in 2246 documents), and the 20 NewsGroups dataset~\cite{Lang95} (containing 61188 terms in 18774 documents). 

We applied downsampling and filtering steps to the datasets to control their size and sparsity; see Appendix~\ref{appendix:reproducibility details for real data MP fit} for more details. Then, we applied the procedure described in Section~\ref{sec:adapting to the data} to find $\alpha>0$ and $\beta \in \{0,0.05,0.1,\ldots,0.95,1\}$ automatically from half the observations (half the cells in each scRNA-seq dataset, half the loci in chromosome 1 in the Hi-C dataset, and half the documents in each topic modeling dataset), and employed the resulting variance estimator~\eqref{eq:variance estimator with alpha and beta} for the remaining half of the observations to find the scaling factors $(\hat{\mathbf{u}},\hat{\mathbf{v}})$ and compute the scaled matrix $\hat{Y} = D(\hat{\mathbf{u}}) Y D(\hat{\mathbf{v}})$. 

Figure~\ref{fig:MP fit all datasets} depicts the fit of the histogram of the eigenvalues of $n^{-1} \hat{Y} \hat{Y}^T$ to the MP law for each of the above-mentioned datasets (on the held-out part of the data matrix). For comparison, we also show the analogous fits for the standard Poisson variance estimator $\widehat{\operatorname{Var}}[Y_{i,j}] = Y_{i,j}$, as well as for a constant variance estimator $\widehat{\operatorname{Var}}[Y_{i,j}] = \alpha$, i.e., assuming homoskedastic noise with unknown variance, where $\alpha$ is set according to~\eqref{eq:estimator for alpha}. To provide an interpretable measure of goodness-of-fit, we applied the KS test as described at the of Section~\ref{sec:adapting to the data} over 10 randomized trials. The averaged KS distances and the corresponding p-values are summarized in Table~\ref{tab:fit to MP law for real data}. The average chosen parameters were $\alpha=1.02, \; \beta=0$ for the Hi-C dataset, $\alpha=1.01, \; \beta = 0.2$ for the PBMC dataset,  $\alpha=1.02, \; \beta = 0.41$ for the Hrvatin et al. dataset, $\alpha=0.88, \; \beta = 1$ for the AP dataset, and $\alpha=0.95, \; \beta = 1$ for the 20 NewsGroups dataset. 

It is evident that for each of the five datasets, the spectrum of the original counts does not agree with the MP law even after adjusting for an unknown scalar noise variance. Indeed, the p-values from the KS test in this case are all extremely small. Hence, none of these datasets can be assumed to have homoskedastic noise. On the other hand, the simple Poisson model is already useful for the Hi-C data, as it provides an excellent fit to the MP law, which is improved only slightly after quadratic variance adjustment (selecting parameters that are very close to those of the standard Poisson). For all other datasets the Poisson model is inadequate, but we obtain accurate fits to the MP law using the quadratic variance estimator~\eqref{eq:variance estimator with alpha and beta} with the chosen parameters $\alpha$ and $\beta$. In particular, the p-values after the quadratic variance adjustment are in a range where the KS test cannot be rejected with high significance (implying that it is quite likely to obtain the observed KS distance if the eigenvalues are actually sampled from the MP distribution). These results suggest that many real-world datasets can be scaled appropriately (by diagonal scaling) to make the empirical spectral distribution very close to the MP law, allowing for adaptive signal detection. Interestingly, the chosen $\alpha$ and $\beta$ for the scRNA-seq datasets agree well with a negative binomial model (where $b=1$). The negative binomial is a standard model for scRNA-seq data, explained by a Poisson observation model with a gamma prior on the Poisson parameter~\cite{sarkar2021separating}.

\begin{figure}
    \centering
    \makebox[\textwidth][c]{
        \subfloat[][PBMC: constant variance]{
        \includegraphics[width=0.25\textwidth]{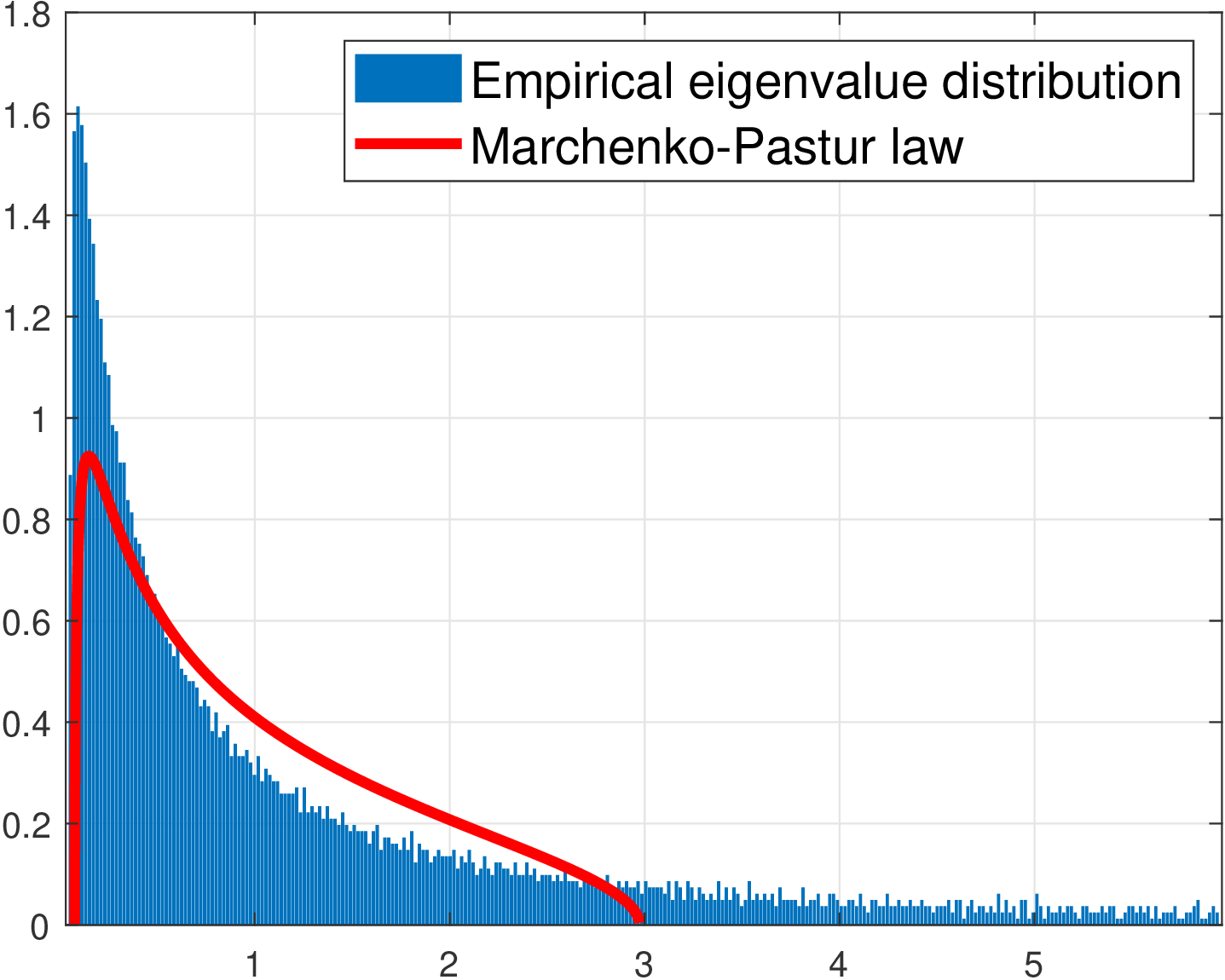}
        }
        
        \subfloat[][PBMC: Poisson variance]{
        \includegraphics[width=0.25\textwidth]{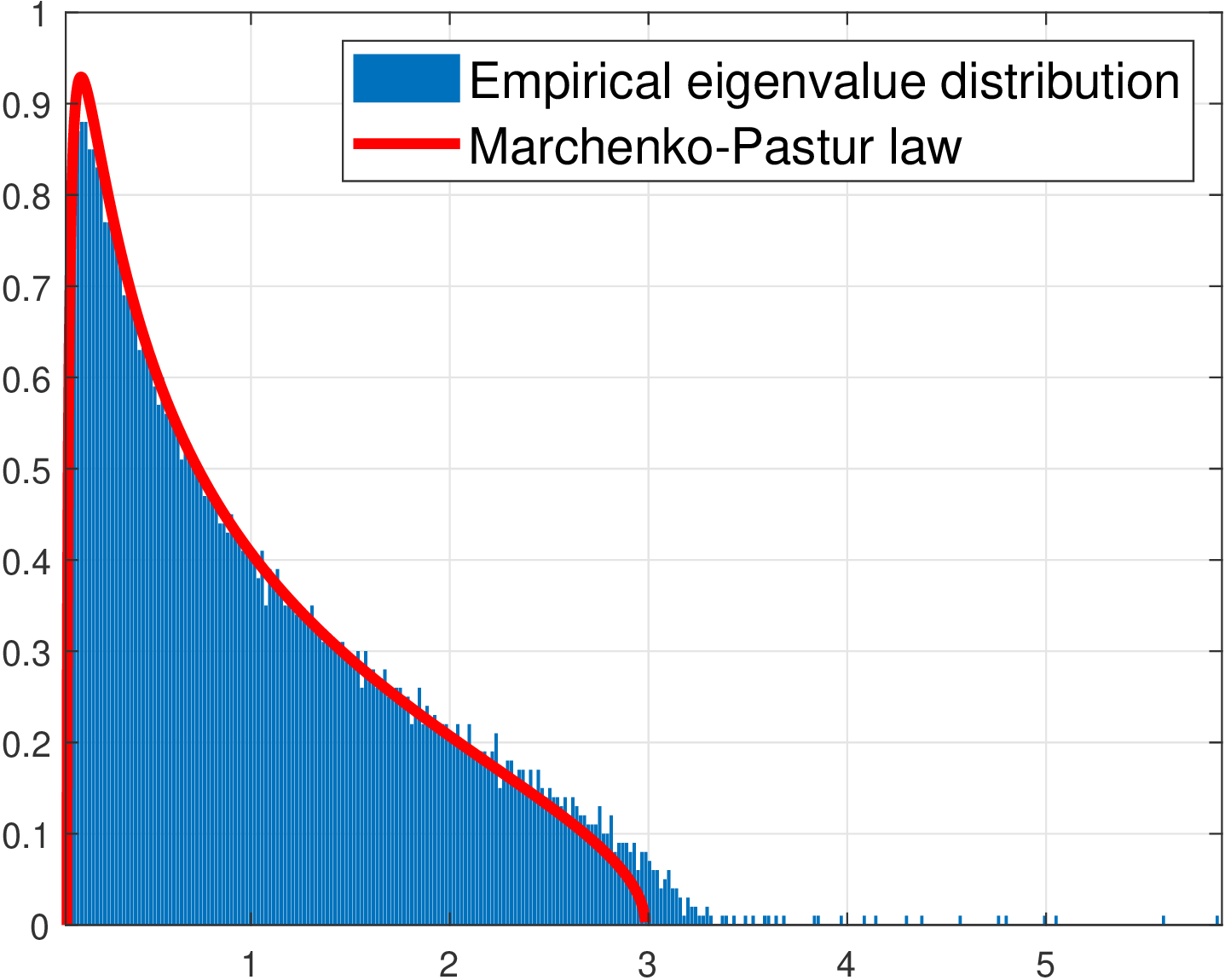}
        }
        
        \subfloat[][PBMC: quadratic variance]{
        \includegraphics[width=0.25\textwidth]{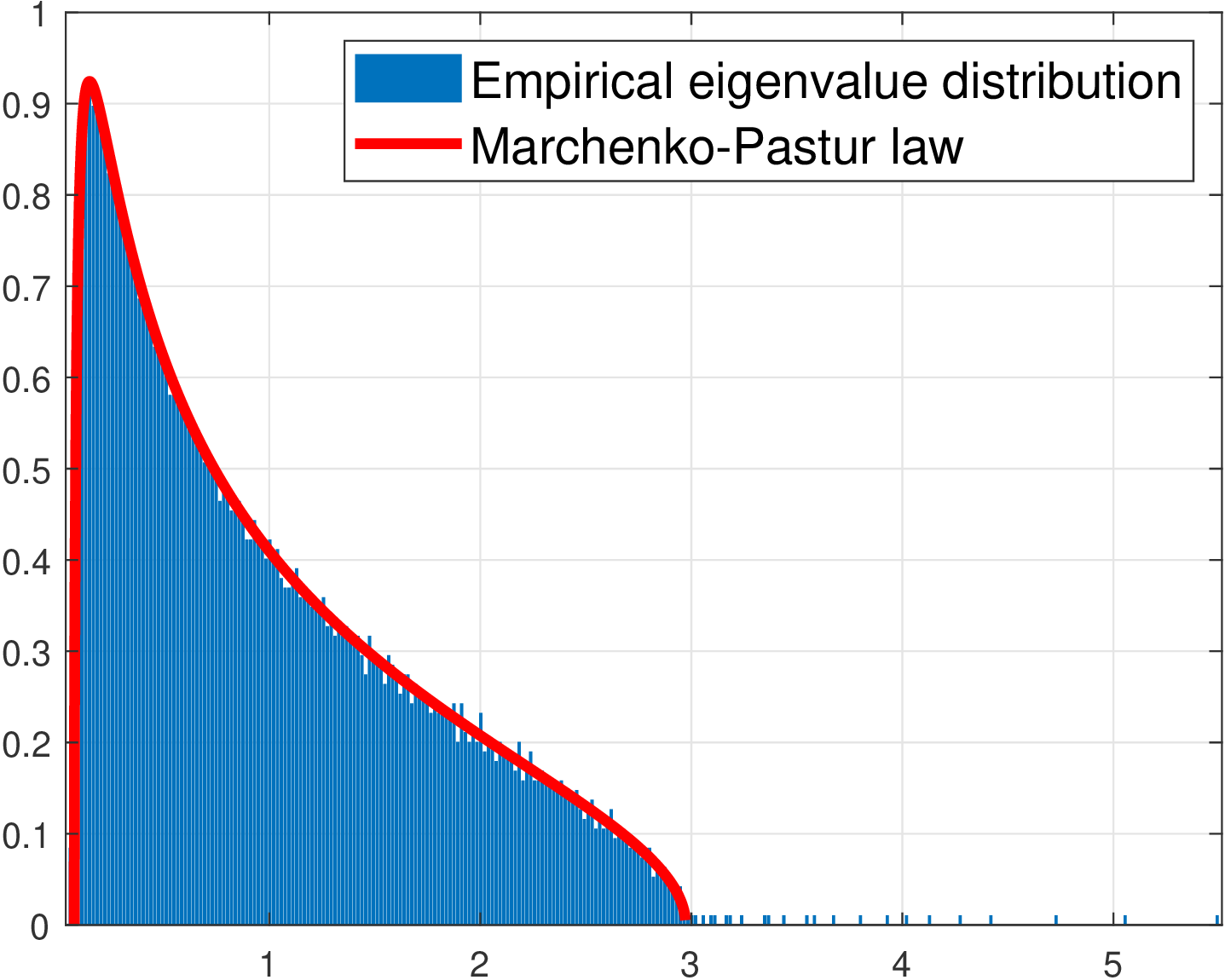}
        } 
    }
    \\
    \makebox[\textwidth][c]{
        \subfloat[][Hrvatin: constant variance]{
        \includegraphics[width=0.25\textwidth]{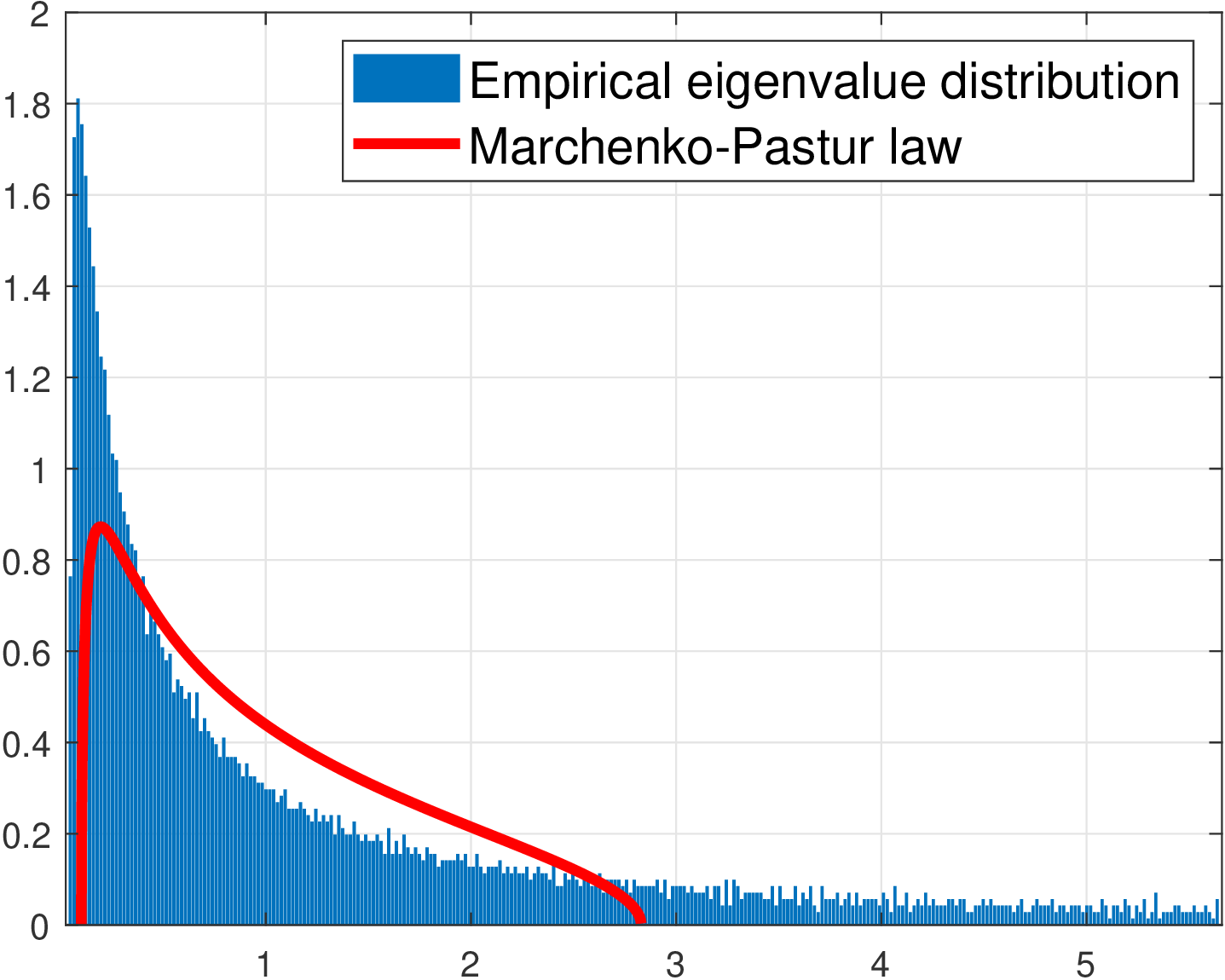}
        }
        
        \subfloat[][Hrvatin: Poisson variance]{
        \includegraphics[width=0.25\textwidth]{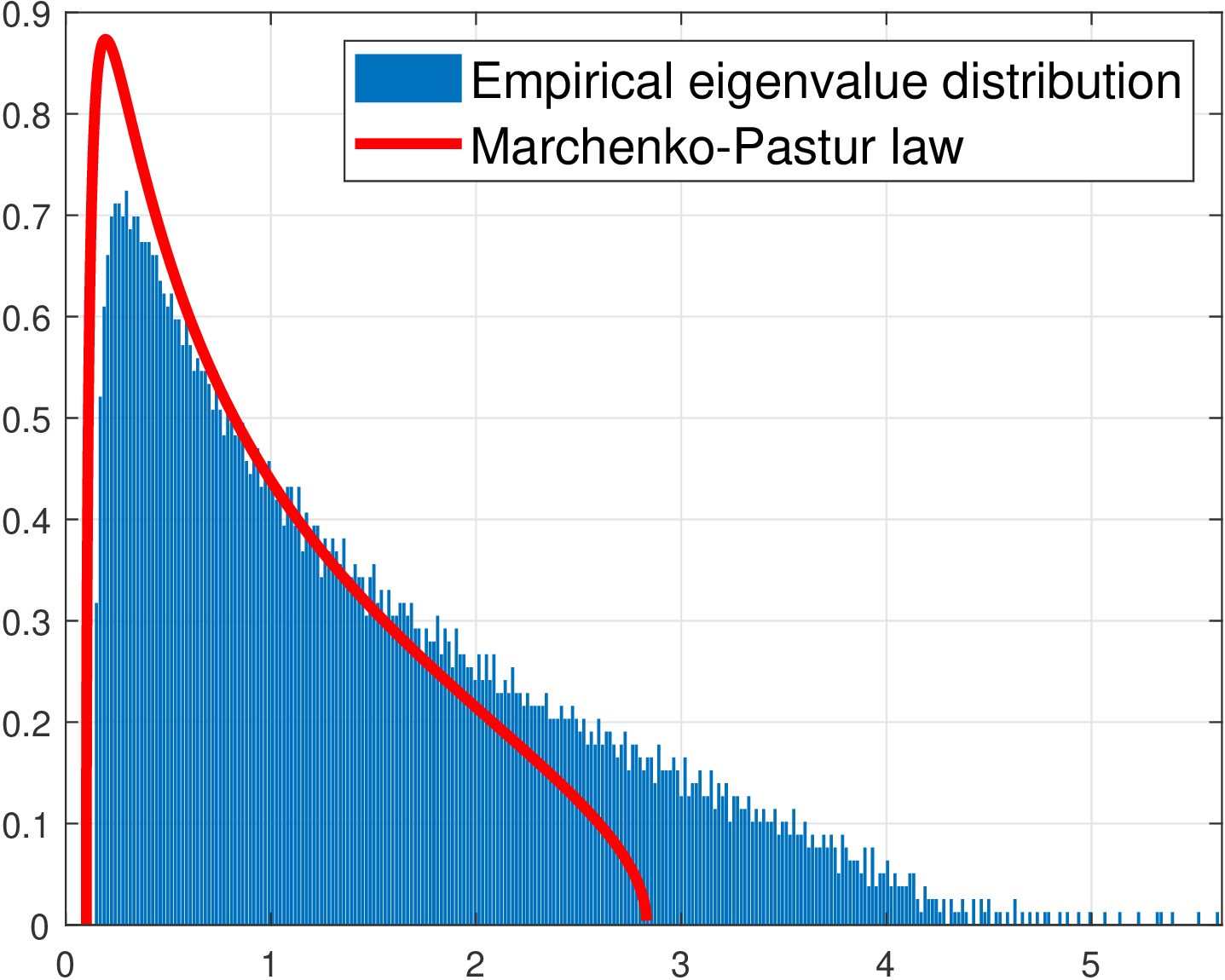}
        }
        
        \subfloat[][Hrvatin: quadratic variance]{
        \includegraphics[width=0.25\textwidth]{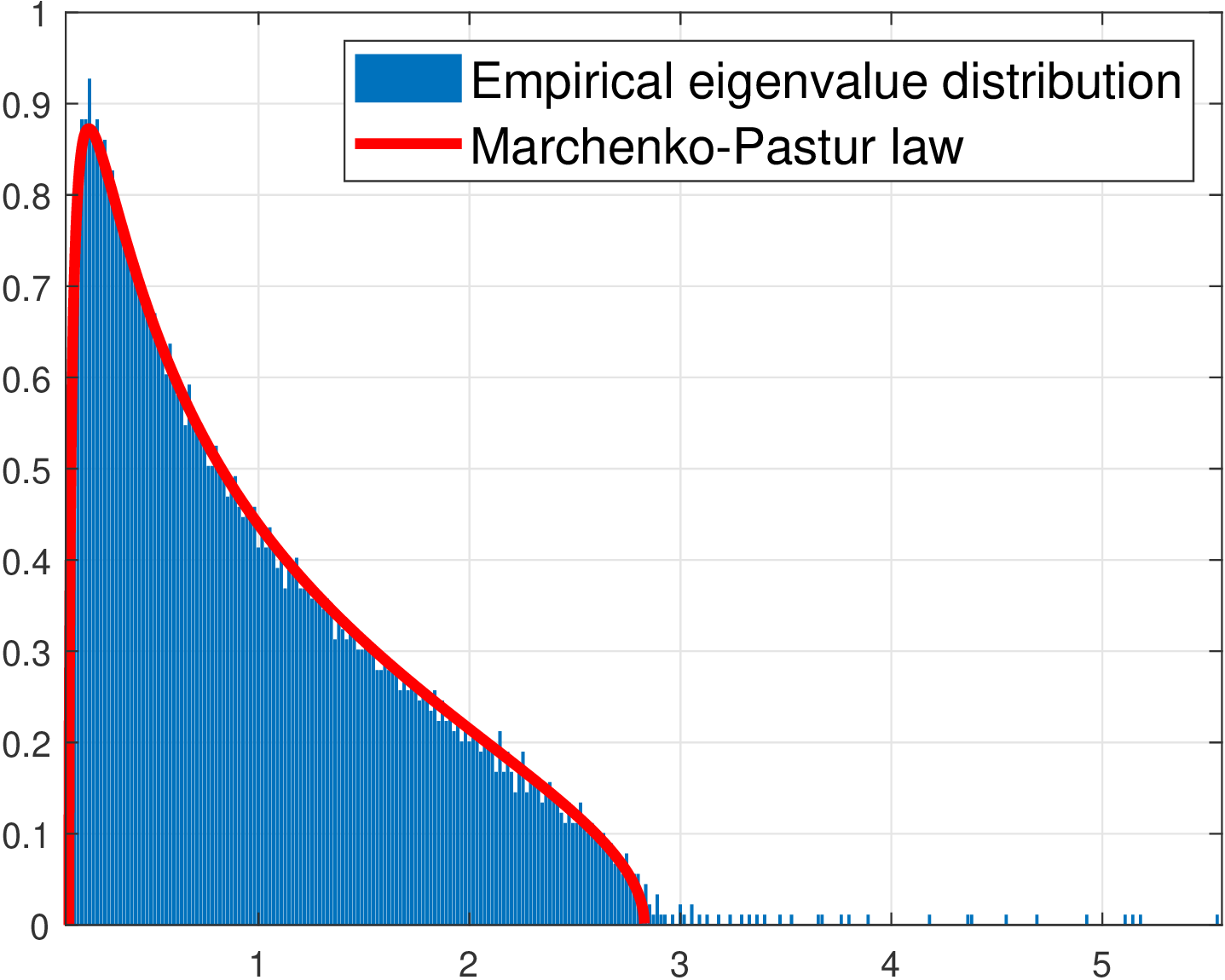}
        } 
    }
    \\
    \makebox[\textwidth][c]{
        \subfloat[][Hi-C: constant variance]{
        \includegraphics[width=0.25\textwidth]{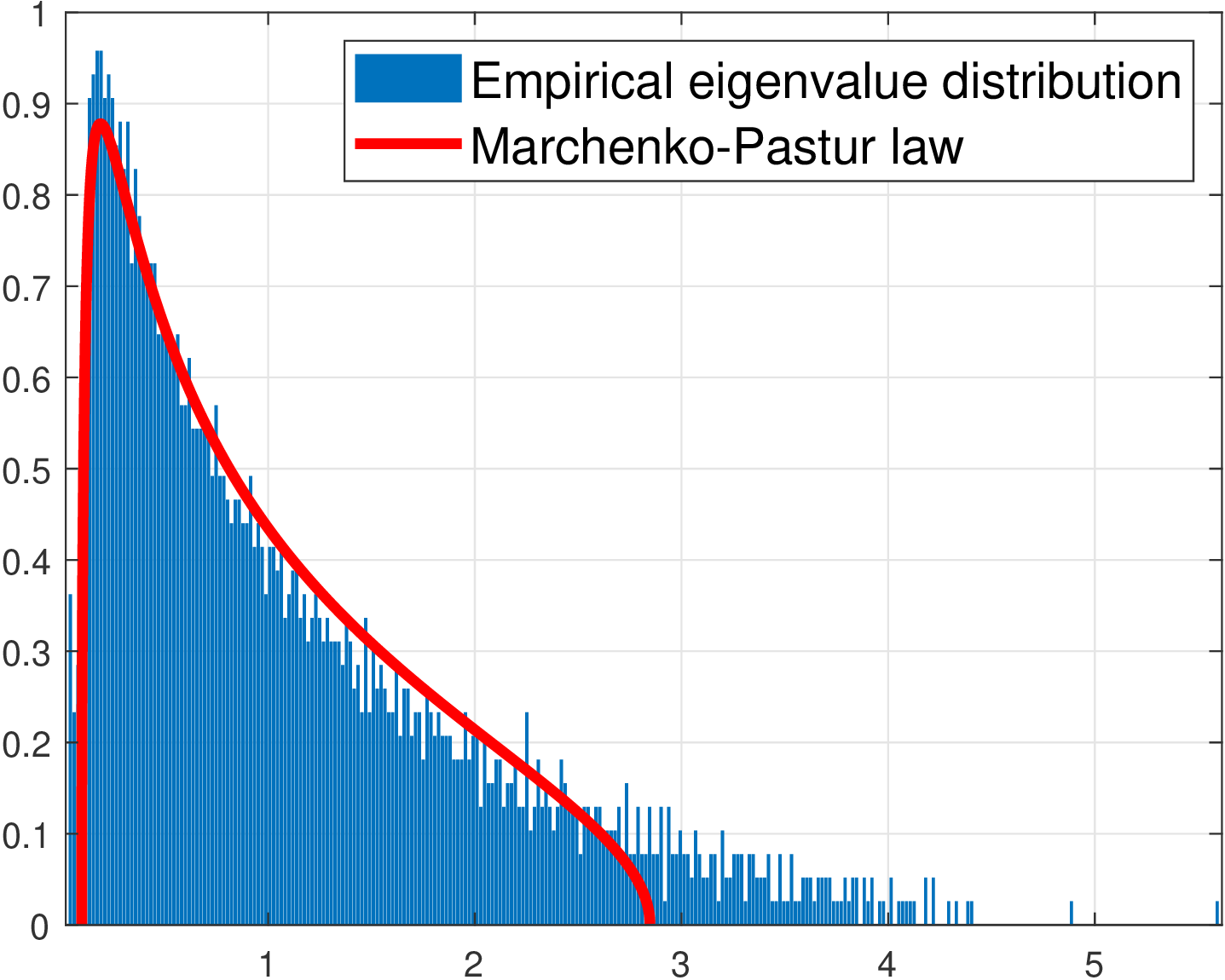}
        }
        
        \subfloat[][Hi-C: Poisson variance]{
        \includegraphics[width=0.25\textwidth]{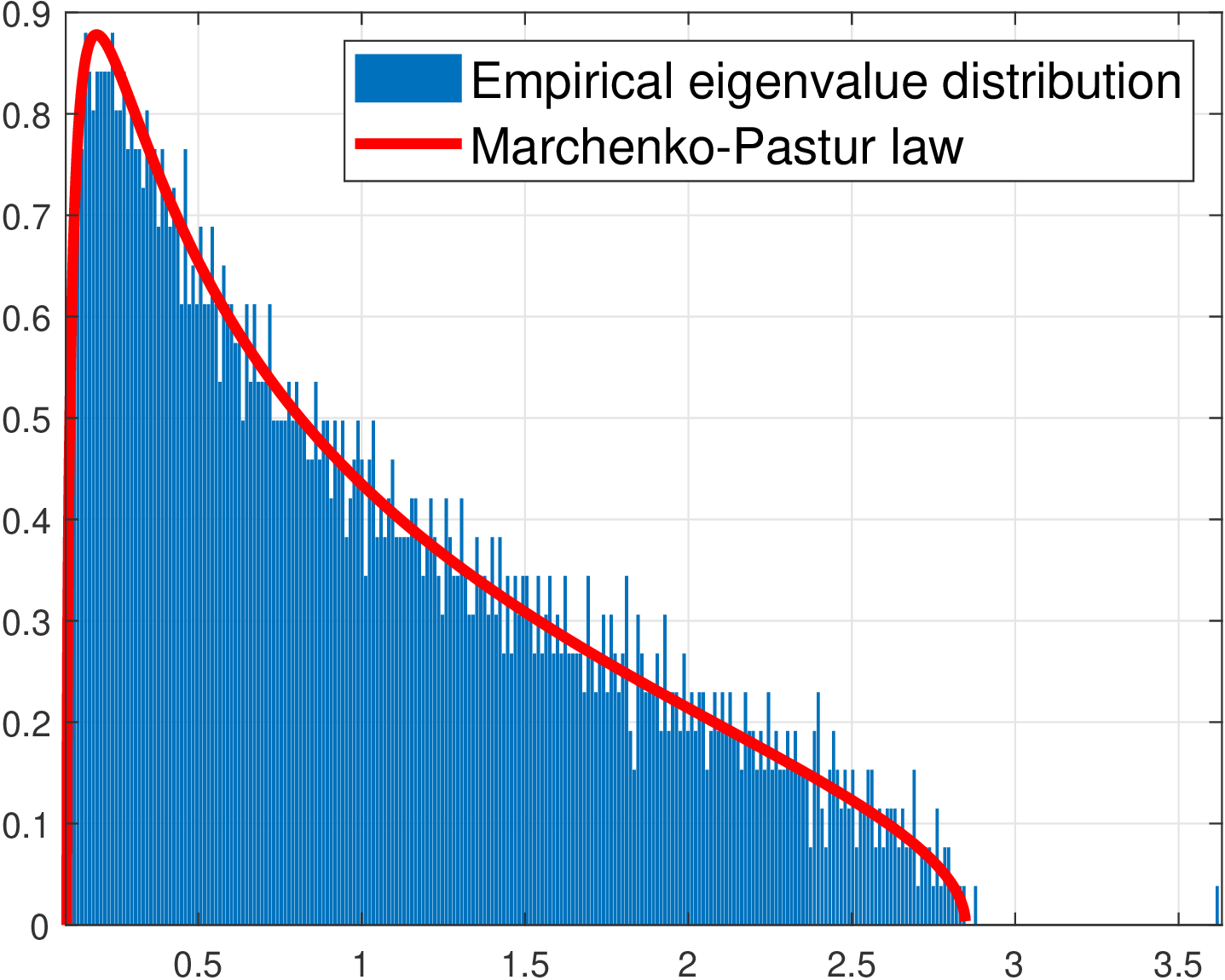}
        }
        
        \subfloat[][Hi-C: Quadratic variance]{
        \includegraphics[width=0.25\textwidth]{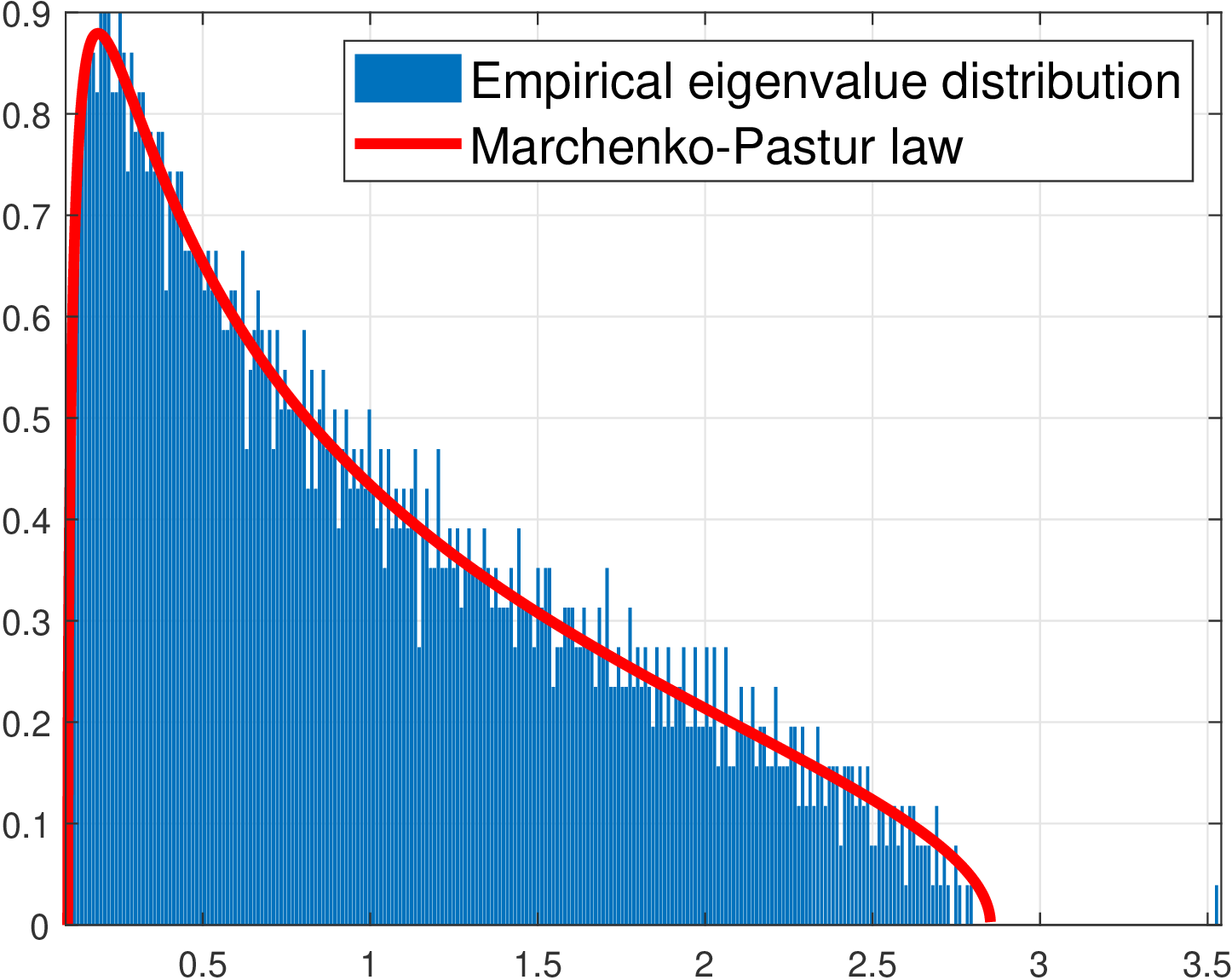}
        } 
    }
    \\
    \makebox[\textwidth][c]{
        \subfloat[][AP: constant variance]{
        \includegraphics[width=0.25\textwidth]{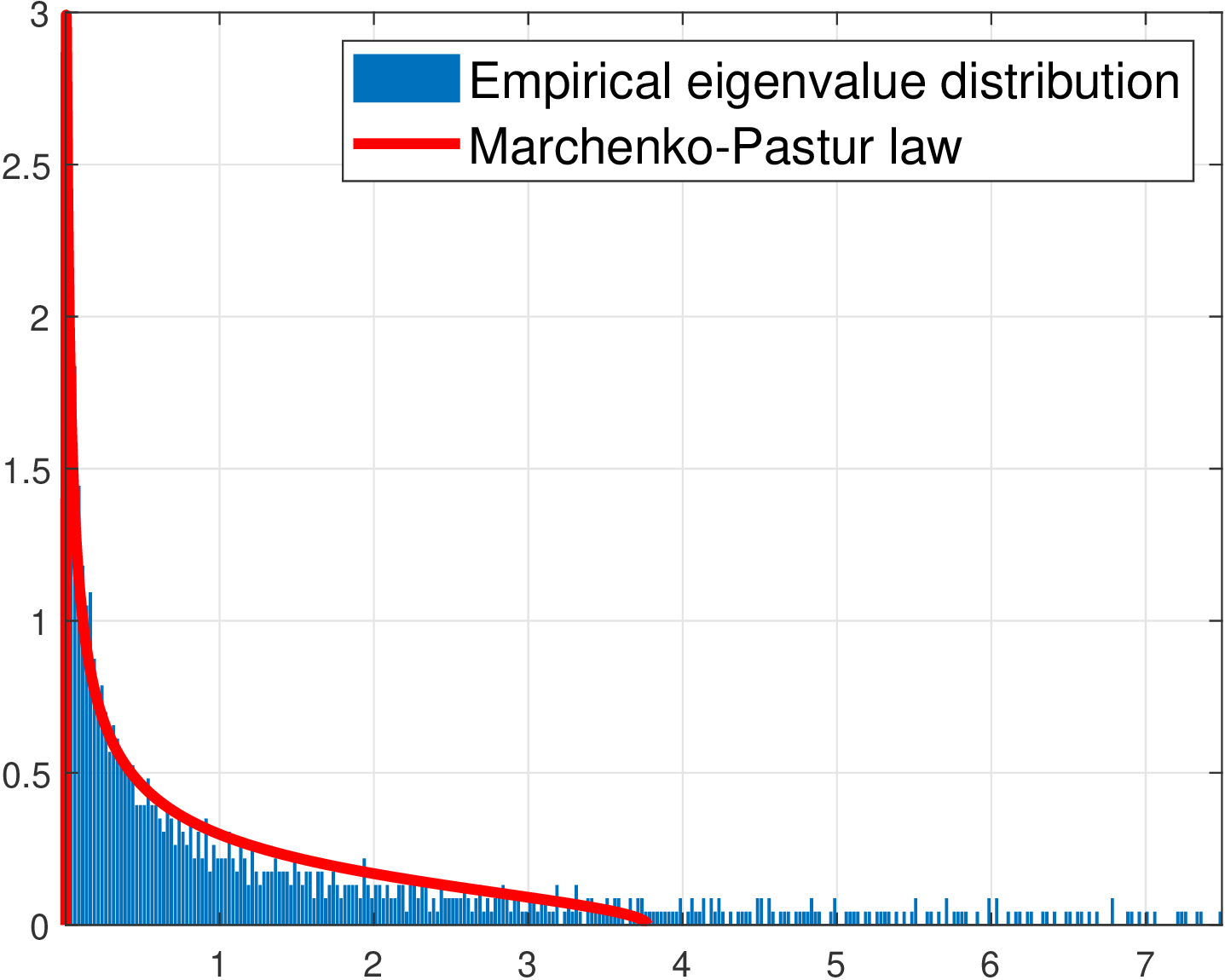}
        }
        
        \subfloat[][AP: Poisson variance]{
        \includegraphics[width=0.25\textwidth]{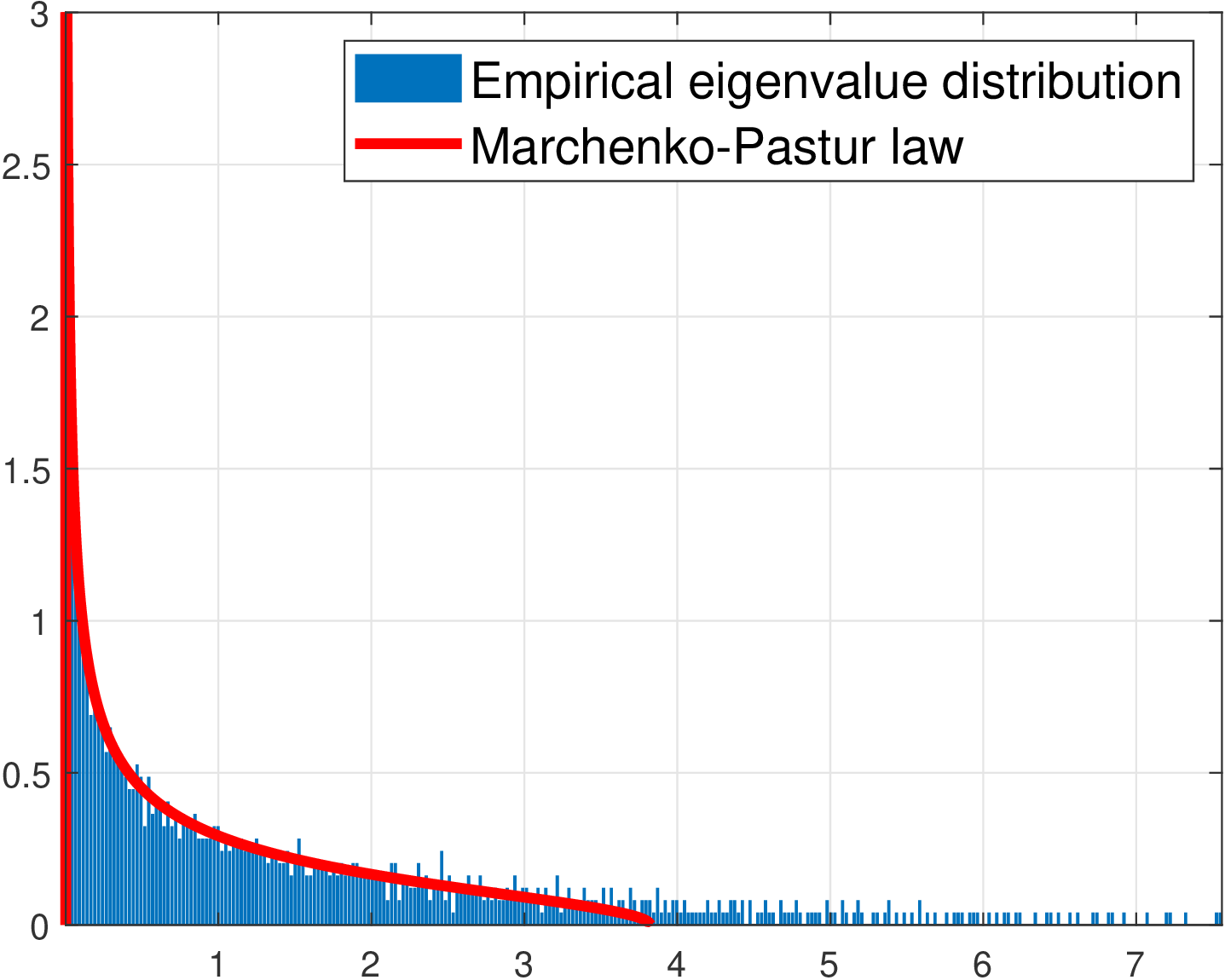}
        }
        
        \subfloat[][AP: quadratic variance]{
        \includegraphics[width=0.25\textwidth]{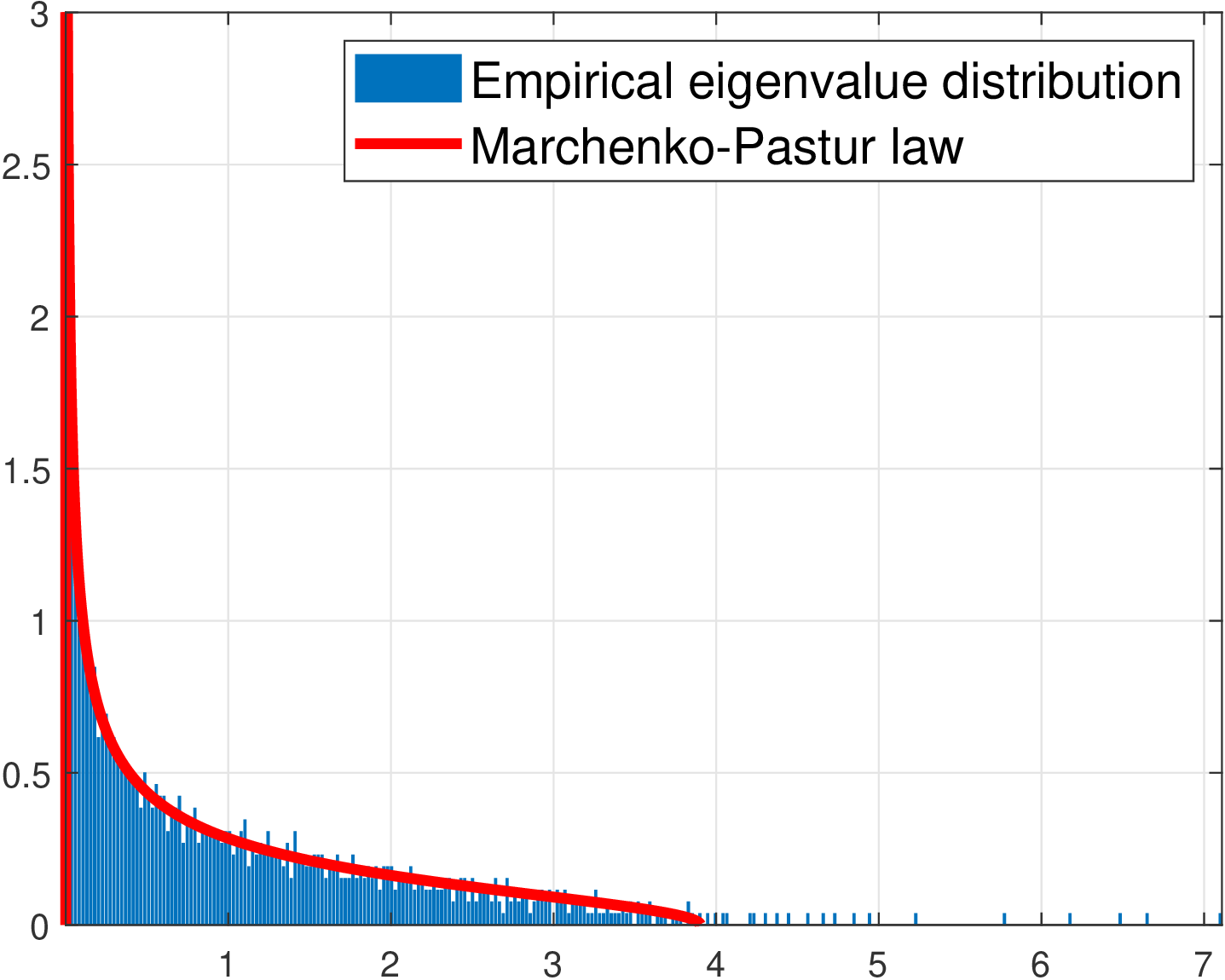}
        } 
    }
    \\
    \makebox[\textwidth][c]{
        \subfloat[][20 NG: constant variance]{
        \includegraphics[width=0.25\textwidth]{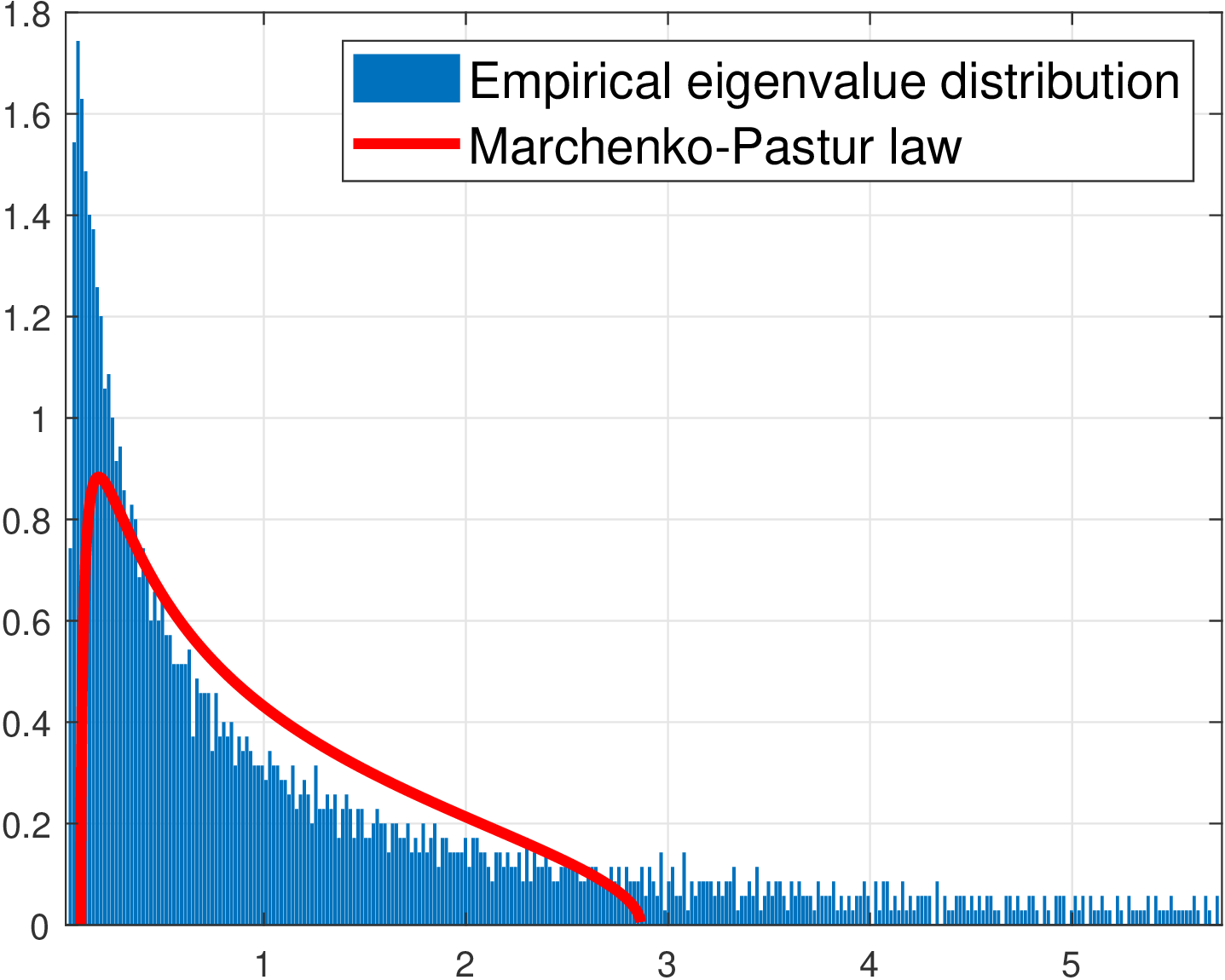}
        }
        
        \subfloat[][20 NG: Poisson variance]{
        \includegraphics[width=0.25\textwidth]{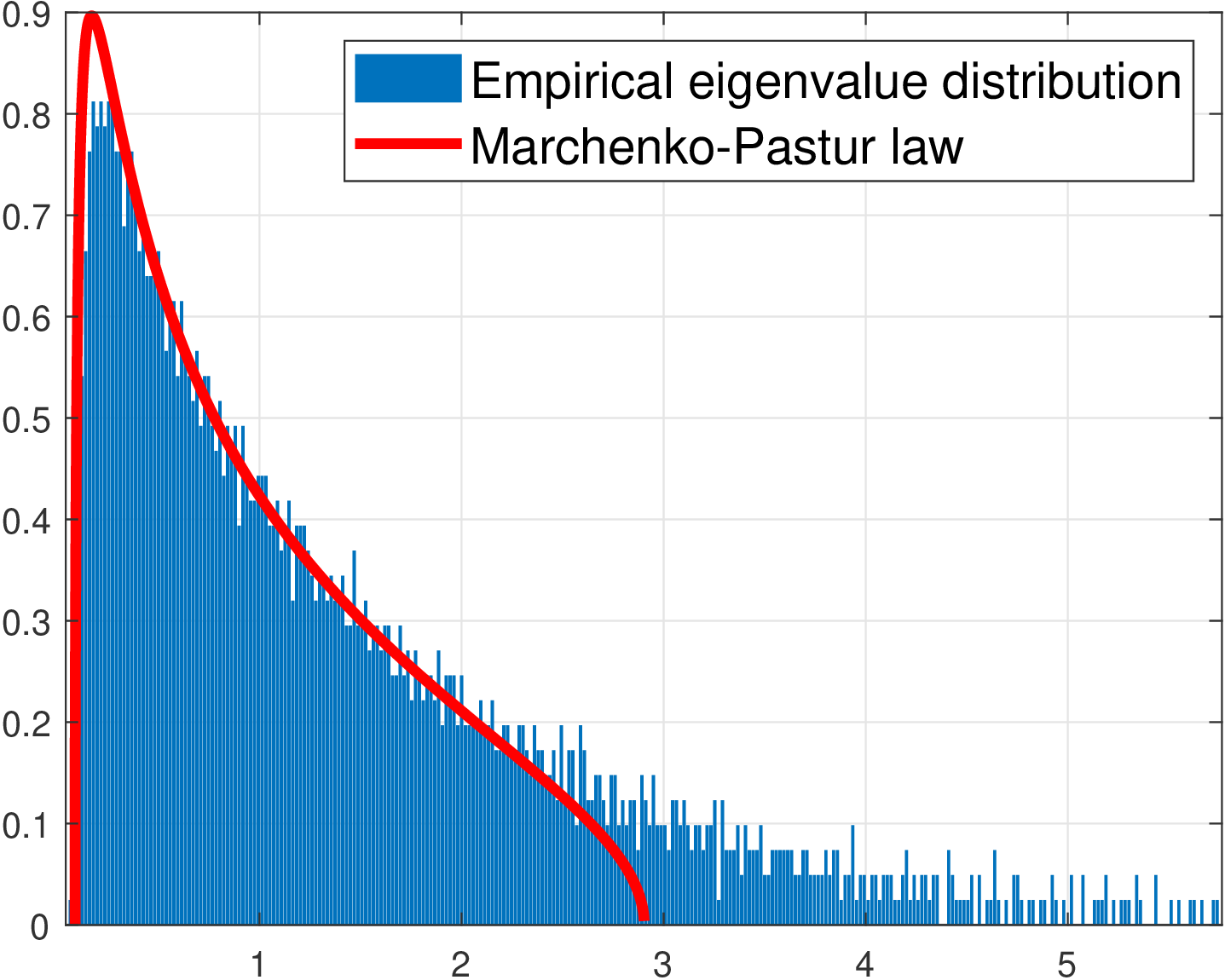}
        }
        
        \subfloat[][20 NG: quadratic variance]{
        \includegraphics[width=0.25\textwidth]{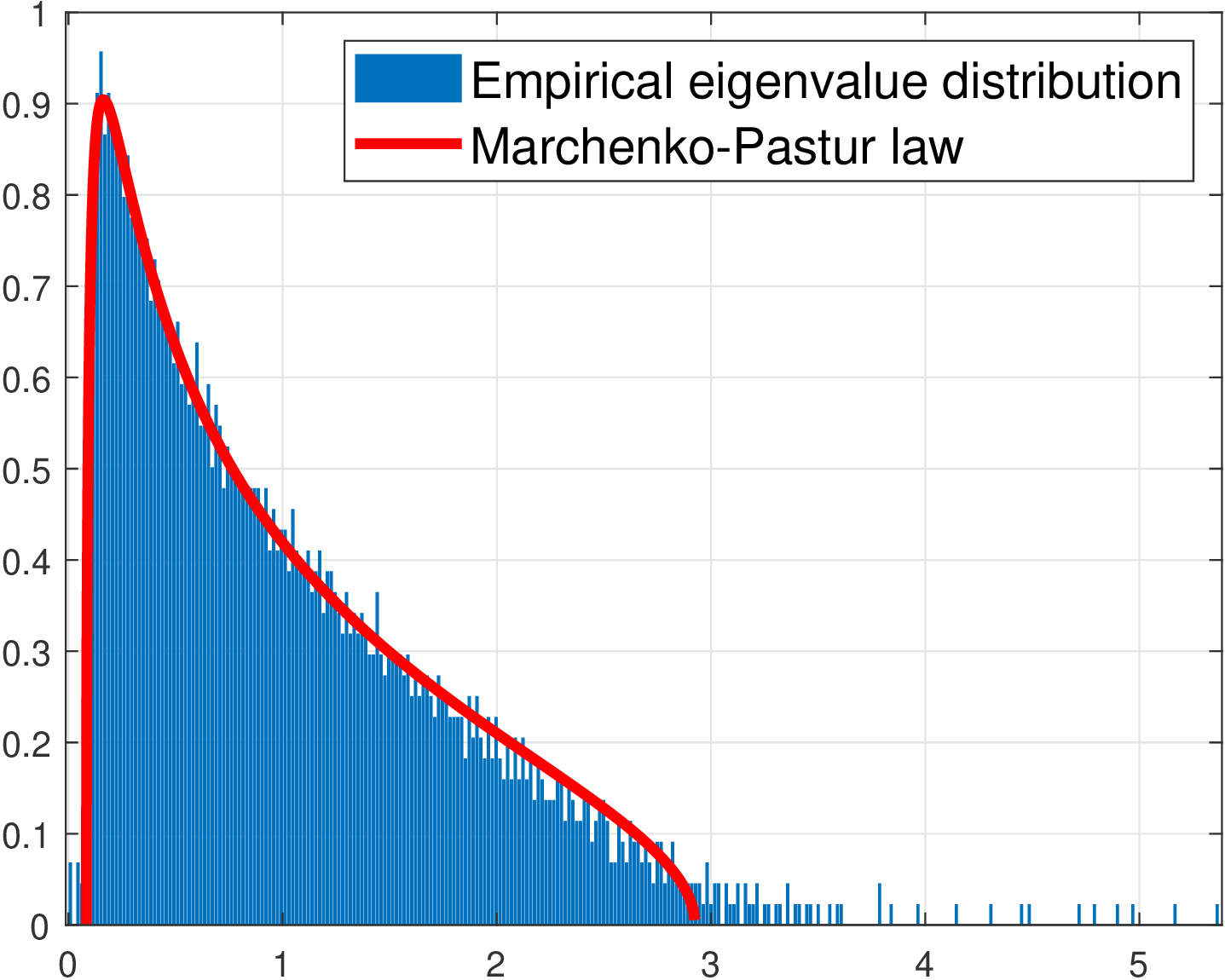}
        } 
    }
    
    \caption{Eigenvalue histograms obtained from several datasets (from top to bottom): Purified PBMC~\cite{zheng2017massively}, Hrvatin et al.~\cite{hrvatin2018single}, Hi-C~\cite{johanson2018genome}, Associated Press~\cite{harman1993first}, and 20 NewsGroups (20 NG)~\cite{Lang95}, versus the Marchenko-Pastur density $dF_{\gamma, 1}$, where $\gamma = m/n$. For each dataset, we used three variance estimators (from left to right) when solving~\eqref{eq:scaling equations from estimated variances general model}: constant variance ($\widehat{\operatorname{Var}}[Y_{i,j}] = \alpha$, where $\alpha$ is chosen by~\eqref{eq:estimator for alpha}), Poisson variance ($\widehat{\operatorname{Var}}[Y_{i,j}] = Y_{i,j}$), and adjusted quadratic variance ( Eq.~\eqref{eq:variance estimator with alpha and beta}, where $\alpha$ and $\beta$ are chosen as described in Section~\ref{sec:adapting to the data}). The eigenvalue histograms are shown for a held-out subset of the data matrix (whereas $\alpha$ and $\beta$ are determined from a disjoint subset of the data matrix).
    }
    \label{fig:MP fit all datasets}
\end{figure}

\begin{table}{
\caption{\noindent Goodness-of-fit between the eigenvalues of $n^{-1} \hat{Y} \hat{Y}^T$ and the MP law for several real datasets, where the scaling factors $(\hat{\mathbf{u}},\hat{\mathbf{v}})$ were obtained by solving~\eqref{eq:scaling equations from estimated variances general model} for three different variance estimators: constant, Poisson, and the quadratic variance~\eqref{eq:variance estimator with alpha and beta}. The entries in the table were obtained by averaging the results from $10$ randomized trials according to the sample splitting scheme described at the end of Section~\ref{sec:adapting to the data}}
\begin{center}
\begin{adjustbox}{max width = \textwidth}
\begin{tabular}{|l|l|l|l|l|l|l|}
\hline
\multirow{2}{*}{Dataset} & \multicolumn{2}{l|}{$\widehat{\operatorname{Var}}[Y_{i,j}]=\alpha$} & \multicolumn{2}{l|}{$\widehat{\operatorname{Var}}[Y_{i,j}]=Y_{i,j}$} & \multicolumn{2}{l|}{$\widehat{\operatorname{Var}}[Y_{i,j}]=\alpha[ (1-\beta) Y_{i,j} + \beta Y_{i,j}^2 ]$} \\ \cline{2-7} 
                         & KS distance                      & p-value                          & KS distance                       & p-value                          & KS distance                                                 & p-value                                                 \\ \hline
PBMC                     & $0.21$                           & $10^{-211}$                      & $0.03$                            & $10^{-6}$                        & $0.007$                                                     & $0.95$                                                  \\ \hline
Hrvatin et al.           & $0.24$                           & $10^{-250}$                      & $0.17$                            & $10^{-131}$                      & $0.01$                                                      & $0.26$                                                  \\ \hline
Hi-C                     & $0.07$                           & $10^{-10}$                       & $0.01$                            & $0.85$                           & $0.005$                                                     & $0.99$                                                  \\ \hline
AP                       & $0.18$                           & $10^{-31}$                       & $0.14$                            & $10^{-20}$                           & $0.02$                                                      & $0.37$                                                  \\ \hline
20 NewsGroups             & $0.2$                           & $10^{-103}$                       & $0.15$                            & $10^{-47}$                           & $0.02$                                                     & $0.1$                                                   \\ \hline
\end{tabular}
\end{adjustbox}

\end{center}
\label{tab:fit to MP law for real data}
}
\end{table}

\subsection{Rank estimation on annotated data} \label{sec:rank estmation real data}
To test the accuracy of our rank estimation method on real data, we used the class labels available for the cells in the PBMC dataset and the labels for the documents in the 20 NewsGroups dataset. We first randomly selected $500$ observations from each one of several classes, and then filtered the resulting matrices for sparsity. For the PBMC dataset we used 8 classes out of 10, and for the 20 NewsGroups we used 10 classes out of 20, choosing classes that should be well distinguishable; see more details in Appendix~\ref{appendix:reproducibility details for real data rank estimation}. Since the number of classes is generally not equal to the rank of the signal matrix but is only a lower bound (assuming that the classes correspond to subspaces that are linearly independent), we further performed the following  ``homogenization'' procedure. We randomly permuted the entries in each feature (a gene in the PBMC dataset and a word in the 20 NewsGroups dataset) across all observations in a class, for each class and each feature independently. This homogenization destroys the correlations that exist across observations or features within a class, so the resulting underlying signal matrix has as many large eigenvalues as the number of classes and the rest of the eigenvalues should be very small. While this homogenization undoubtedly removes information from the data (the within-class structure), it allows us to use the number of classes as a surrogate for the rank while preserving important characteristics of the data, such as the distribution of values within each feature and each class. 

Figure~\ref{fig:PBMC+20NG homogenized mat scree plot} illustrates the sorted singular values of the resulting homogenized matrices before and after biwhitening, where we used our adaptive version described in Section~\ref{sec:adapting to the data} to choose $\alpha>0$ and $\beta\in \{0,0.05,\ldots,1\}$ automatically. It is evident that the signal singular values are more easily detectable after biwhitening and emerge above the MP upper edge (noting that the $8$th signal singular value for the PBMC dataset is slightly above the MP upper edge but too close to it to be clearly visible). 
We repeated our data preprocessing and homogenization procedure for 10 randomized trials, each time estimating the rank by applying our method as well as the other methods described in Section~\ref{section:rank estimation}. Table~\ref{tab:PBMC+20NG homogenized mat rank selection methods} summarizes the average estimated ranks and their standard deviations. Overall, it is evident that our method provides the most accurate rank estimates for both datasets, while other methods consistently overselect or underselect the rank.

\begin{figure}[tbhp]
    \centering
    \makebox[\textwidth][c]{
        \subfloat[][Homogenized PBMC, original counts ]{
        \includegraphics[width=0.50\textwidth]{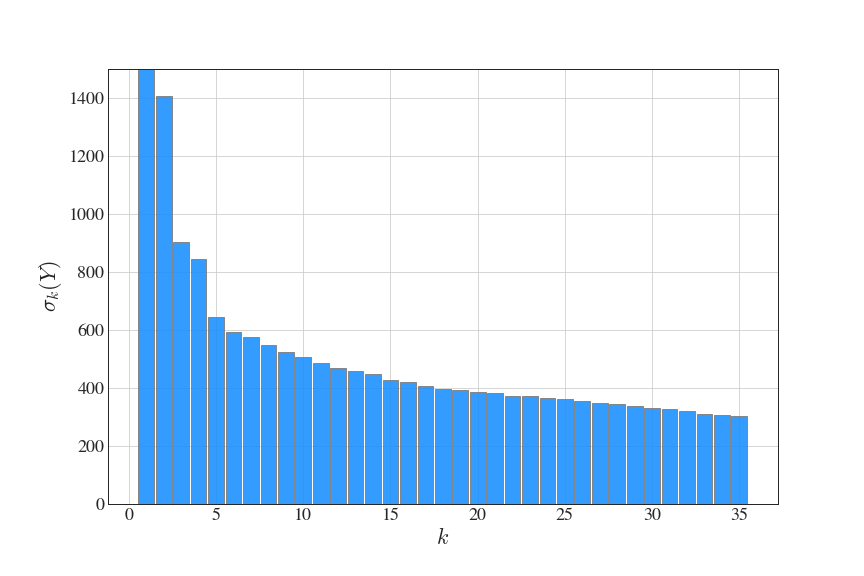}
        }
        
        \subfloat[][Homogenized PBMC, after biwhitening]{
        \includegraphics[width=0.50\textwidth]{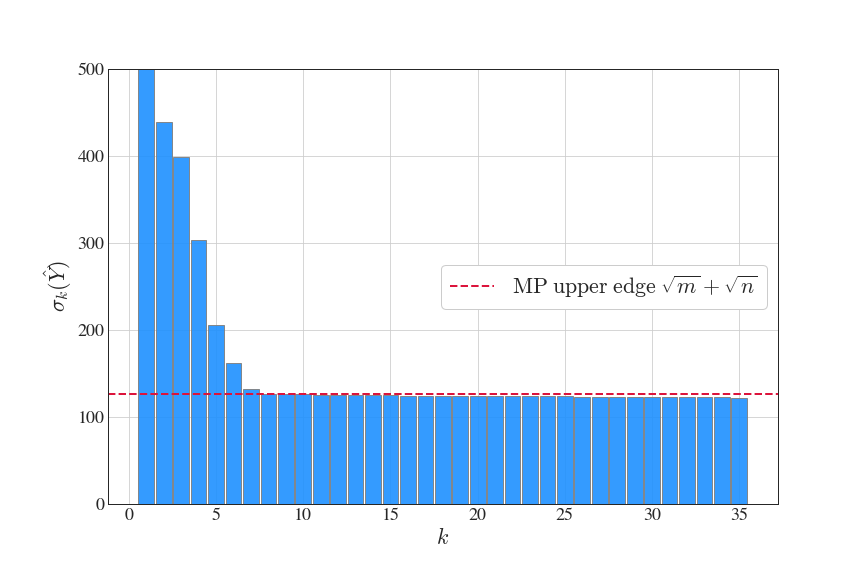}
        }
        
    }
    
    \makebox[\textwidth][c]{
        \subfloat[][Homogenized 20 NewsGroups, original counts]{
        \includegraphics[width=0.50\textwidth]{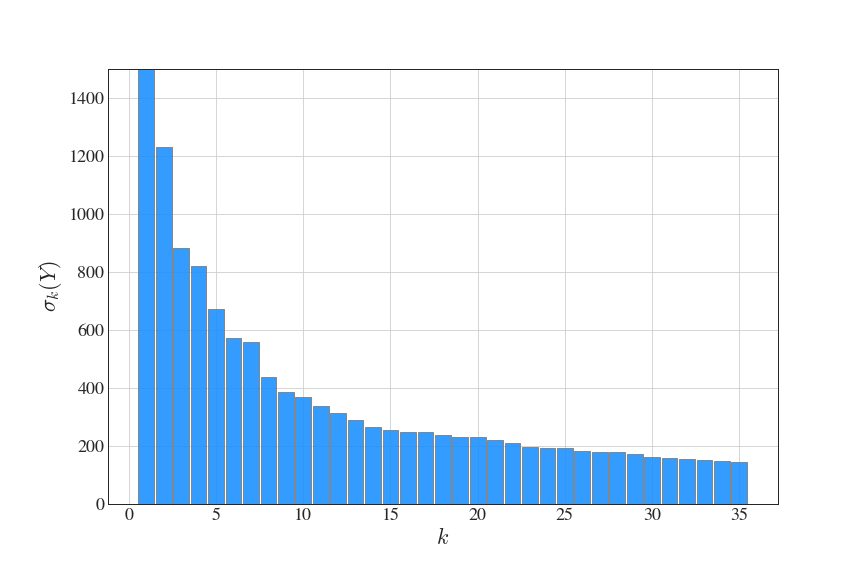}
        }
        
        \subfloat[][Homogenized 20 NewsGroups, after biwhitening]{
        \includegraphics[width=0.50\textwidth]{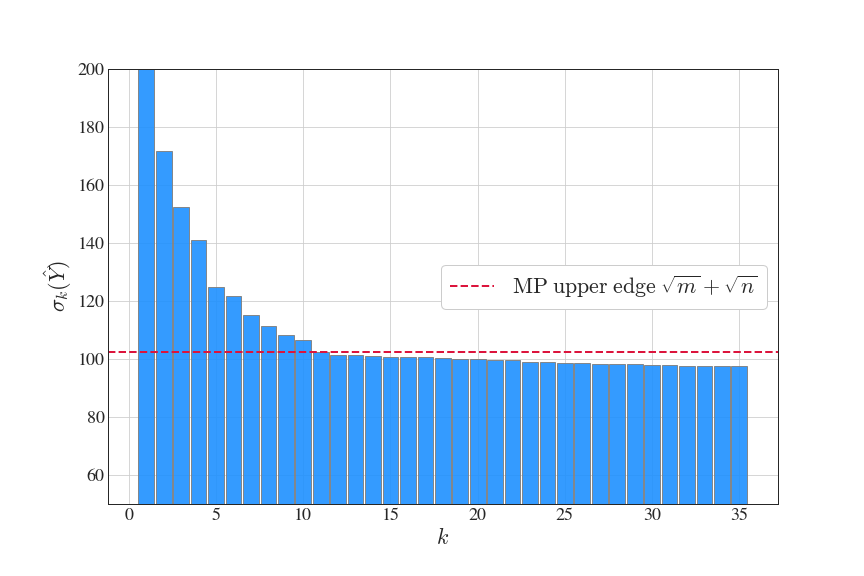}
        }
    }
    
    \caption{Sorted singular values of the homogenized matrices for the PBMC dataset (top) and the 20 NewsGroups dataset (bottom) before (left) and after (right) biwhitening. For the PBMC dataset we used 8 classes (cell typpes), and for the 20 NewsGroups we used 10 classes (news groups), hence the expected ranks are $8$ and $10$, respectively (see the details of homogenization). To perform biwhitening we used the procedure described in Section~\ref{sec:adapting to the data} to find $\alpha>0$ and $\beta\in\{0,0.05,\ldots,1\}$ adaptively.}
    \label{fig:PBMC+20NG homogenized mat scree plot}
\end{figure}

\begin{table}
{
\caption{\noindent Mean (standard deviation) of ranks selected by several methods (see Section~\ref{section:rank estimation}) for two real datasets with ground truth labels. The datasets were preprocessed and homogenized (observations with the same label were randomly permuted along each feature independently) over 10 randomized trials.}
\begin{center}
\begin{adjustbox}{max width = \textwidth}
{    \begin{tabular}{|l| *{6}{c|}   }
        \hline
         Dataset & $r$ & Our method & EKC & pairPA & DDPA+ & Signflip PA   \\ \hline
        PBMC & 8 & 8.4 (0.7) & 21.7 (3.1) & 69.1 (1.9) & 100.5 (68.6) & 5 (0.0) \\ \hline
        20 NewsGroups & 10 & 10 (0.0) & 40.2 (11.2) & 9.6 (1.4) & 3.1 (0.7) & 14.2 (0.9) \\ \hline
    \end{tabular}
}
\end{adjustbox}

    \label{tab:PBMC+20NG homogenized mat rank selection methods}
\end{center}
}
\end{table}

\section{Discussion}
Our biwhitening procedure for rank estimation has several important advantages over alternative methods. First and foremost, it can handle almost any pattern of entries in $X$, including those that lead to severe noise heteroskedasticity. In particular, and as suggested by our simulations in Section~\ref{section:rank estimation}, our method stabilizes the noise variances across rows and columns and prevents extreme rows or columns from dominating the spectrum of the noise, allowing our method to detect weak signal components that otherwise would be masked by the noise. Second, our method enforces the largest noise eigenvalue to admit a simple analytic expression -- the MP upper edge. This property obviates the need for estimating the largest noise eigenvalue by Monte Carlo simulations (such as permutations and signflips of the data), which are sensitive to the structure and magnitude of the unknown signal matrix. Lastly, our approach provides a simultaneous validation of our model assumptions through the fit of the resulting spectrum to the MP law. Such validation is an invaluable tool for exploratory data analysis, where ground truth information is seldom available. 

Since this work is concerned with rank estimation, it is worthwhile to discuss the closely related task of recovering the principal components. In~\cite{hong2018asymptotic} it was shown that the performance of standard PCA can significantly degrade  under heteroskedastic noise, even if the noise varies only along one dimension of the matrix. Therefore, while our approach is able to accurately detect informative signal components in heteroskedastic noise, it may be suboptimal to plug our estimated rank directly into standard PCA when the noise is strongly heteroskedastic. In such cases, one possibility is to apply our method in conjunction with recently proposed methods for PCA and matrix denoising under heteroskedastic noise; see e.g.,~\cite{zhang2018heteroskedastic} (or special cases such as Poisson noise~\cite{cao2015poisson}), which require knowledge of the rank. Another possibility is to apply standard PCA after biwhitening, which is particularly appealing since biwhitening stabilizes the average noise variances across rows and columns, alleviating much of the effect of heteroskedastic noise. However, the scaling of rows and columns introduces a bias into the principal components, modifying them in a nontrivial way. In certain applications this may be acceptable, and applying PCA after biwhitening can be favorable if more principal components are detected and utilized for subsequent analysis. In other applications, where interpretability of the principal components is important, the bias introduced by the scaling may need to be corrected. This topic is a promising future research direction but is beyond the scope of this paper and is left for future work. 

\section{Acknowledgements}
The authors would like to thank Edgar Dobriban, George Linderman, Jay Stanley, and Xiaoou Li for useful and insightful discussions. B.L., T.Z., and Y.K. acknowledge support by NIH grant R01GM131642. B.L. and Y.K. also acknowledge support by NIH grants UM1DA051410, U01DA053628, and U54AG076043. Y.K. acknowledges support by NIH grants R01GM135928 and 2P50CA1219.

\begin{appendices}

\section{Existence and uniqueness of $\mathbf{x}$ and $\mathbf{y}$ in~\eqref{eq:matrix scaling general equations} for $r_i=n$ and $c_j=m$} \label{appendix: existence and uniqueness}
We begin with the following definition. 
\begin{defn}[Completely decomposable matrix]
We say that a nonnegative matrix $A\in\mathbb{R}^{m\times n}$ is \textit{completely decomposable} if there exist proper nonempty subsets $\mathcal{I}_1\subset [m]$ and $\mathcal{I}_2\subset [n]$ such that $[A]_{i\in \mathcal{I}_1, \; j \in \mathcal{I}_2}$ and $[A]_{i\in \mathcal{I}_1^c, \; j \in \mathcal{I}_2^c}$ are both zero matrices, where $[A]_{i\in \mathcal{I}_1, \; j \in \mathcal{I}_2}$ is the submatrix of $A$ obtained by taking its rows in $\mathcal{I}_1$ and columns in $\mathcal{I}_2$, and $\mathcal{I}_1^c$ and $\mathcal{I}_2^c$ are the complements of  $\mathcal{I}_1$ and $\mathcal{I}_2$ in $[m]$ and $[n]$, respectively. In other words, $A$ is completely decomposable if there exist permutation matrices $P\in\mathbb{R}^{m\times m}$ and $Q\in\mathbb{R}^{n\times n}$ such that
\begin{equation}
    P A Q = 
    \begin{bmatrix}
    B_{1,1} & \mathbf{0}_{|\mathcal{I}_1|\times |\mathcal{I}_2|} \\
    \mathbf{0}_{|\mathcal{I}_1^c|\times |\mathcal{I}_2^c|} & B_{2,2}
    \end{bmatrix}, \label{eq:completeley decomposable matrix}
\end{equation}
where $\mathbf{0}_{d_1\times d_2}$ is a matrix of zeros of size $d_1\times d_2$. We say that $A$ is \textit{not completely decomposable} if $P$ and $Q$ such as in~\eqref{eq:completeley decomposable matrix} do not exist. 
\end{defn}
A useful equivalent characterization of a completely decomposable matrix can be obtained by inspecting the connectivity of the bipartite graph described by $A$. Specifically, consider the undirected and unweighted bipartite graph $\mathcal{G}_A(\mathcal{U}, \mathcal{V},E)$ whose nodes $\mathcal{U}$ correspond to the rows of $A$, nodes $\mathcal{V}$ correspond to the columns of $A$, and edges $E$ between $\mathcal{U}$ and $\mathcal{V}$ correspond to the nonzero entries of $A$, i.e., the $i$'th node in $\mathcal{U}$ is connected to the $j$'th node in $\mathcal{V}$ if and only if $A_{i,j} > 0$. Then, it immediately follows that $A$ is completely decomposable if and only if the graph $\mathcal{G}_A$ is disconnected, and furthermore, the connected components of $\mathcal{G}_A$ correspond to the blocks $B_{1,1}$ and $B_{2,2}$ in~\eqref{eq:completeley decomposable matrix} after permuting the nodes $\mathcal{U}$ and $\mathcal{V}$ according to $P$ and $Q$ from~\eqref{eq:completeley decomposable matrix}, respectively.

We now consider two important zero patterns of $A$.
First, it is clear that if $A$ admits any zero row or column, then $\mathbf{x}$ and $\mathbf{y}$ that solve~\eqref{eq:matrix scaling general equations} cannot exist. Hence, any zero rows or columns in a realization of $Y$ must first be removed before attempting to solve~\eqref{eq:scaling equations from estimated variances}. 
Second, observe that according to~\eqref{eq:completeley decomposable matrix}, if $A$ is completely decomposable and does not have any zero rows or columns, then after a certain permutation of its rows and columns it can be written as a direct sum of smaller nonnegative matrices that are not completely decomposable. In other words, if each row and column of $A$ has at least one positive entry, then there exist permutation matrices $P$ and $Q$ such that 
\begin{equation}
    P A Q = \bigoplus_{k=1}^K B_{k,k} = 
    \begin{bmatrix}
    B_{1,1} & \mathbf{0} & \mathbf{0} & \mathbf{0} \\
    \mathbf{0} & B_{2,2} & \mathbf{0} & \mathbf{0} \\
    \mathbf{0} & \mathbf{0} & \ddots & \mathbf{0} \\
    \mathbf{0} & \mathbf{0} & \mathbf{0} & B_{K,K}
    \end{bmatrix}, \label{eq:completely decomposable direct sum}
\end{equation}
where $\bigoplus$ is the direct sum operation, $B_{k,k}\in \mathbb{R}^{d_1^{(k)}\times d_2^{(k)}}$ are nonnegative matrices that are not completely decomposable, and $\mathbf{0}$ represents a block of zeros of appropriate size (not necessarily square). Importantly, since permuting rows and columns does not change their sums, the task of scaling $A$ to row sums $n$ and column sums $m$ is equivalent to that of scaling each of the matrices $B_{k,k}$ to these row and column sums. However, scaling the matrix $B_{k,k} \in \mathbb{R}^{d_1^{(k)}\times d_2^{(k)}}$ to row sums $n$ and column sums $m$ is possible only if 
\begin{equation}
    {d_1^{(k)} n } = d_2^{(k)} m, \label{eq:aspect ratio necessary condition}
\end{equation}
since the sum of all row sums is equal to the sum of all the entries in the matrix and must be the same as the sum of all column sums. Equation~\eqref{eq:aspect ratio necessary condition} implies that the aspect ratios (i.e., the number of columns divided by the number of rows) of each of the blocks $\{B_{k,k}\}$ must be exactly the same as the aspect ratio of matrix $A$, which is clearly a restrictive requirement.

To circumvent the above-mentioned issue, observe that whenever the realization of $Y$ is completely decomposable, the singular value decomposition of $Y$ can be written explicitly using the singular value decompositions of the blocks $\{B_{k,k}\}$ from the decomposition~\eqref{eq:completely decomposable direct sum} of $Y$ (replacing $A$). In particular, the singular values of $Y$ are given by concatenating the singular values of each of the blocks $B_{k,k}$. This suggests that one should treat each block $B_{k,k}$ separately as a $d_1^{k}\times d_2^{k}$ matrix, and scale it accordingly to row sums $d_2^{k}$ and column sums $d_1^{k}$. Correspondingly, the theory in Section~\ref{sec:method and results} would apply to each block $B_{k,k}$ separately. 
Therefore, for a given realization of $Y$, we propose to first remove its zero rows and columns, and to find its blocks $B_{k,k}$ in the decomposition~\eqref{eq:completely decomposable direct sum} by finding the connected components in the bipartite graph represented by $Y$. We then treat each block $B_{k,k}$ as a $d_1^{k}\times d_2^{k}$ matrix that should be scaled to row sums $d_2^{k}$ and column sums $d_1^{k}$ (instead of $n$ and $m$, respectively), and the rank for each  block $B_{k,k}$ should be chosen separately according to Section~\ref{sec:method and results}. 
Note that each $B_{k,k}$ does not have any zero rows and columns and is not completely deomposable, hence in what follows we proceed by treating the case where $A$ assumes the same properties.

For a matrix $A$ that does not have any zero rows and columns and is not completely decomposable, the following is the precise requirement from $A$ so that positive $\mathbf{x}$ and $\mathbf{y}$ that satisfy~\eqref{eq:matrix scaling general equations} for $r_i = n$ and $c_j = m$ exist and are unique.
\begin{cond} \label{cond:existence of x and y for nonegative A}
For all non-empty subsets $\mathcal{I}_1\subset [m]$ and $\mathcal{I}_2\subset [n]$ for which $[A]_{i\in \mathcal{I}_1, \; j \in \mathcal{I}_2}$ is a matrix of zeros, $|\mathcal{I}_1| n + | \mathcal{I}_2 | m < m n$.
\end{cond}
We then have the following proposition, which characterizes the existence and uniqueness of positive $\mathbf{x}$ and $\mathbf{y}$ that satisfy~\eqref{eq:matrix scaling general equations} for the case of $r_i=n$ and $c_j=m$.
\begin{prop}[Existence and uniqueness of $\mathbf{x}$ and $\mathbf{y}$ for $r_i=n$, $c_j=m$] \label{prop:existcne and uniquness for A}
Let $r_i=n$ and $c_j=m$ for all $i\in [m]$ and $j\in [n]$ in~\eqref{eq:matrix scaling general equations}, and suppose that $A$ does not have any zero rows and columns and is not completely decomposable. Then, there exists a pair $(\mathbf{x},\mathbf{y})$ of positive vectors that satisfies~\eqref{eq:matrix scaling general equations} if and only if Condition~\ref{cond:existence of x and y for nonegative A} holds. If such a pair $(\mathbf{x},\mathbf{y})$ exists, it is unique up to a positive scalar, namely it can only be replaced with $(a\mathbf{x},a^{-1}\mathbf{y})$ for any $a>0$.
\end{prop}
\begin{proof}
Proposition~\ref{prop:existcne and uniquness for A} follows directly from Theorem 1 in~\cite{brualdi1974dad}. Specifically, our Condition~\ref{cond:existence of x and y for nonegative A} is a equivalent to condition (1) in~\cite{brualdi1974dad} for the case of $r_i = n$, $c_j=m$, and under the assumption that $A$ is not completely decomposable. 
\end{proof}
Note that Condition~\ref{cond:existence of x and y for nonegative A} does not hold if $A$ includes any zero submatrix whose number of rows and columns exceed $\lfloor m/2 \rfloor$ and $\lfloor n/2 \rfloor$, respectively. 
In more generality, we anticipate that positive $\mathbf{x}$ and $\mathbf{y}$ that satisfy~\eqref{eq:matrix scaling general equations} for the case of $r_i=n$ and $c_j=m$ might not exist if $A$ is too sparse.
Since Condition~\ref{cond:existence of x and y for nonegative A} as stated is somewhat obscure and is non-trivially verified from a given matrix $A$, it is worthwhile to provide a simpler condition only in terms of the number of zeros in the rows and columns of $A$. This is the purpose of the following proposition, which describes a sufficient condition for a matrix $A$ to simultaneously satisfy Condition~\ref{cond:existence of x and y for nonegative A} and not be completely decomposable.
\begin{prop} \label{prop:simple condtion for existence and uniquness for A}
Suppose that $A$ has no zero rows or columns and both requirements below are met:
\begin{enumerate}
    \item For each $k \leq \lfloor {n}/{2} \rfloor$, $A$ has less than $\lceil m k/n \rceil$ rows that have at least $n-k$ zeros each. \label{Cond: cond 2 in prop}
    \item For each $\ell \leq \lfloor {m}/{2} \rfloor$, $A$ has less than $\lceil n \ell /m \rceil$ columns that have at least $m-\ell$ zeros each. \label{Cond: cond 3 in prop}
\end{enumerate}
Then, Condition~\ref{cond:existence of x and y for nonegative A} holds and $A$ is not completely decomposable.
\end{prop}
The proof of Proposition~\ref{prop:simple condtion for existence and uniquness for A} can be found in Appendix~\ref{appendix:proof of simple condition for existence and uniqueness for A}.

Importantly, the conditions in Proposition~\ref{prop:simple condtion for existence and uniquness for A} can be easily verified for any given matrix $A$ by counting the number zeros in each row and each column.
In case that a given matrix $A$ does not satisfy these conditions, it can be modified by removing its sparsest rows and columns until these conditions are met. 
In particular, one can check if the rows (columns) of $A$ are in violation of the requirements in Proposition~\ref{prop:simple condtion for existence and uniquness for A}, and if so remove the sparsest row (column) of $A$, repeating the process until no violations are found.

\section{Fit against the MP law for Poisson noise with $\gamma = 1/3,1/4$} \label{appendix:MP fit Poisson additional figures}
In Figure~\eqref{fig:MPfit2_extra_aspect_ratios} we depict the results of the experiment described in Section~\ref{sec:experiments - fit to MP law} for the aspect ratios $\gamma = m/n = 1/3,\;1/4$.
\begin{figure}[] 
    \makebox[\textwidth][c]{
        \hspace{-0.5cm}
        \subfloat[][$\gamma = 1/3$, $n = 100,\, 500,\, 5000$]{
        \includegraphics[width=0.6\textwidth]{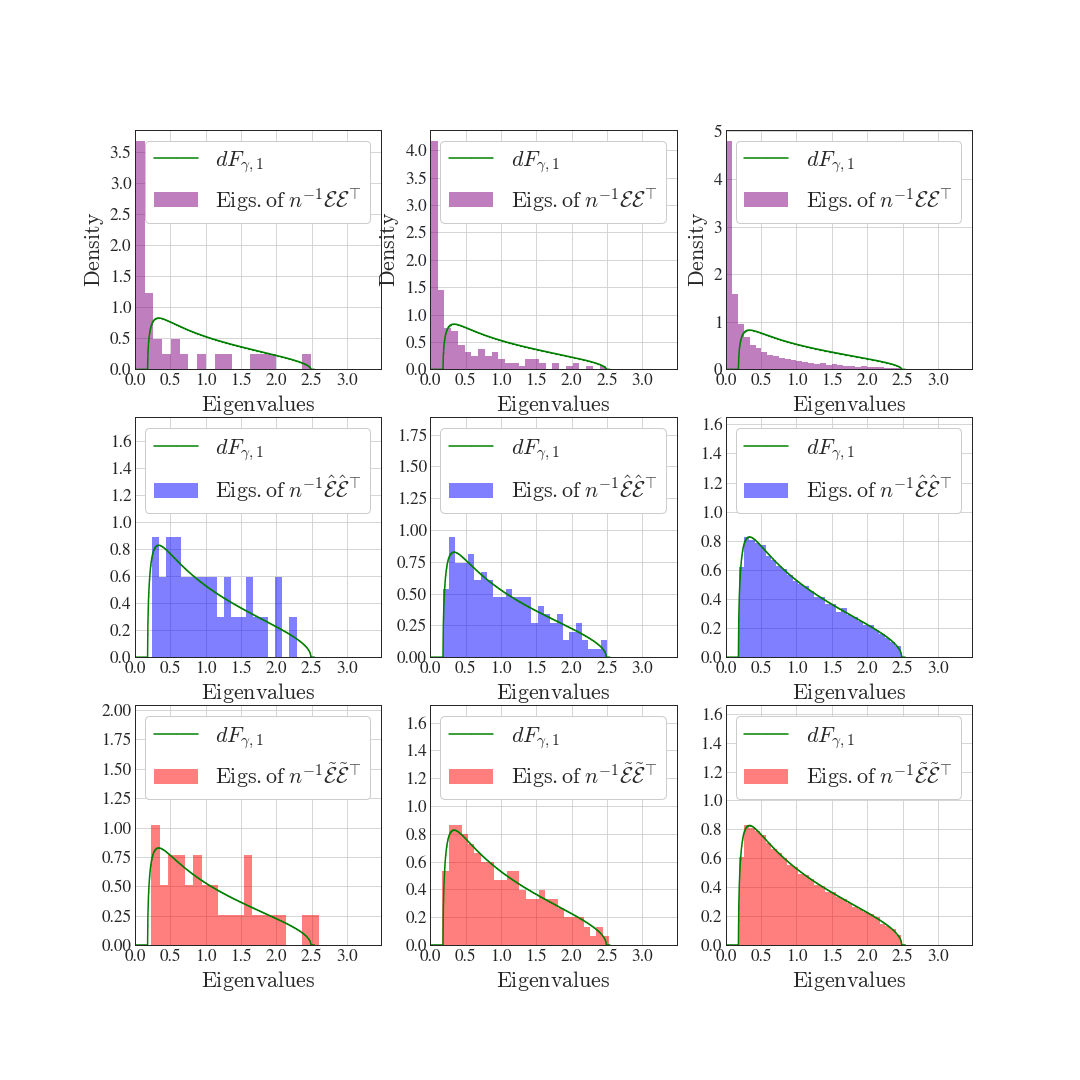}
        }
        \hspace{-1.3cm}
        \subfloat[][$\gamma = 1/4$, $n = 100,\, 500,\, 5000$]{
        \includegraphics[width=0.6\textwidth]{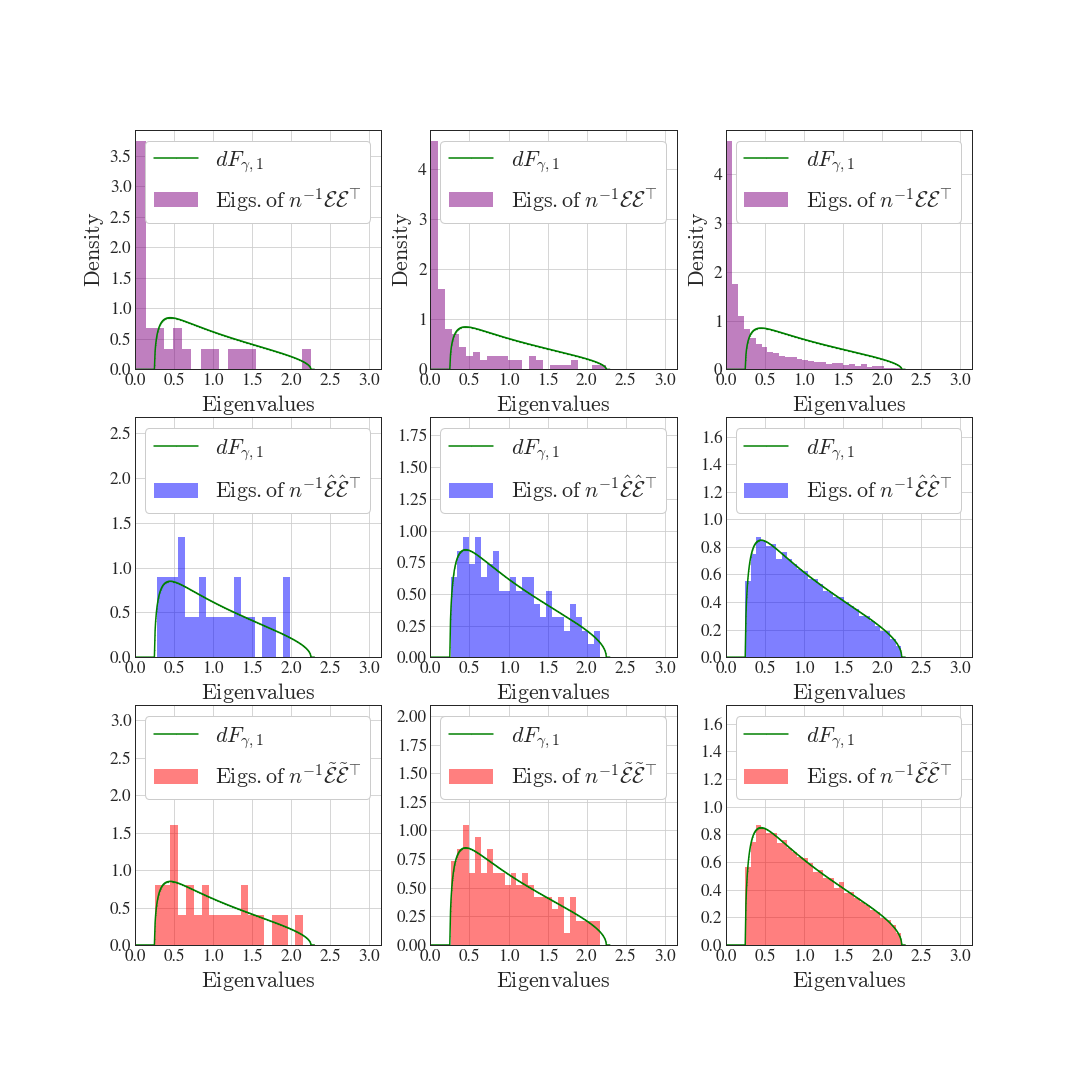}
        }
    }

    \caption{Spectrum of simulated Poisson noise versus the standard ($\sigma = 1$) Marchenko-Pastur density $dF_{\gamma, 1}$, for aspect ratios $\gamma = m/n = 1/3,\,1/4$ and matrix dimensions $n=100,500,5000$ (from left to right in each panel). The top row in each panel (purple) corresponds to the eigenvalues of $\Sigma_n$ (i.e., without any row or column scaling). The center row in each panel (blue) corresponds to the eigenvalues of $\hat{\Sigma}_n$ (i.e., after scaling with the estimated scaling factors). The bottom row in each panel (red) correspond to the eigenvalues of $\widetilde{\Sigma}_n$ (i.e., after scaling with the exact scaling factors). 
    }
     \label{fig:MPfit2_extra_aspect_ratios}
\end{figure}

\section{Reproducibility details}
\subsection{Figure~\ref{fig:introduction example}} \label{reproducibility details introduction example}
We first generated an $m\times r$ matrix $B$ by sampling its entries independently from the log-normal distribution with mean $0$ and variance $4$ (i.e., from $\operatorname{exp}(2Z)$, where $Z\sim \mathcal{N}(0,1)$). Then, we generated an $r\times n$ matrix $C$ by sampling its entries independently from the uniform distribution over $[0,1]$. Lastly, we computed ${X} = BC$, normalized $X$ by a scalar so that its average entry is $1$, and sampled the entries of $Y$ from the Poisson distribution as in~\eqref{eq:Y poisson model}. After generating $Y$, Algorithm~\ref{alg:noise standardization} was applied to $Y$ with scaling tolerance $\delta = 10^{-12}$. 

\subsection{Figures~\ref{fig:MPfit2} and~\ref{fig:MPfiterror}} \label{reproducibility details Poisson convergence to MP law}
We generated the matrix $X$ by sampling its entries independently from $\operatorname{Unif}(1,2)$, namely, the uniform distribution over $[1,2]$, and multiplied the resulting matrix from left and right by diagonal matrices whose entries (on the main diagonal) were sampled independently from $\exp(\operatorname{Unif}(-2,2))$. Then, each entry $Y_{i,j}$ was sampled independently from \sloppy $\operatorname{Poisson}(X_{i,j})$. 
Note that Theorems~\ref{thm:Marchenko-Pastur for biwhitened noise} and~\ref{thm:Marchenko-Pastur for biwhitened noise from estimated variances} in Section~\ref{sec:method and results} do not make any assumptions about the rank of $X$, and indeed, in this experiment the matrix $X$ has full rank with probability $1$.
After generating $Y$, we obtained the eigenvalues of the matrices $\Sigma_n = n^{-1}\mathcal{E}\mathcal{E}^\top$, $\hat{\Sigma}_n =  n^{-1}\hat{\mathcal{E}}\hat{\mathcal{E}}^\top$, $\tilde{\Sigma}_n = n^{-1}\widetilde{\mathcal{E}}\widetilde{\mathcal{E}}^\top$ (corresponding to the original noise matrix, the biwhitened noise matrix using the estimated scalings factors, and the biwhitened noise matrix using the exact scaling factors, respectively). To compute these matrices, the Sinkhorn-Knopp algorithm (Algorithm~\ref{alg:SK}) was used with tolerance $\delta=10^{-11}$.

\subsection{Rank estimation accuracy (Section~\ref{section:rank estimation})} \label{appendix:reproducibility for rank selection simulations}
For the first two experiments (Figures~\ref{fig:mild heteroskedasticity} and~\ref{fig:strong heteroskedasticity}), the Poisson parameter matrix was generated as $X = cUV$, $U \in \mathbb{R}^{500 \times 20}$, $V \in \mathbb{R}^{20 \times 750}$, where $c$ is a positive scalar. For Figure~\ref{fig:mild heteroskedasticity}, we used $U_{ij}\sim \exp(\mathrm{Unif}(-1,1))$ and $V_{ij} \sim \mathrm{Unif}(0,1)$, whereas for Figure~\ref{fig:strong heteroskedasticity} we used $U_{ij}\sim \exp(2\cdot \mathcal{N}(0,1))$ and $V_{ij} \sim \mathrm{Unif}(0,1)$. In both cases we used the scalar $c$ to control the average value of $X$.
For the third experiment (Figure~\ref{fig:strong factor}), the Poisson parameter matrix was generated as $X = c UV + pq^\top$, $U \in \mathbb{R}^{500 \times 19}$, $V \in \mathbb{R}^{19 \times 750}$, and $c$ is a positive scalar, where $U_{ij}\sim \exp(\mathrm{Unif}(-1,1))$, $V_{ij} \sim \mathrm{Unif}(0,1)$, and the rank-1 factor $pq^\top$ was generated by sampling $p_i, q_j \in \exp(2\cdot \mathcal{N}(0,1))$. We then varied the scalar $c$ to control the average value of the matrix $cUV$.

\subsection{Experiments on real data} \label{appendix:reproducibility details for real data}
\subsubsection{Fits to the MP law (Section~\ref{sec:real data fit to the MP law})} \label{appendix:reproducibility details for real data MP fit}
For the PBMC dataset, for each trial of sample-splitting we randomly chose $20000$ cells and split them into two equal groups to create two matrices. Then, for each of these matrices we removed all columns (genes) that had less than or equal to $200$ nonzeros, and removed all rows (cells) that had less than or equal to $200$ nonzeros in the resulting matrix. We applied the same pipeline to the Hrvatin dataset except that we initially retained $10000$ cells from the data, and later used $100$ as a threshold for the sparsity of genes and cells. For the Hi-C dataset, for each trial of sample-splitting we randomly split the loci of chromosome 1 into two equal groups to create two matrices. We then removed from each of them all columns (chromosome 1 loci) that had less than or equal to $10$ nonzeros, and further removed all rows (chromosome 2 loci) that had less than or equal to $10$ nonzeros in the resulting matrix. For the AP dataset, for each trial of sample-splitting we randomly split the documents into two groups to create two matrices. Then, we removed from each of them all columns (terms/words) with $30$ or less nonzeros, and further removed all rows (documents) with $2$ or less nonzeros in the resulting matrix. In addition, we removed duplicate documents and terms from the filtered matrices. For the 20 NewsGroups dataset, we used the same pipeline except that that we first randomly chose $10000$ documents and removed $1\%$ of the most popular terms across the chosen documents, before the rest of the sample splitting procedure and sparsity filtering. 

\subsubsection{Rank estimation with ground truth (Section~\ref{sec:rank estmation real data})} \label{appendix:reproducibility details for real data rank estimation}
For the PBMC dataset, each cell in the dataset was initially labelled with one of $10$ cell types. We chose the following $8$ cell types that should be well separated: `naive\_t', `b\_cells', `cd14\_monocytes', `naive\_cytotoxic', `memory\_t', `regulatory\_t', `cd56\_nk', and `cytotoxic\_t'. We randomly sampled $500$ cells from each type to form a matrix with $4000$ columns, and removed the rows (genes) that have 100 or fewer nonzeros.
For the 20 Newsgroups dataset, each document was initially labelled with one of $20$ topics. We selected the following $10$ topics: `alt.atheism', `comp.sys.mac.hardware', `comp.windows.x', `misc.forsale', `rec.motorcycles', `rec.sport.hockey', `sci.space', `soc.religion.christian', `talk.politics.guns', `talk.politics.misc'. We then randomly chose $500$ documents from each topic to form a matrix with $5000$ columns. We removed the rows (terms) that had $100$ or less nonzeros, and further removed the resulting zero columns. 

\section{Fit against the MP law for several families with quadratic variance functions} \label{appendix:experiments for other distributions}

\newcommand{\Ex}{\mathbb{E}}
\newcommand{\paren}[1]{\left(#1\right)}
\newcommand{\curly}[1]{\left\{#1\right\}}
\newcommand{\brac}[1]{\left[#1\right]}

In this section we provide results analogous to the ones described in Section~\ref{sec:experiments - fit to MP law} (fit of the spectrum of Poisson noise after scaling to the MP law) for the binomial, negative binomial, and generalized Poisson; see Section~\ref{sec:quadratic variance functions}.

\subsection{Binomial}
The binomial distribution depends on two parameters: the success probability and the number of trials. We generated the success probability matrix $P=(p_{i,j})_{i\in[m],\;j\in[n]}$ in the same way as we generated the Poisson parameter matrix $X$ in Section~\ref{sec:experiments - fit to MP law} (see Appendix~\ref{reproducibility details Poisson convergence to MP law}), except that we also normalized each column to sum to $1$. As for the number of binomial trials, we set it as a constant $\ell = 5$ for all $i\in[m],\;j\in[n]$.
Then, we sampled $Y_{i,j}$ independently from $\operatorname{Binomial}(p_{i,j},\ell)$. We used the true variances~\eqref{eq: quadratic variance} with $a=0$, $b=1$, $c=-1/\ell$ and their unbiased estimators~\eqref{eq:quadratic noise variance estimator} to solve the systems of equations~\eqref{eq:scaling equations from true variances general model} and~\eqref{eq:scaling equations from estimated variances general model}, respectively, using the Sinkhorn-Knopp algorithm with tolerance $\delta=10^{-11}$. 
In Figure~\ref{fig:MPBinomFit} we plot the eigenvalue histograms (normalized appropriately) of $\Sigma_n$, $\hat{\Sigma}_n$, and $\tilde{\Sigma}_n$, for aspect ratios $\gamma = m/n = 1/2, 1/3,1/4,1/5$ and column dimensions $n = 100,500,5000$.
Similarly to the Poisson (Figure~\ref{fig:MPfit2}), we obtain an accurate fit to the MP law even for moderate matrix dimensions.

\begin{figure}
    \makebox[\textwidth][c]{
        \hspace{-0.5cm}
        \subfloat[][$\gamma = 1/2$, $n = 100,\, 500,\, 5000$]{
        \includegraphics[width=0.6\textwidth]{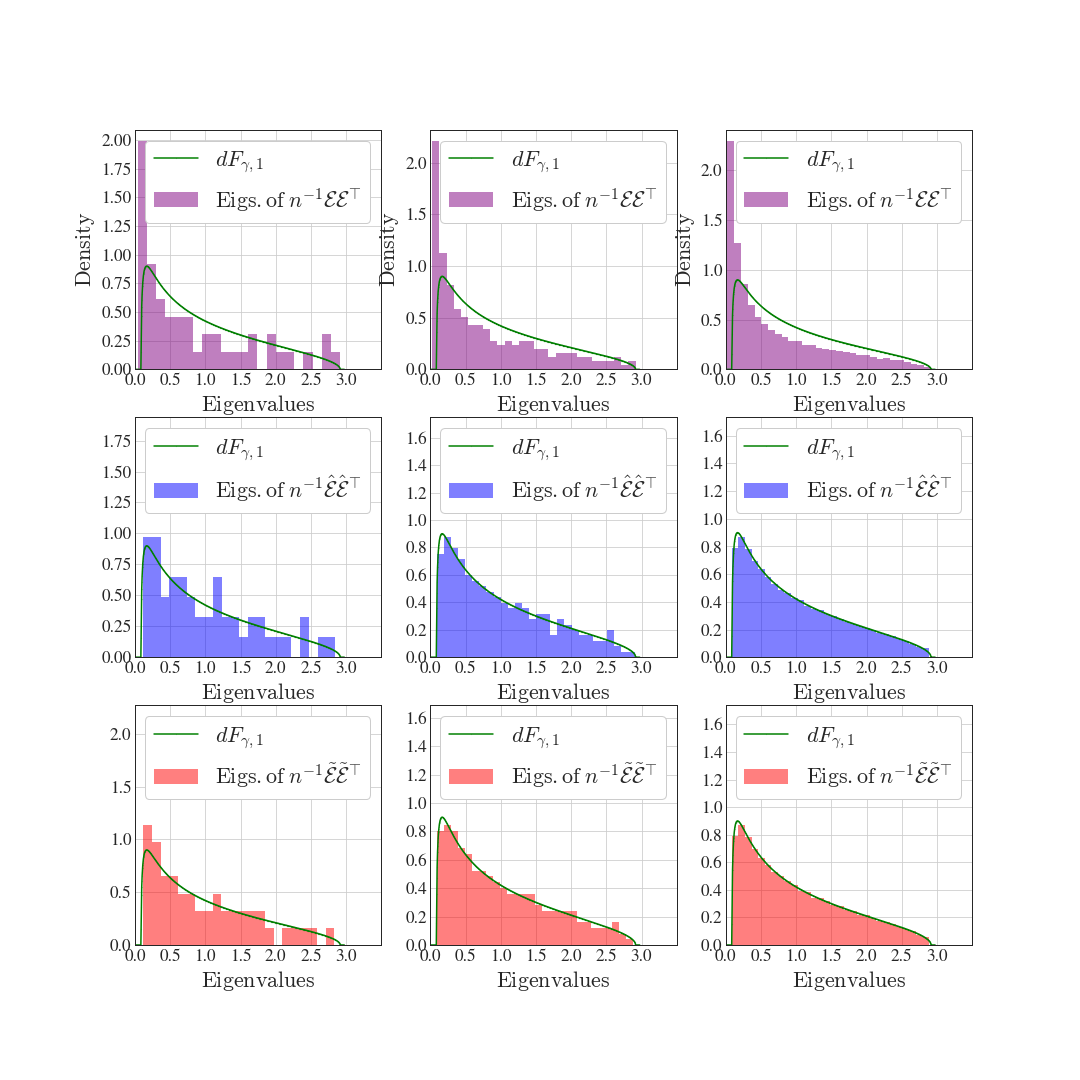}
        }
        \hspace{-1.3cm}
        \subfloat[][$\gamma = 1/3$, $n = 100,\, 500,\, 5000$]{
        \includegraphics[width=0.6\textwidth]{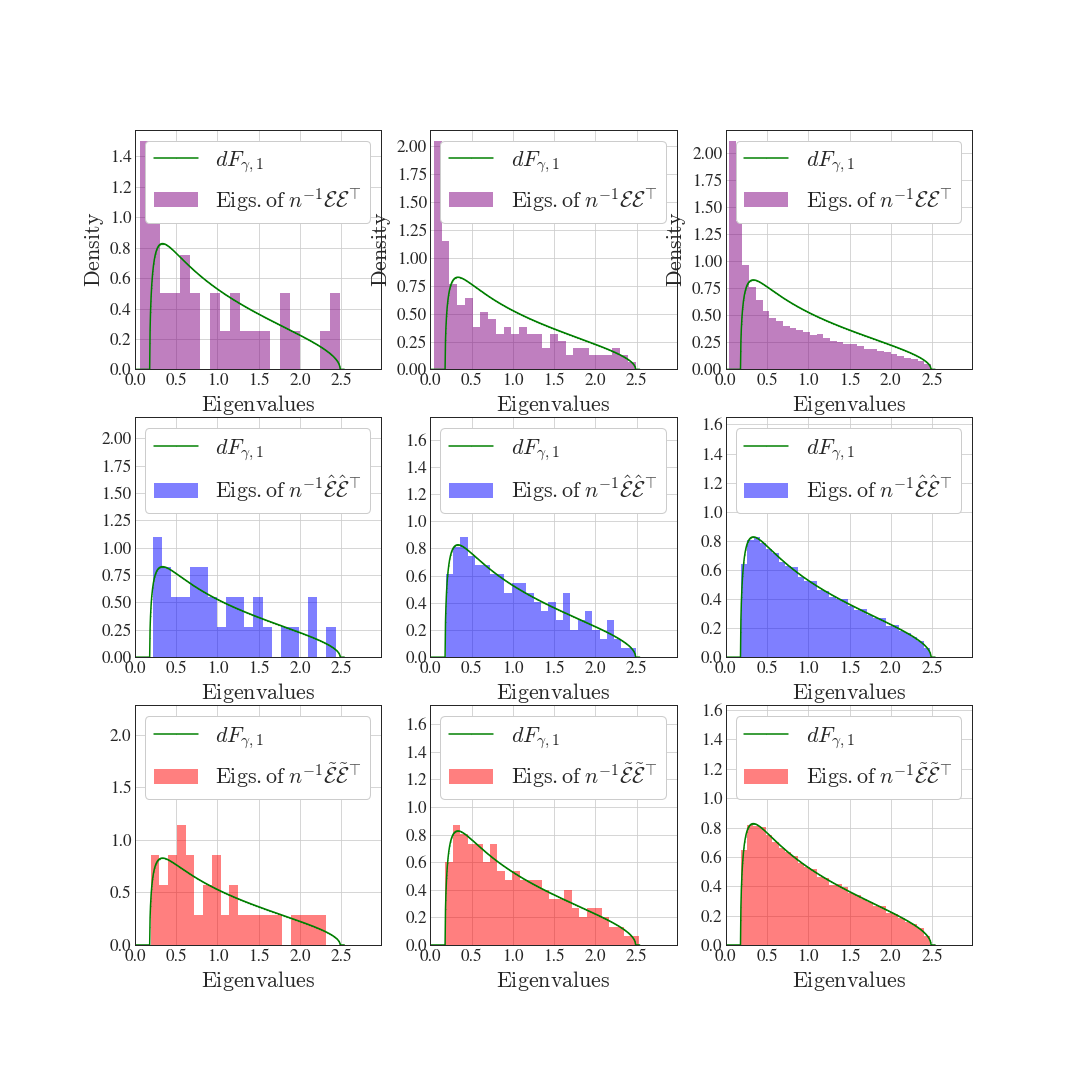}
        }
    }
    
    \makebox[\textwidth][c]{
        \hspace{-0.5cm}
        \subfloat[][$\gamma = 1/4$, $n = 100,\, 500,\, 5000$]{
        \includegraphics[width=0.6\textwidth]{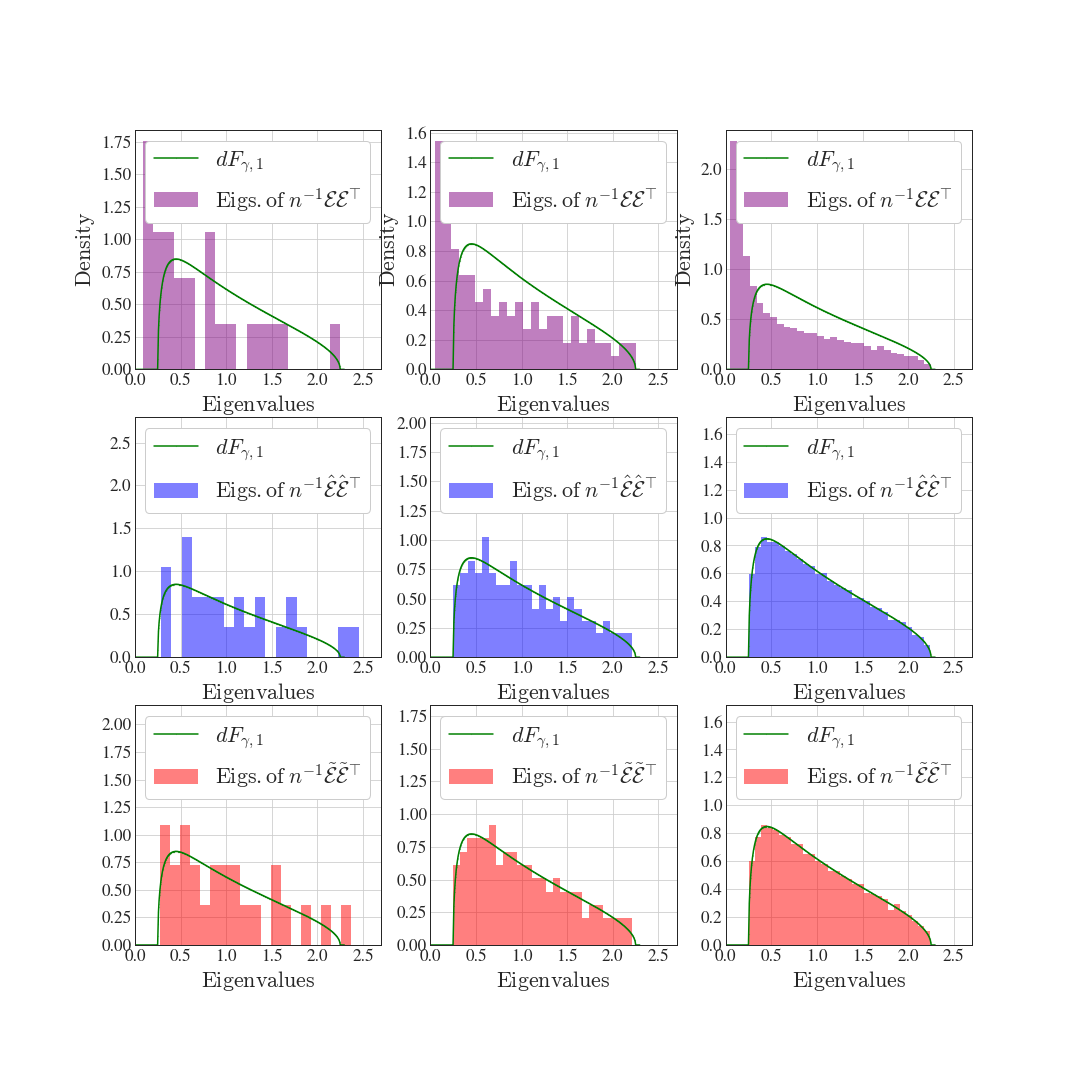}
        }
        \hspace{-1.3cm}
        \subfloat[][$\gamma = 1/5$, $n = 100,\, 500,\, 5000$]{
        \includegraphics[width=0.6\textwidth]{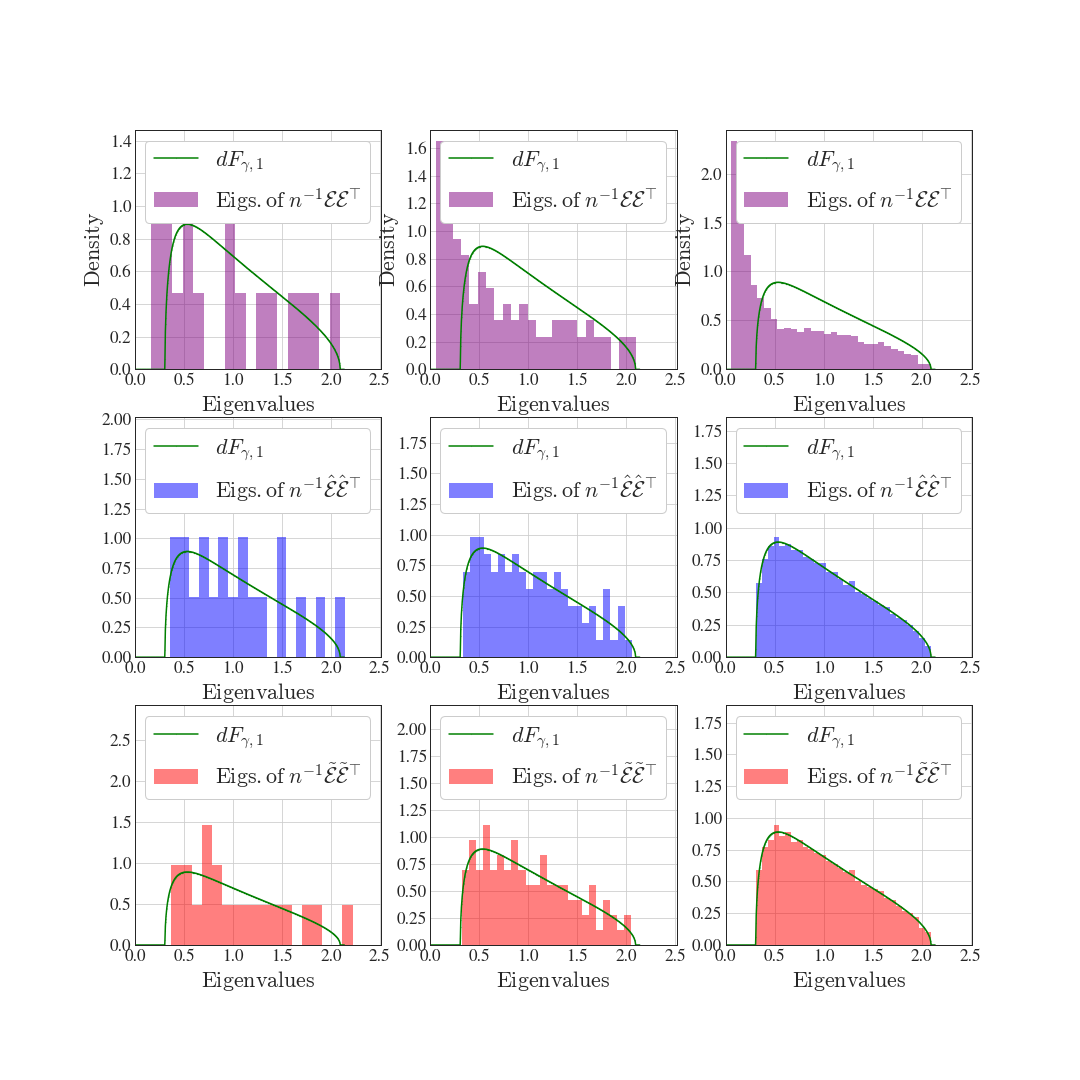}
        }
    }
    
    \caption{Spectrum of simulated binomial noise versus the standard ($\sigma = 1$) Marchenko-Pastur density $dF_{\gamma, 1}$, for various aspect ratios $\gamma$ and matrix dimensions $n=100,500,5000$ (from left to right in each panel). The top row in each panel (purple) corresponds to the eigenvalues of $\Sigma_n$ (i.e., without any row or column scaling). The center row in each panel (blue) corresponds to the eigenvalues of $\hat{\Sigma}_n$ (i.e., after scaling with the estimated scaling factors). The bottom row in each panel (red) correspond to the eigenvalues of $\widetilde{\Sigma}_n$ (i.e., after scaling with the true scaling factors). 
    }
    \label{fig:MPBinomFit}
\end{figure}

\subsection{Negative Binomial}

The negative binomial distribution depends on two parameters: the number of failures and the probability of success. In this experiment, we set the number of failures for all entries in the matrix to be $\rho = 3$. We generated the matrix $X$ as for the experiment in Section~\ref{sec:experiments - fit to MP law} (see Appendix~\ref{reproducibility details Poisson convergence to MP law}), and formed the matrix of success probabilities for the negative binomials as $p_{i,j} = {X_{i,j}}/(r+X_{i,j})$. We then sampled each $Y_{i,j}$ independently from $\operatorname{NegBinomial}(p_{i,j},\rho)$, and computed the eigenvalues of $\Sigma_n$, $\hat{\Sigma}_n$, and $\widetilde{\Sigma}_n$, where we used the true variance~\eqref{eq: quadratic variance} with $a=0$, $b=1$, $c=1/\rho$ and its unbiased estimator~\eqref{eq:variance estimator with alpha and beta} to solve the systems of equations~\eqref{eq:scaling equations from true variances general model} and~\eqref{eq:scaling equations from estimated variances general model} (using the Sinkhorn-Knopp algorithm with tolerance $\delta = 10^{-11}$). 
In Figure~\ref{fig:MPNegBinomFit} we plot the resulting eigenvalue histograms (normalized appropriately) of $\Sigma_n$, $\hat{\Sigma}_n$, and $\widetilde{\Sigma}_n$, for aspect ratios $\gamma = m/n = 1/2, 1/3,1/4,1/5$ and column dimensions $n = 100,500,5000$.
As before, the results are very similar to the ones in the Poisson case (Figure~\ref{fig:MPfit2}), demonstrating the convergence to the MP law as the dimensions grow.

\begin{figure}
    \makebox[\textwidth][c]{
    \hspace{-0.5cm}
        \subfloat[][$\gamma = 1/2$, $N = 100,\, 500,\, 5000$]{
        \includegraphics[width=0.6\textwidth]{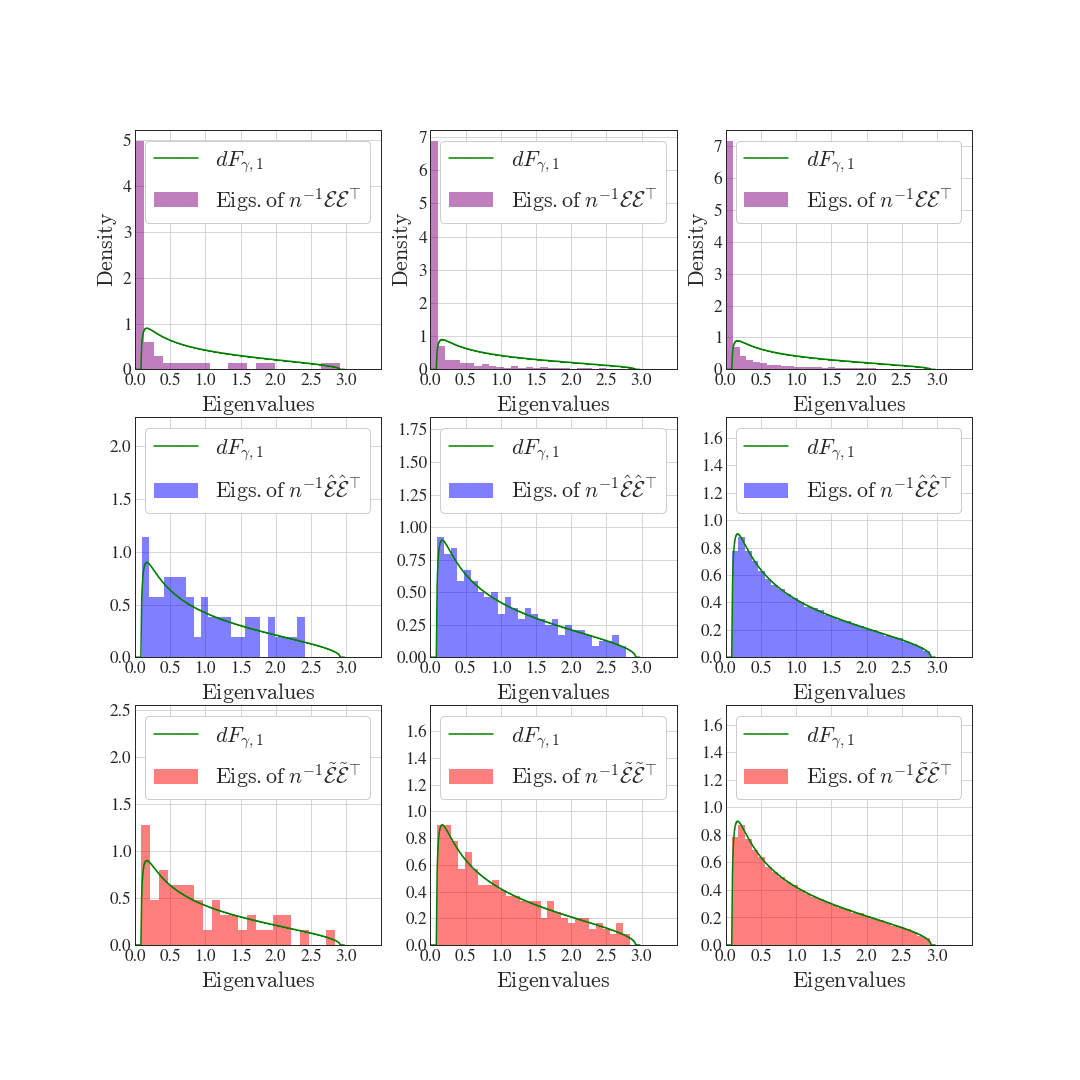}
        }
        \hspace{-1.3cm}
        \subfloat[][$\gamma = 1/3$, $N = 100,\, 500,\, 5000$]{
        \includegraphics[width=0.6\textwidth]{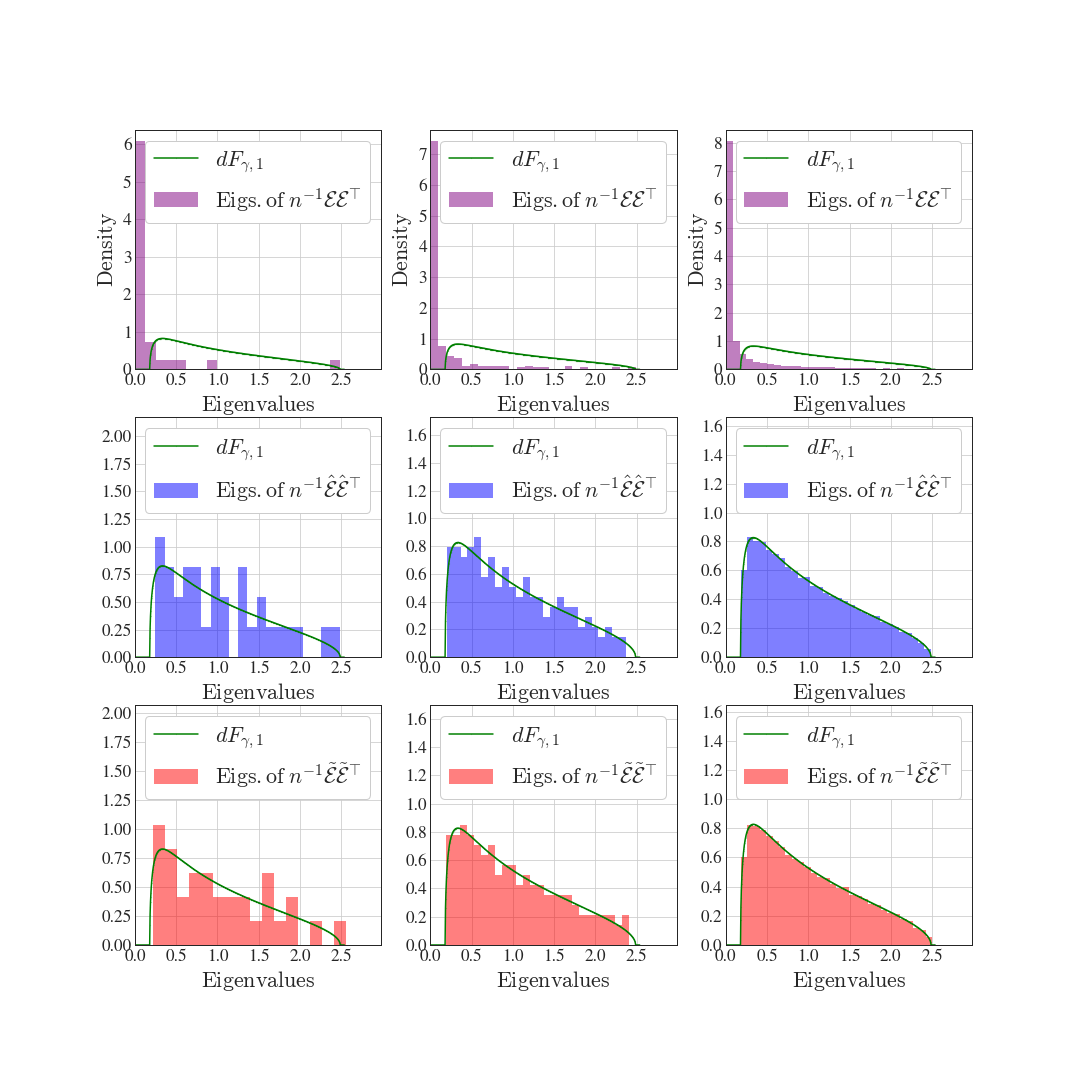}
        }
    }
    
    \makebox[\textwidth][c]{
    \hspace{-0.5cm}
        \subfloat[][$\gamma = 1/4$, $N = 100,\, 500,\, 5000$]{
        \includegraphics[width=0.6\textwidth]{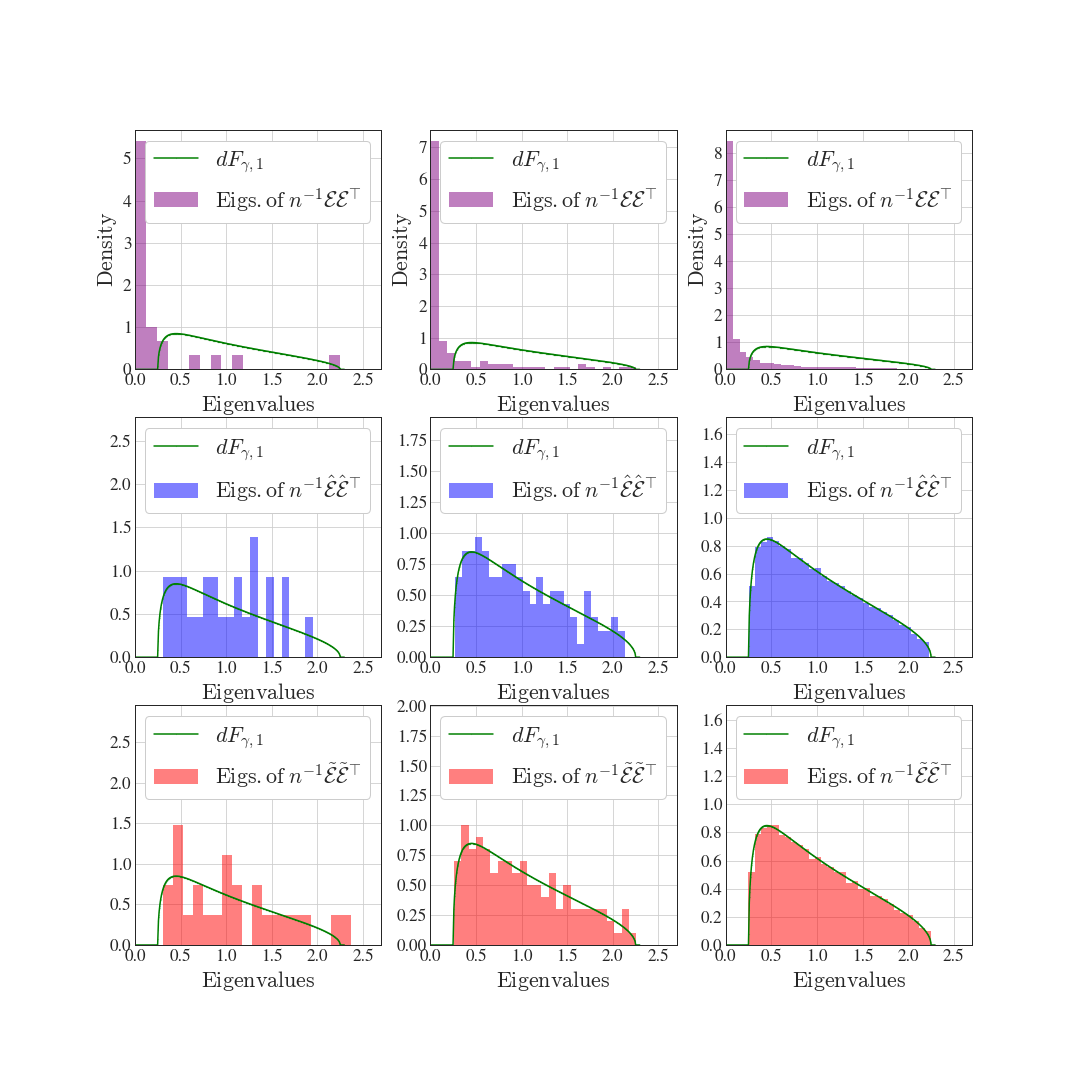}
        }
        \hspace{-1.3cm}
        \subfloat[][$\gamma = 1/5$, $N = 100,\, 500,\, 5000$]{
        \includegraphics[width=0.6\textwidth]{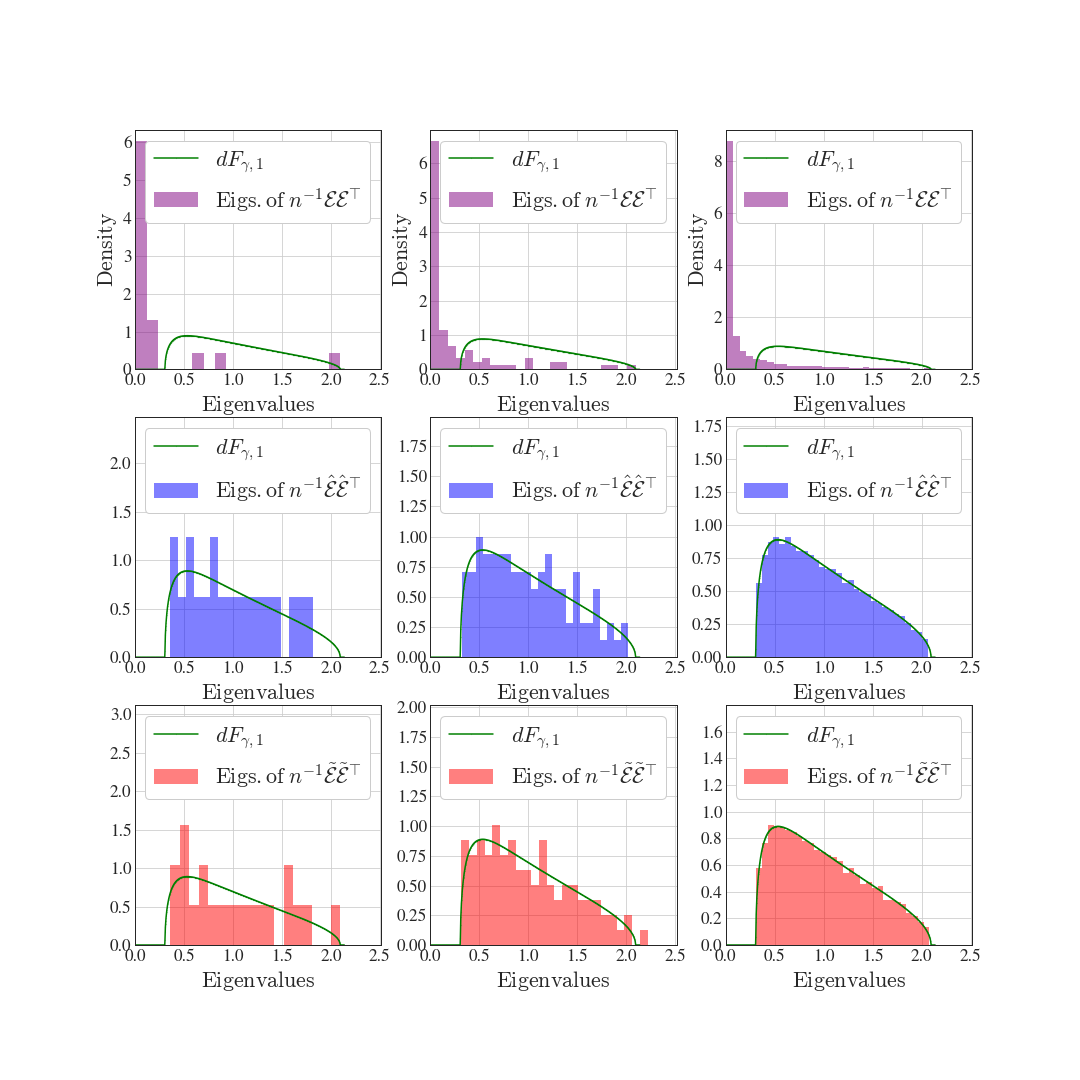}
        }
    }
    
    \caption{Spectrum of simulated negative binomial noise versus the standard ($\sigma = 1$) Marchenko-Pastur density $dF_{\gamma, 1}$, for various aspect ratios $\gamma$ and matrix dimensions $n=100,500,5000$ (from left to right in each panel). The top row in each panel (purple) corresponds to the eigenvalues of $\Sigma_n$ (i.e., without any row or column scaling). The center row in each panel (blue) corresponds to the eigenvalues of $\hat{\Sigma}_n$ (i.e., after scaling with the estimated scaling factors). The bottom row in each panel (red) correspond to the eigenvalues of $\widetilde{\Sigma}_n$ (i.e., after scaling with the true scaling factors)}
    \label{fig:MPNegBinomFit}
\end{figure}

\subsection{Generalized Poisson}
We simulated data from the Generalized Poisson distribution~\cite{consul1973generalization} as $Y_{i,j} \sim \operatorname{GP}(X_{i,j}, \eta)$, where $X_{i,j}$ is the rate parameter and  $\eta$ is the dispersion parameter. We randomly generated the rate parameters $X_{i,j}$ in the same way as we generated the Poisson parameters in the example of Section~\ref{sec:experiments - fit to MP law} (see Appendix~\ref{reproducibility details Poisson convergence to MP law}), and fixed the dispersion parameter $\eta = 0.1$. We used the true noise variances~\eqref{eq: quadratic variance} with $a=0$, $b = 1/(1-\eta)^2$, $c=0$ (see~\cite{consul1973generalization}) and their unbiased estimators~\eqref{eq:quadratic noise variance estimator} to solve the systems of equations~\eqref{eq:scaling equations from true variances general model} and~\eqref{eq:scaling equations from estimated variances general model}, respectively, using the Sinkhorn-Knopp algorithm with tolerance $\delta=10^{-11}$. 
Next, we computed $\mathcal{E} = Y - \Ex\brac{Y}$, $\Sigma_n = n^{-1}\mathcal{E}\mathcal{E}^T$, $\hat{\Sigma}_n = n^{-1}\hat{\mathcal{E}}\hat{\mathcal{E}}^T$, and $\widetilde{\Sigma}_n = n^{-1}\widetilde{\mathcal{E}}\widetilde{\mathcal{E}}^T$.  
In Figure~\ref{fig:MPGPFit} we plot the eigenvalue histograms (normalized appropriately) of $\Sigma_n$, $\hat{\Sigma}_n$, and $\tilde{\Sigma}_n$, for aspect ratios $\gamma = m/n = 1/2, 1/3,1/4,1/5$ and column dimensions $n = 100,500,5000$.
The results are very similar to the ones in the Poisson case (Figure~\ref{fig:MPfit2}), and we see an excellent fit to the MP law even for moderate matrix dimensions.

\begin{figure}
    \makebox[\textwidth][c]{
        \hspace{-0.5cm}
        \subfloat[][$\gamma = 1/2$, $N = 100,\, 500,\, 5000$]{
        \includegraphics[width=0.6\textwidth]{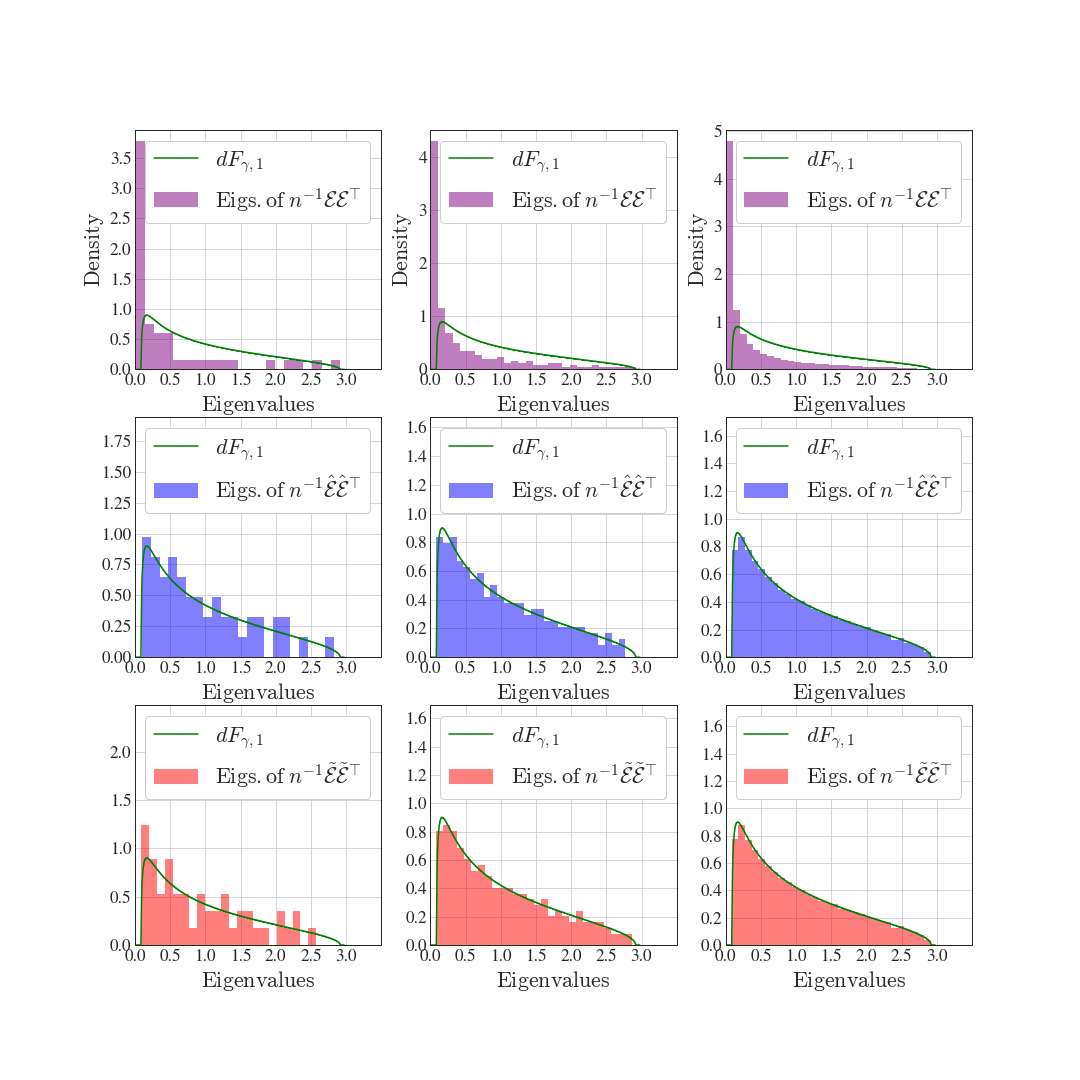}
        }
        \hspace{-1.3cm}
        \subfloat[][$\gamma = 1/3$, $N = 100,\, 500,\, 5000$]{
        \includegraphics[width=0.6\textwidth]{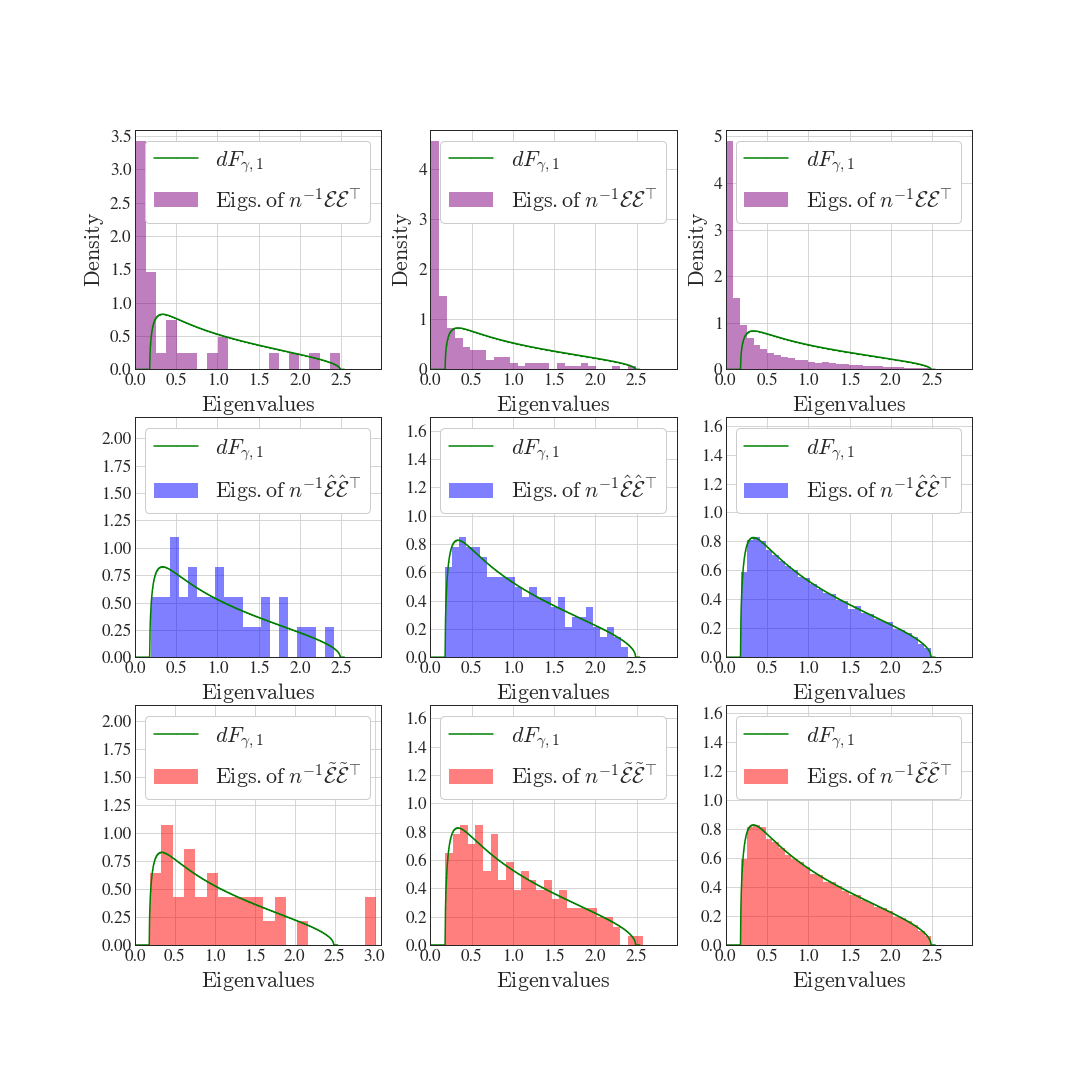}
        }
    }
    
    \makebox[\textwidth][c]{
        \hspace{-0.5cm}
        \subfloat[][$\gamma = 1/4$, $N = 100,\, 500,\, 5000$]{
        \includegraphics[width=0.6\textwidth]{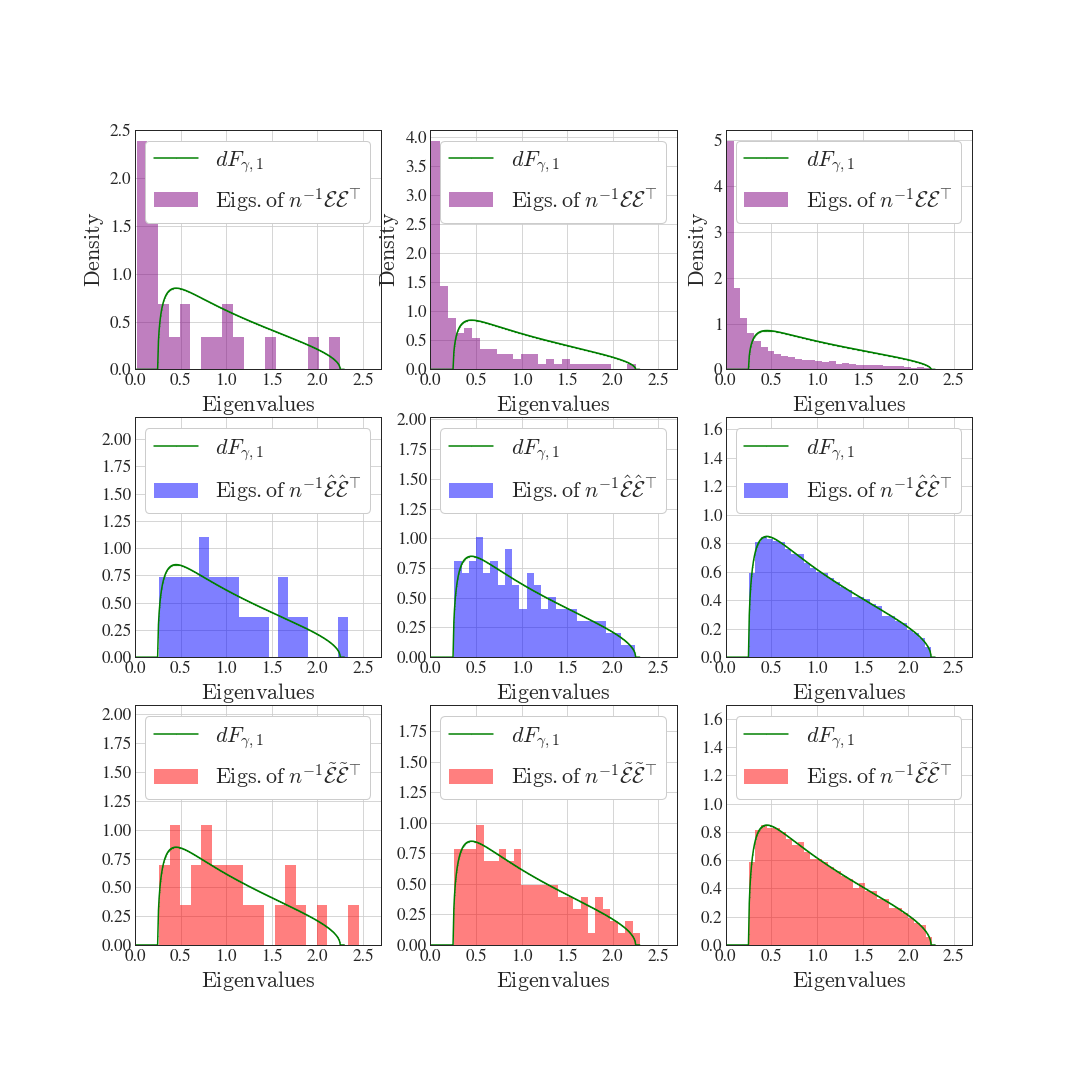}
        }
        \hspace{-1.3cm}
        \subfloat[][$\gamma = 1/5$, $N = 100,\, 500,\, 5000$]{
        \includegraphics[width=0.6\textwidth]{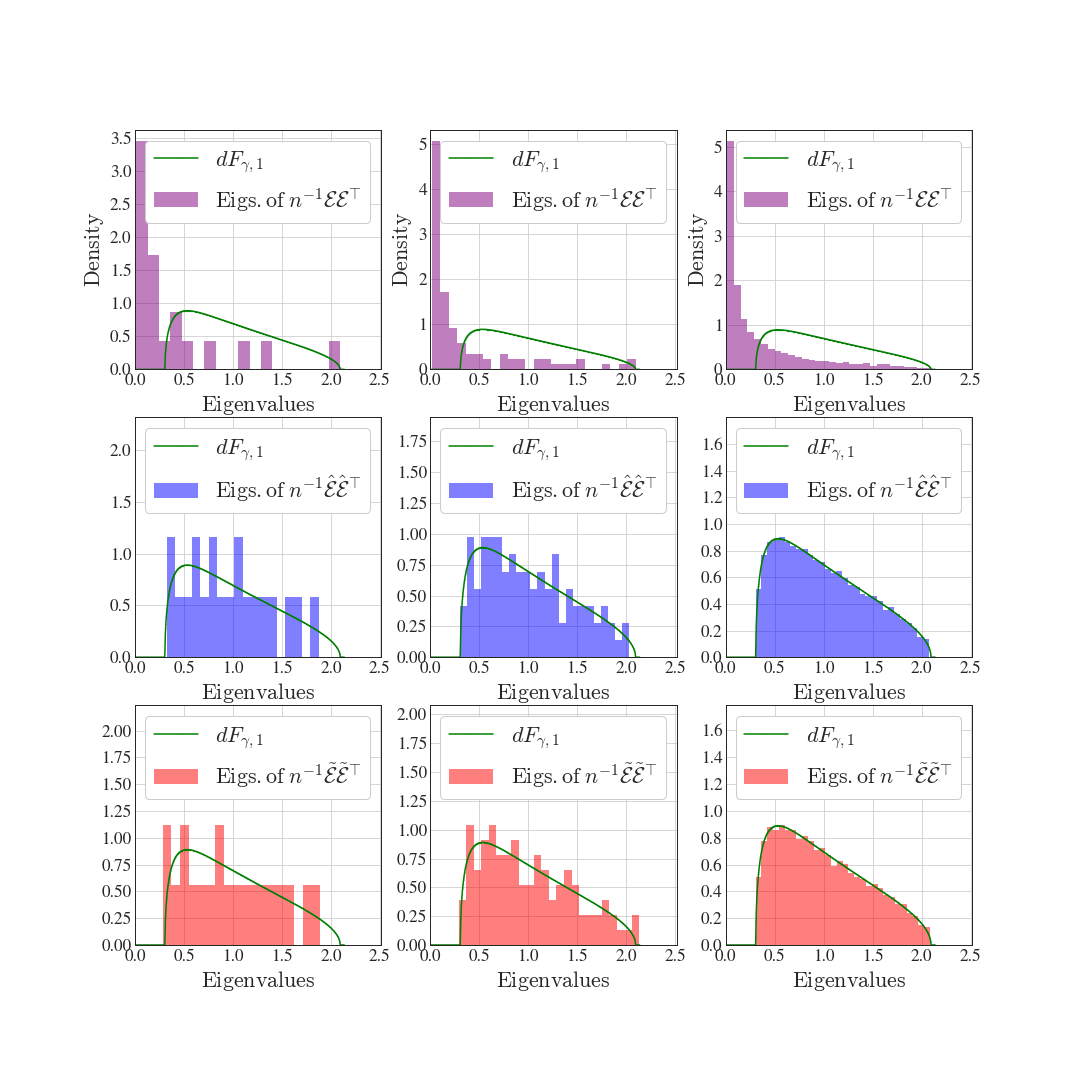}
        }
    }
    
    \caption{Spectrum of simulated Generalized Poisson noise versus the standard ($\sigma = 1$) Marchenko-Pastur density $dF_{\gamma, 1}$, for various aspect ratios $\gamma$ and matrix dimensions $n=100,500,5000$ (from left to right in each panel). The top row in each panel (purple) corresponds to the eigenvalues of $\Sigma_n$ (i.e., without any row or column scaling). The center row in each panel (blue) corresponds to the eigenvalues of $\hat{\Sigma}_n$ (i.e., after scaling with the estimated scaling factors). The bottom row in each panel (red) correspond to the eigenvalues of $\widetilde{\Sigma}_n$ (i.e., after scaling with the true scaling factors)}
    \label{fig:MPGPFit}
\end{figure}

\section{Numerical experiments for Section~\ref{sec:adapting to the data}} \label{appendix:adaptive method numerical experiments}

We exemplify the adaptive approach described in Section~\ref{sec:adapting to the data} on a simulated negative binomial matrix with $m=1000$, $n=2000$,  and rank $r=10$, where the number of negative binomial failures was set to 3. Consequently, the QVF satisfies $a=0$, $b=1$, $c=1/3$, and it easy to verify that the corresponding parameters in~\eqref{eq:variance estimator with alpha and beta} are $\alpha = 1$ and $\beta = 0.25$. 
To create the matrix of negative binomial probabilities, we first generated $X = UV$, $U \in \mathbb{R}^{1000 \times 10}$, $V \in \mathbb{R}^{10 \times 2000}$, where $U_{i,j} \sim \exp(2\cdot \mathcal{N}(0,1))$ and $V_{i,j} \sim \exp(1\cdot \mathcal{N}(0,1))$. We then normalized $X$ by a scalar to make its average entry equal to $1$, and set the negative binomial failure probabilities as $p_{i,j} = X_{i,j} / (\rho_{i,j} + X_{i,j})$, where $\rho_{i,j}$ is the number of failures for $Y_{i,j}\sim \operatorname{NegBinoimal}(\rho_{i,j},p_{i,j})$. This choice of $p_{i,j}$ ensures that $\mathbb{E}[Y_{i,j}] = X_{i,j}$. 

Figure~\ref{fig:negative binomial example Kolmogorv error vs beta} depicts the KS distance~\eqref{eq: beta estimator} for the grid of values $\beta\in \{0,0.05,\ldots,0.95,1\}$. It is evident that the smallest KS distance is $0.01$ and is attained for $\beta \in [0.2,0.4]$ (each one with its corresponding $\alpha$ chosen according to~\eqref{eq:estimator for alpha}). The reason that the distance is saturated at $0.01$ is that for these values of $\beta$ the supremum in~\eqref{eq: beta estimator} is attained at $\tau = \beta_+$, where the distance is precisely $r/m = 0.01$. 
In order to compare between the estimated scaling factors $(\hat{\mathbf{u}},\hat{\mathbf{v}})$ and the true scaling factors $(\mathbf{u},\mathbf{v})$, we eliminated the fundamental ambiguity in the scaling factors (see Proposition~\ref{prop:existcne and uniquness for Lambda}) by ensuring that $\Vert \mathbf{u} \Vert_1 = \Vert \mathbf{v} \Vert_1$ and $\Vert \hat{\mathbf{u}} \Vert_1 = \Vert \hat{\mathbf{v}} \Vert_1$. Figures~\eqref{fig:negative binomial example estimated vs true scaling factors beta theorey} and~\eqref{fig:negative binomial example estimated vs true scaling factors beta est} compare between the pair $(\hat{\mathbf{u}},\hat{\mathbf{v}})$ and the pair $(\mathbf{u},\mathbf{v})$ (each pair concatenated into a single vector) for the correct value $\beta=0.25$ and for $\beta = 0.4$, which is the value furthest away from the correct $\beta$ that achieves the minimal KS distance. As expected, the estimated scaling factors using $\beta=0.25$ are nearly identical to the true scaling factors. However, the estimated scaling factors using $\beta = 0.4$ are also very close to the true scaling factors, and allow for an excellent fit to the MP law and correct rank estimation, as can be seen in Figure~\ref{fig:negative binomial example}. We found empirically that all values of $\beta$ that attain the smallest KS distance (which is $r/m=0.01$) provide excellent fits to the MP law and lead to correct rank estimation. 

In addition to the above, we conducted a similar experiment to demonstrate the robustness of the QVF assumption~\eqref{eq: quadratic variance} to certain violations. Specifically, we randomly sampled the number of failures for each negative binomial entry uniformly at random from $\{1,\ldots,10\}$ (while keeping all other aspects of the experiment unchanged). Therefore, $\operatorname{Var}[Y_{i,j}] = X_{i,j} + c_{i,j} X_{i,j}^2$, where $c_{i,j}$ is varying between $0.1$ and $1$ across different indices $i,j$, and no single QVF exists for $Y$. Nonetheless, according to the results in~\cite{landa2020scaling}, the scaling factors of the matrix $(X_{i,j} + c_{i,j} X_{i,j}^2)_{i\leq m, \; j\leq n}$ are expected to concentrate around the scaling factors of the matrix $(X_{i,j} + \mathbb{E}[c_{i,j}] X_{i,j}^2)_{i\leq m, \; j\leq n}$, which corresponds to a standard QVF of the form~\eqref{eq: quadratic variance} with $c = \mathbb{E}[c_{i,j}] = \sum_{k=1}^{10} {k}^{-1}/10 = 0.29$.
Indeed, Figures~\ref{fig:negative binomial example KS distance rand} and~\ref{fig:negative binomial example rand} show that the procedure described in Section~\ref{sec:adapting to the data} can identify a range of parameters $\alpha$ and $\beta$ that provide accurate estimates of the true scaling factors $(\mathbf{u},\mathbf{v})$ (obtained from~\eqref{eq:scaling equations from true variances general model} using the true noise variances $\operatorname{Var}[Y_{i,j}] = X_{i,j} + c_{i,j} X_{i,j}^2$), and consequently, an excellent fit to the MP law and correct rank estimation. 
Therefore, the assumption~\eqref{eq: quadratic variance} can also serve as a useful approximation in situations where the parameters $a,b,c$ are not constant but randomly perturbed (with respect to some baseline values). 

\begin{figure} 
  \centering
  \makebox[\textwidth]{
  	{
  	\subfloat[KS distance vs. $\beta$]  
  	{ \label{fig:negative binomial example Kolmogorv error vs beta}
    \includegraphics[width=0.33\textwidth]{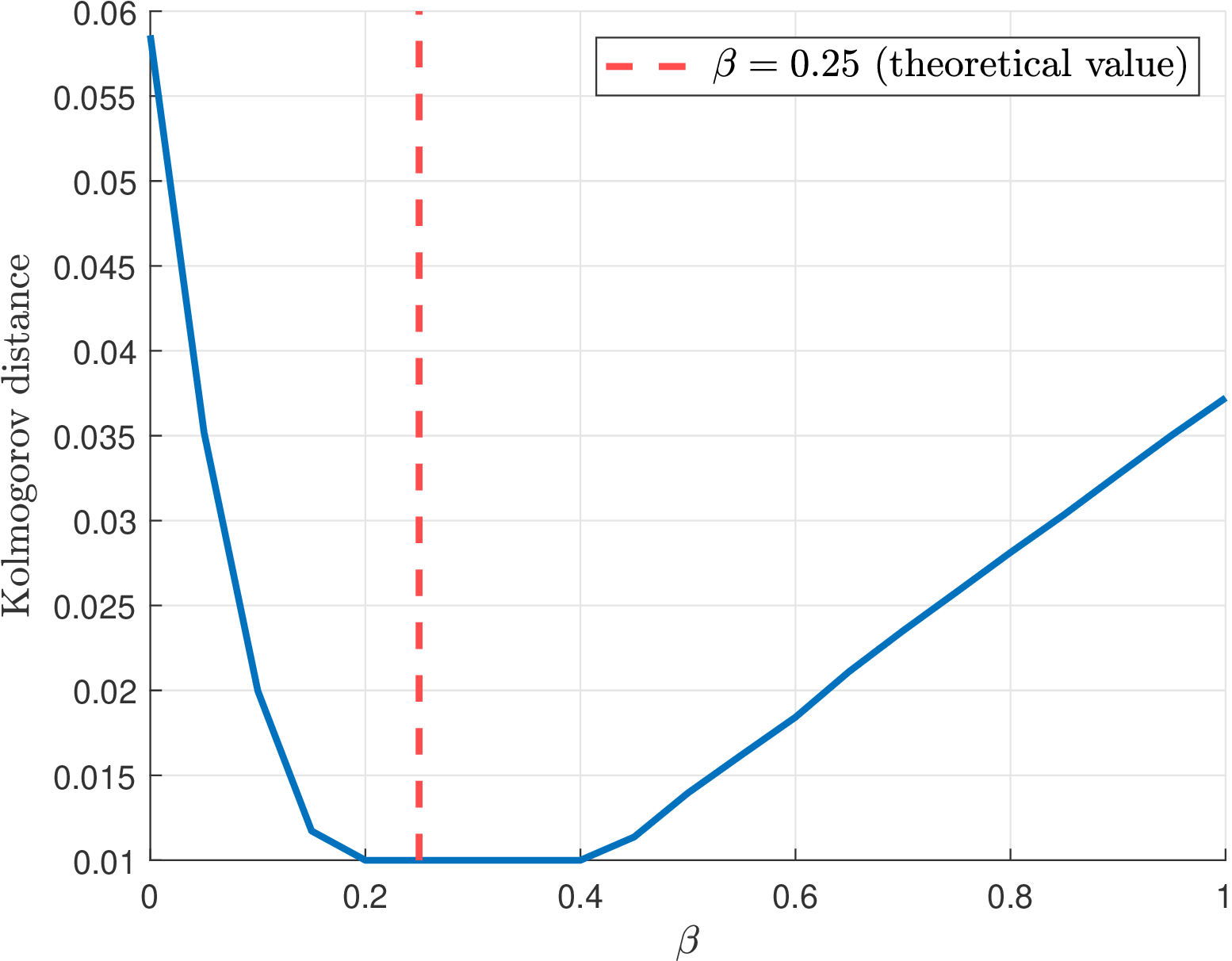} 
    }
    \subfloat[$\beta = 0.25$, $\alpha=1.01$] 
  	{ \label{fig:negative binomial example estimated vs true scaling factors beta theorey}
    \includegraphics[width=0.33\textwidth]{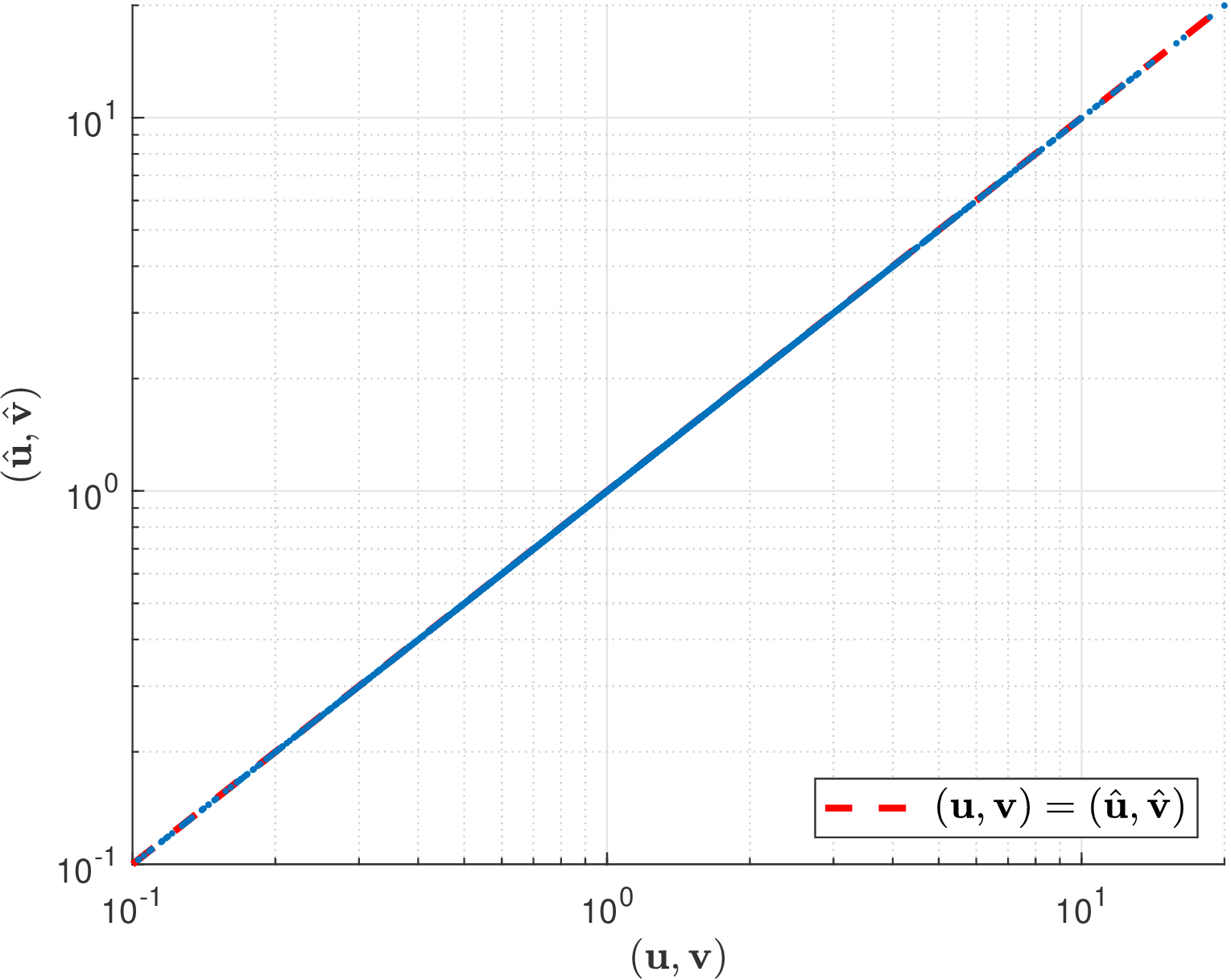}
    }
    \subfloat[$\beta = 0.4$, $\alpha=0.88$]  
  	{ \label{fig:negative binomial example estimated vs true scaling factors beta est}
    \includegraphics[width=0.33\textwidth]{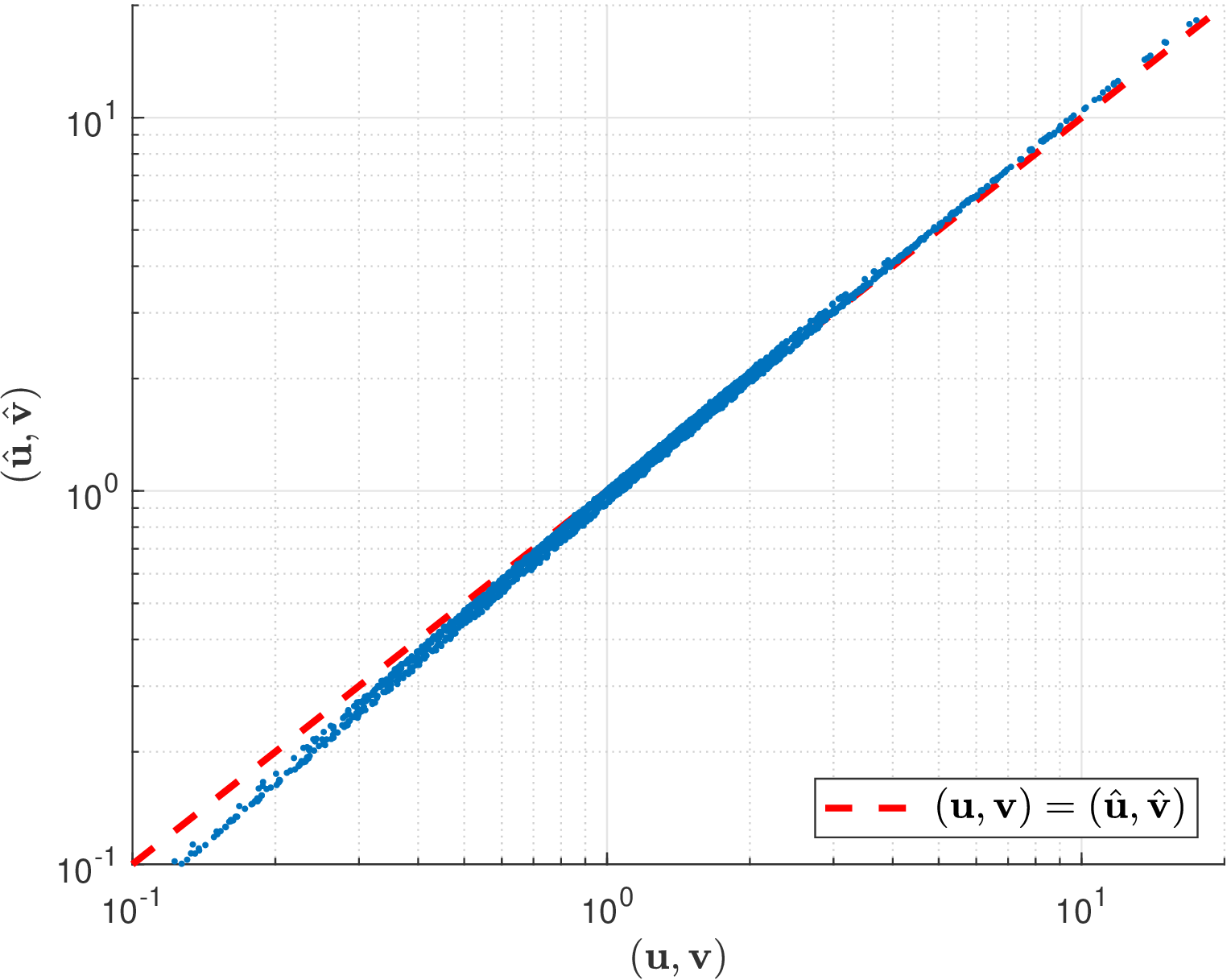} 
    }
    }
    }
    \caption
    {(a) KS distance~\eqref{eq: beta estimator} for $\beta\in \{0,0.05,\ldots,0.95,1\}$. (b) Estimated scaling factors $(\hat{\mathbf{u}},\hat{\mathbf{v}})$ versus the true scaling factors $({\mathbf{u}},{\mathbf{v}})$ using the theoretical $\beta=0.25$ and $\alpha$ chosen by~\eqref{eq:estimator for alpha}. (c) Estimated scaling factors $(\hat{\mathbf{u}},\hat{\mathbf{v}})$ versus the true scaling factors $({\mathbf{u}},{\mathbf{v}})$ using $\beta=0.4$ (the value furthest away from $\beta=0.25$ that attains the minimal KS distance of $0.01$) and $\alpha$ chosen by~\eqref{eq:estimator for alpha}.} \label{fig:negative binomial example KS distance}
    \end{figure} 

\begin{figure} 
  \centering
  	{
  	\subfloat[][Sorted eigenvalues, original data]
  	{
    \includegraphics[width=0.35\textwidth]{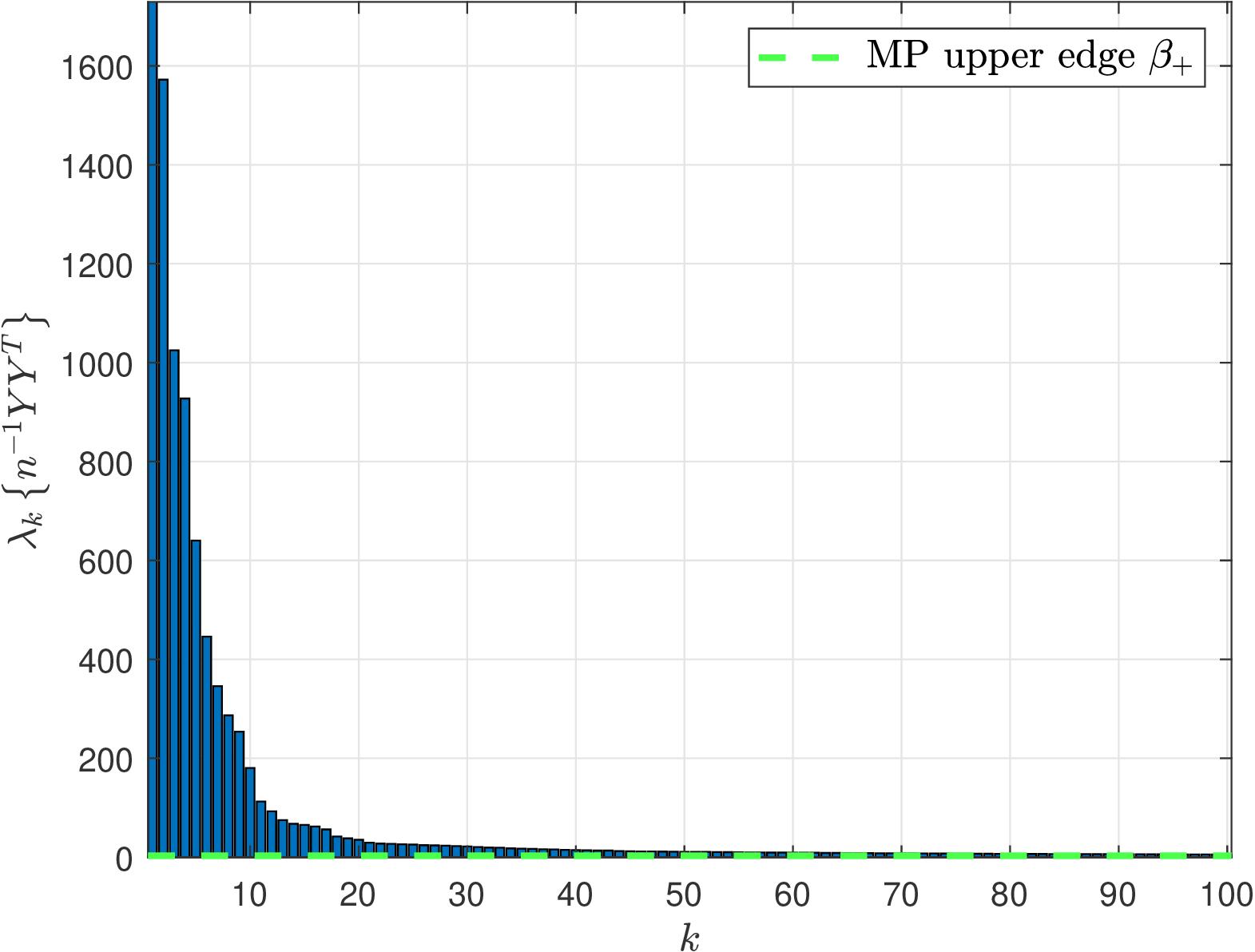} 
    }
    \subfloat[][Sorted eigenvalues, {after} biwhitening] 
  	{
    \includegraphics[width=0.35\textwidth]{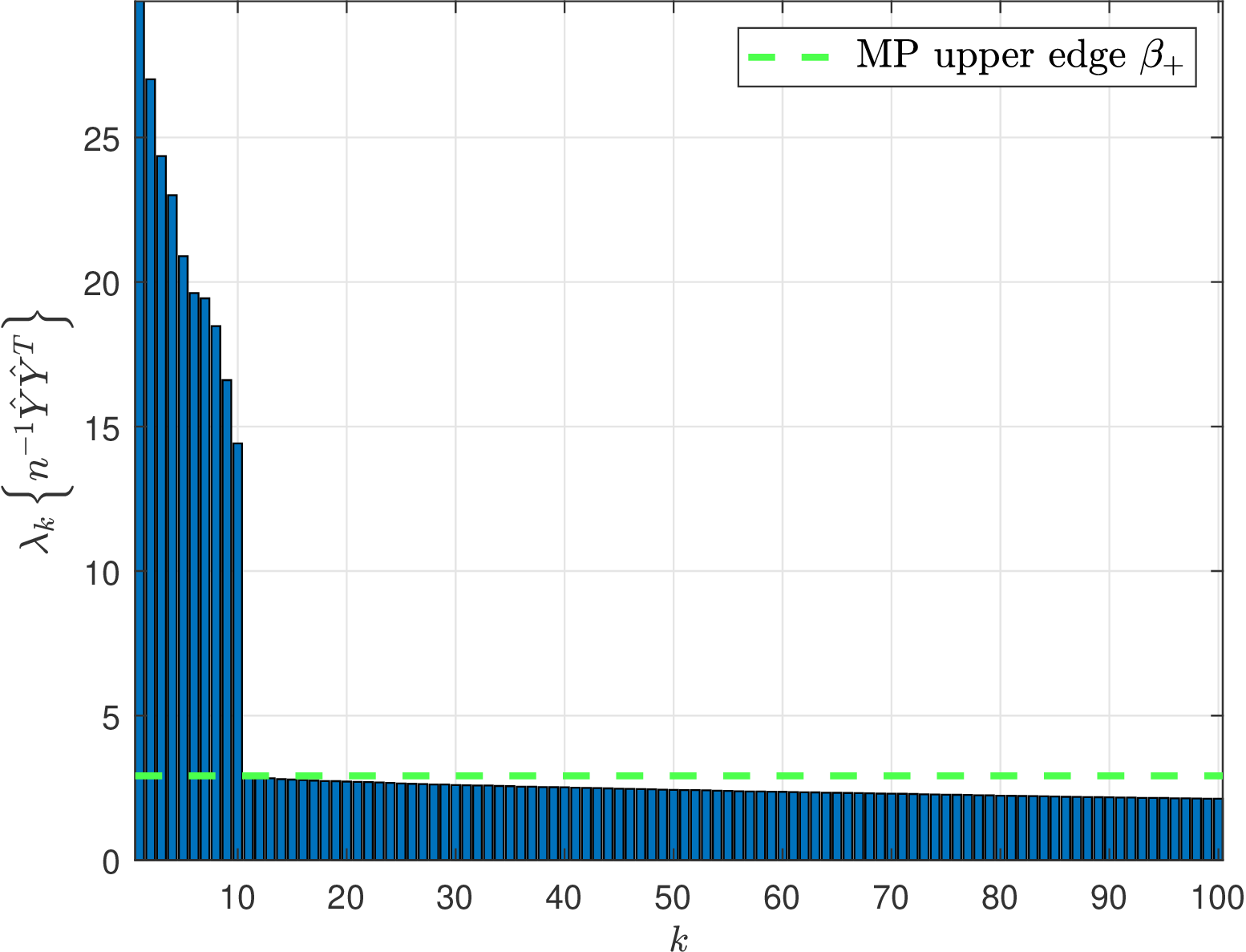} 
    }
    \\
    \subfloat[][Eigenvalue density, original data]  
  	{
    \includegraphics[width=0.35\textwidth]{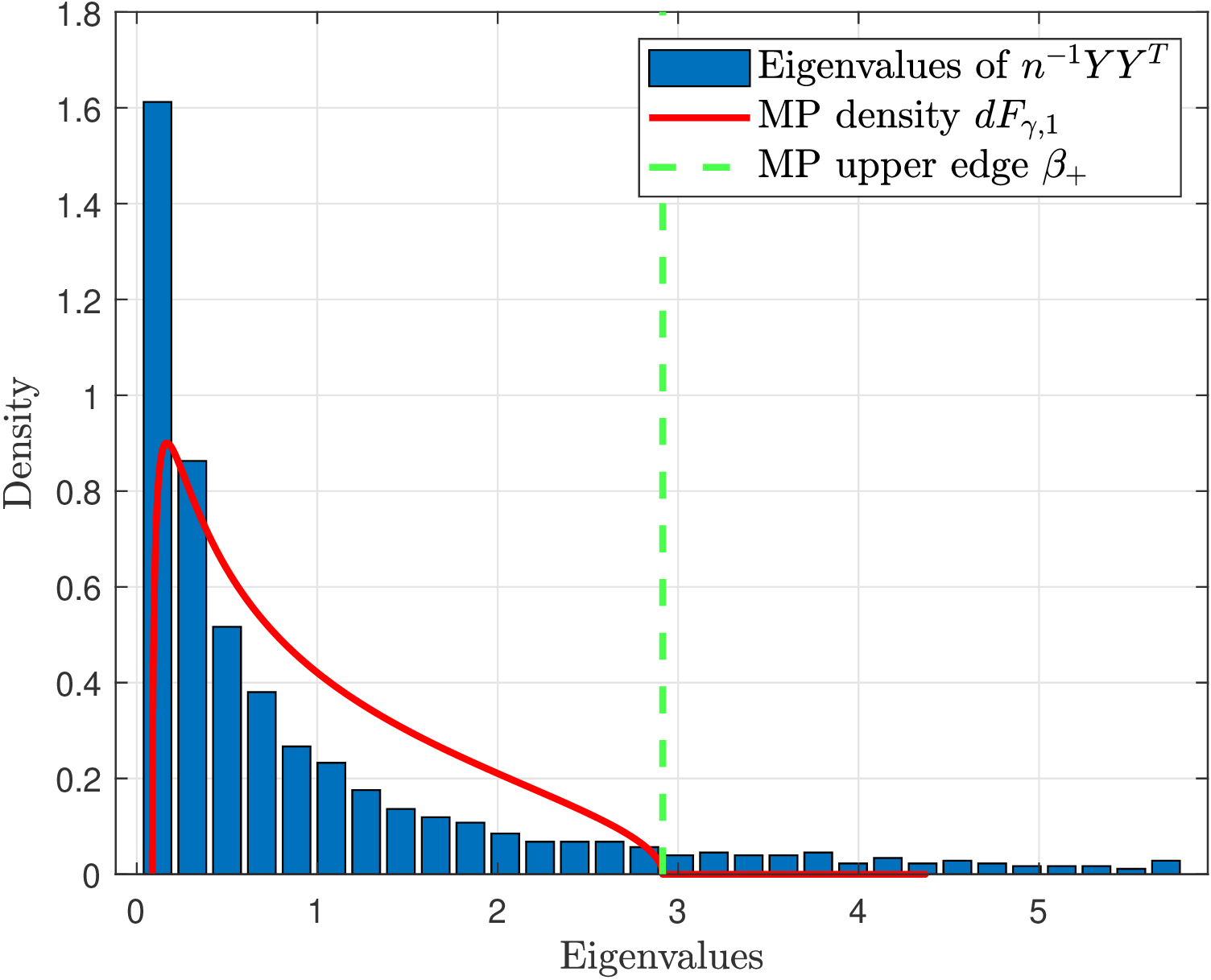}  
    }
    \subfloat[][Eigenvalue density, after biwhitening] 
  	{
    \includegraphics[width=0.35\textwidth]{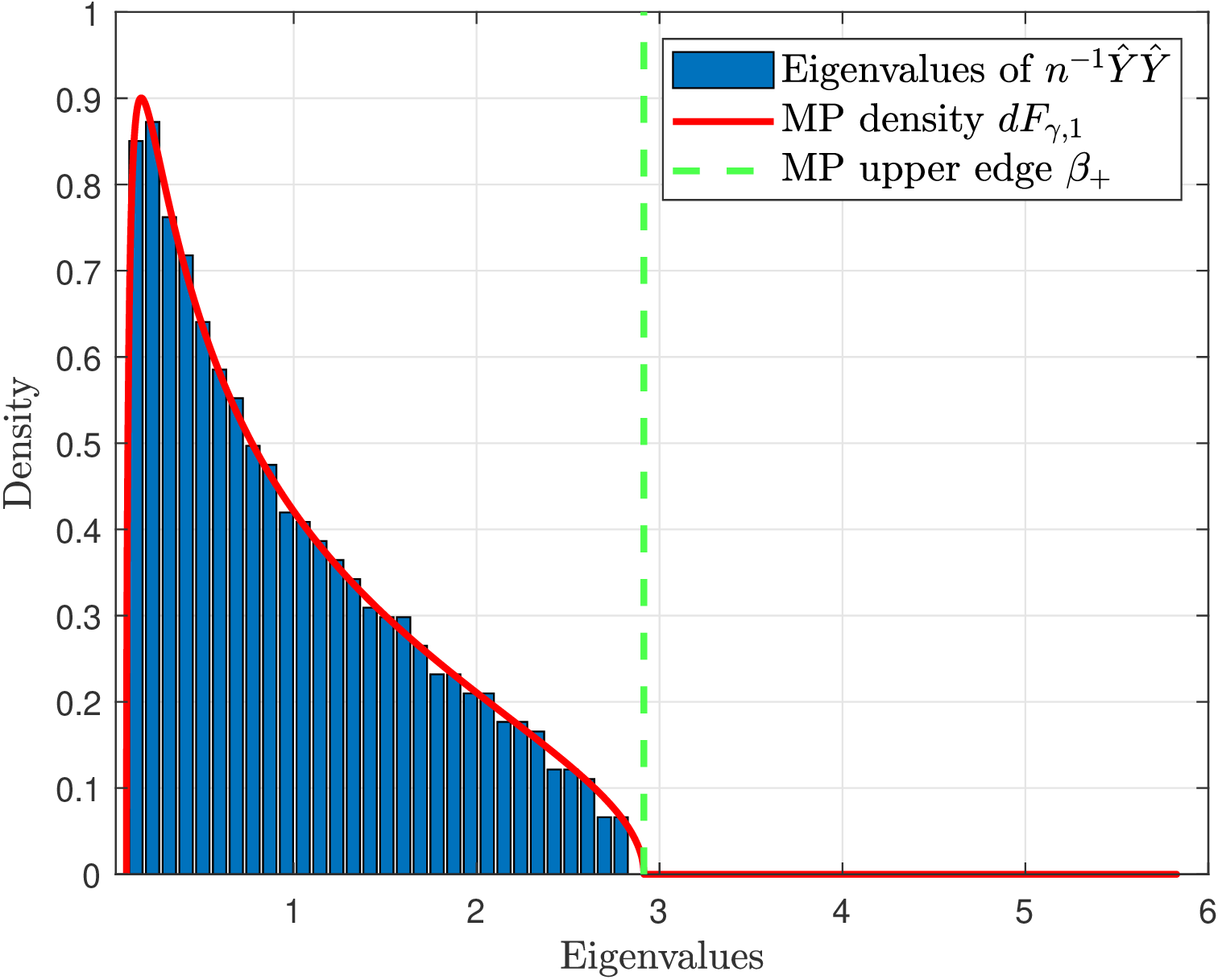} 
    } 

    }
    \caption
    {The spectrum of a simulated negative binomial matrix $Y$ with $m=1000$, $n=2000$, and $r=10$, versus the spectrum of its biwhitened version $D(\hat{\mathbf{u}}) Y D(\hat{\mathbf{v}})$. The number of negative binomial failures was set to $3$. To find $\hat{\mathbf{u}}$ and $\hat{\mathbf{v}}$ we solved~\eqref{eq:scaling equations from estimated variances general model} with the variance estimator~\eqref{eq:variance estimator with alpha and beta}, where $\beta=0.4$ and $\alpha$ was chosen according to~\eqref{eq:estimator for alpha}. 
    } \label{fig:negative binomial example}
\end{figure} 

\begin{figure}
  \centering
  \makebox[\textwidth][c]{
  	{
  	\subfloat[KS distance]  
  	{ \label{fig:negative binomial example Kolmogorv error vs beta rand}
    \includegraphics[width=0.33\textwidth]{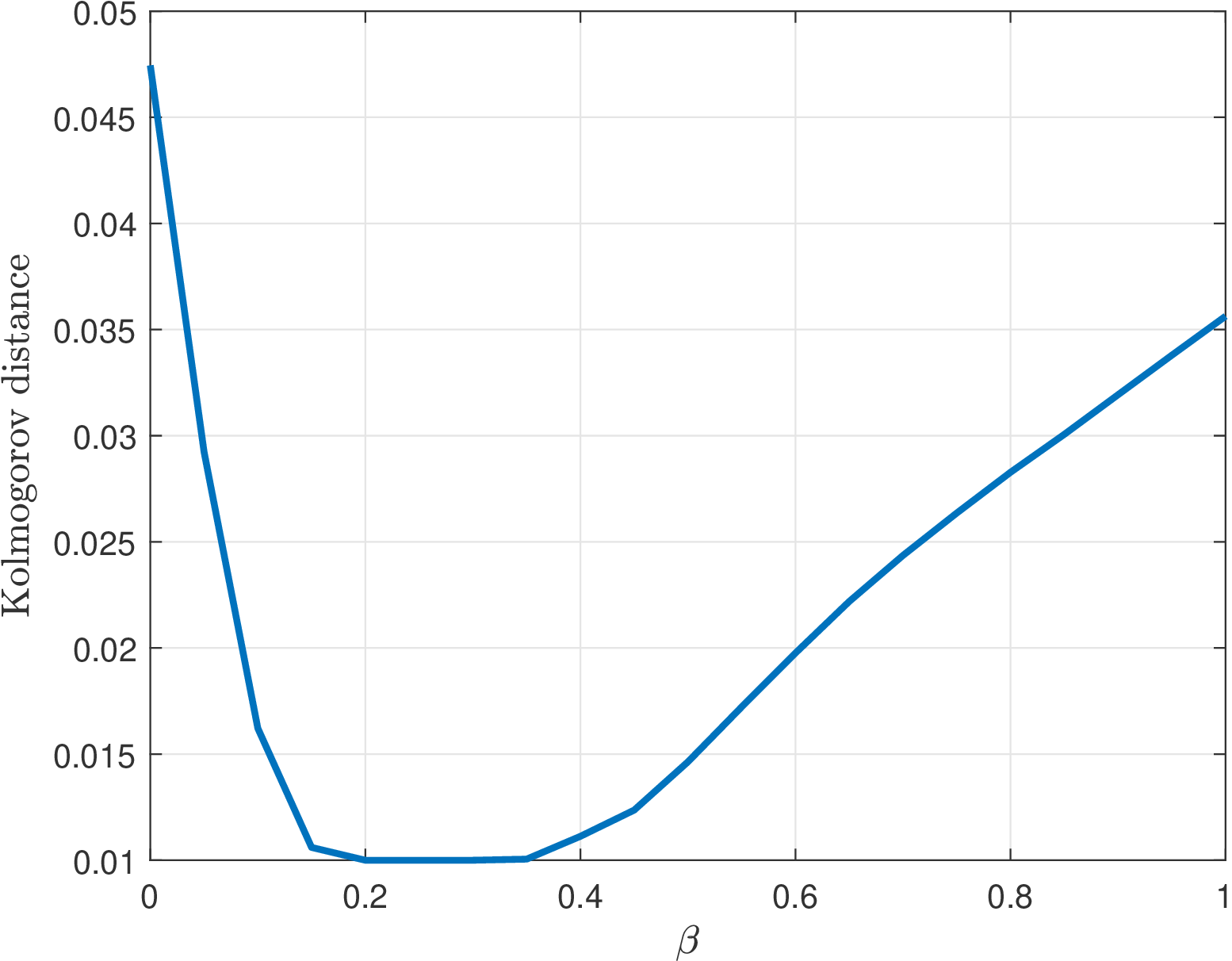} 
    }
    \subfloat[$\beta = 0.2$, $\alpha=1.03$] 
  	{ \label{fig:negative binomial example estimated vs true scaling factors beta theorey rand}
    \includegraphics[width=0.33\textwidth]{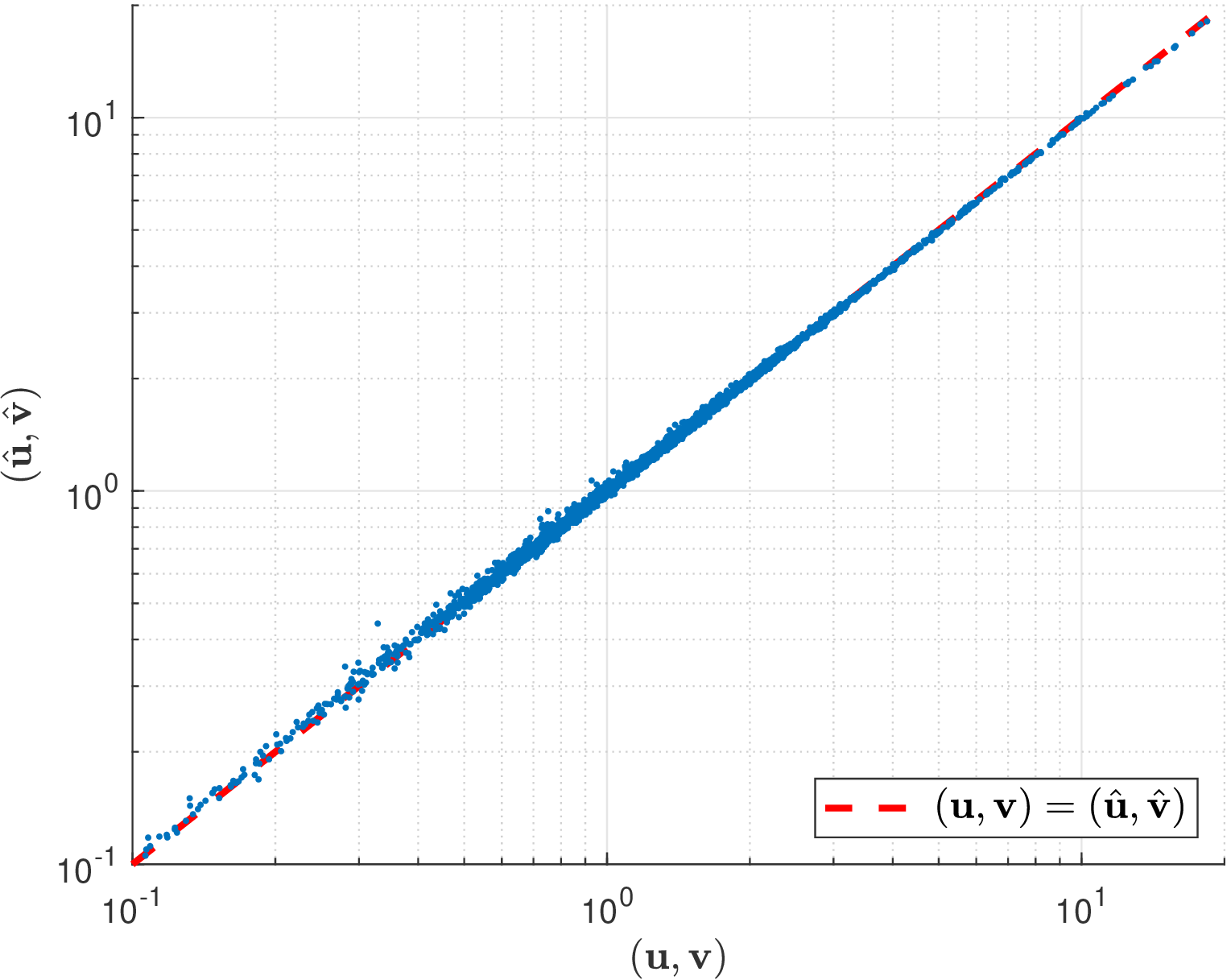}
    }
    \subfloat[$\beta = 0.3$, $\alpha=0.95$]  
  	{ \label{fig:negative binomial example estimated vs true scaling factors beta est rand}
    \includegraphics[width=0.33\textwidth]{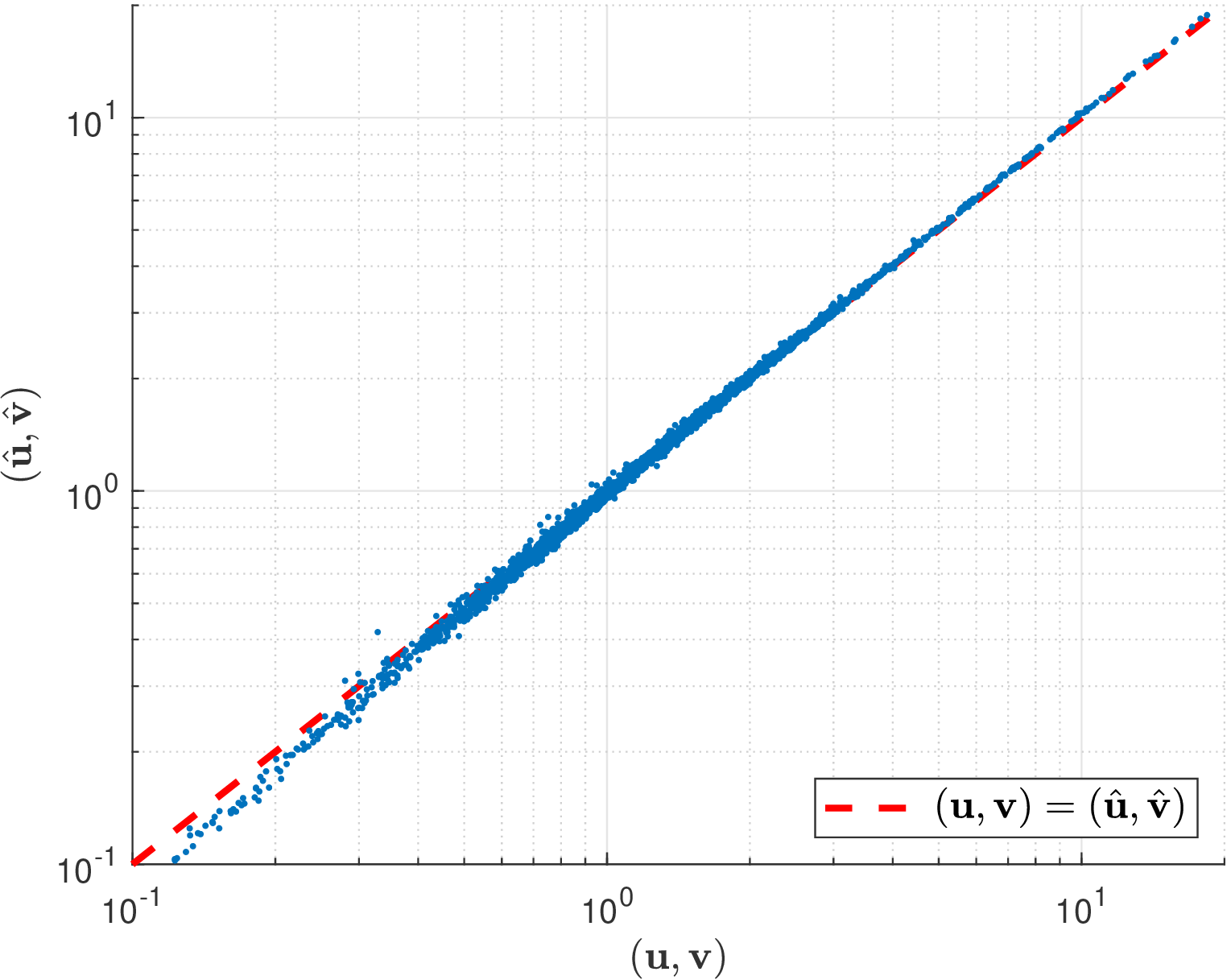} 
    }
    }
    }
    \caption
    {(a) KS distance~\eqref{eq: beta estimator} for $\beta\in \{0,0.05,\ldots,0.95,1\}$. (b) Estimated scaling factors $(\hat{\mathbf{u}},\hat{\mathbf{v}})$ versus the true scaling factors $({\mathbf{u}},{\mathbf{v}})$ using $\beta=0.2$ and $\alpha$ chosen by~\eqref{eq:estimator for alpha}. (c) Estimated scaling factors $(\hat{\mathbf{u}},\hat{\mathbf{v}})$ versus the true scaling factors $({\mathbf{u}},{\mathbf{v}})$ using $\beta=0.3$ and $\alpha$ chosen by~\eqref{eq:estimator for alpha}.} \label{fig:negative binomial example KS distance rand}
    \end{figure}

\begin{figure}
  \centering
  	{
  	\subfloat[][Sorted eigenvalues, original data]
  	{
    \includegraphics[width=0.35\textwidth]{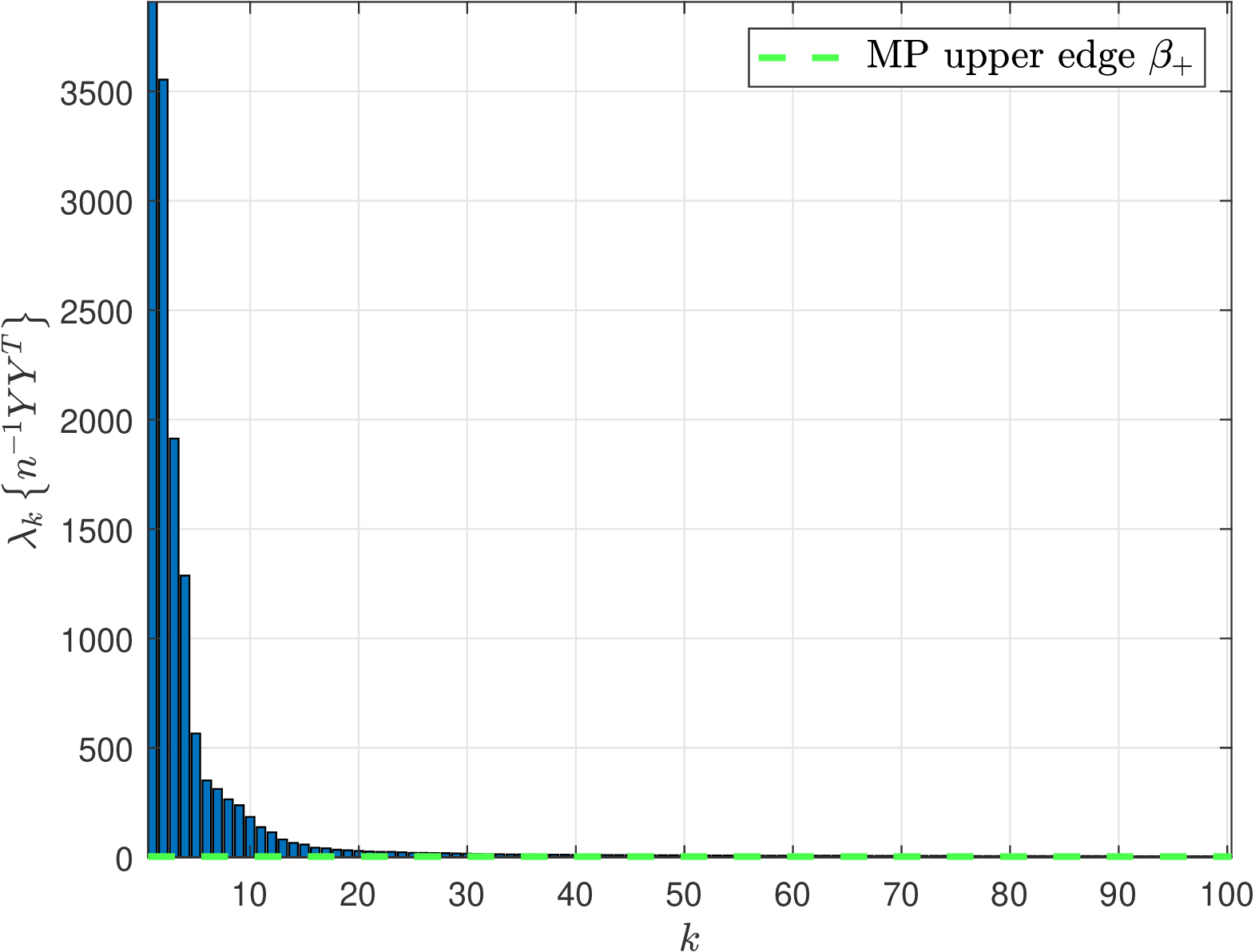} 
    }
    \subfloat[][Sorted eigenvalues, {after} biwhitening] 
  	{
    \includegraphics[width=0.35\textwidth]{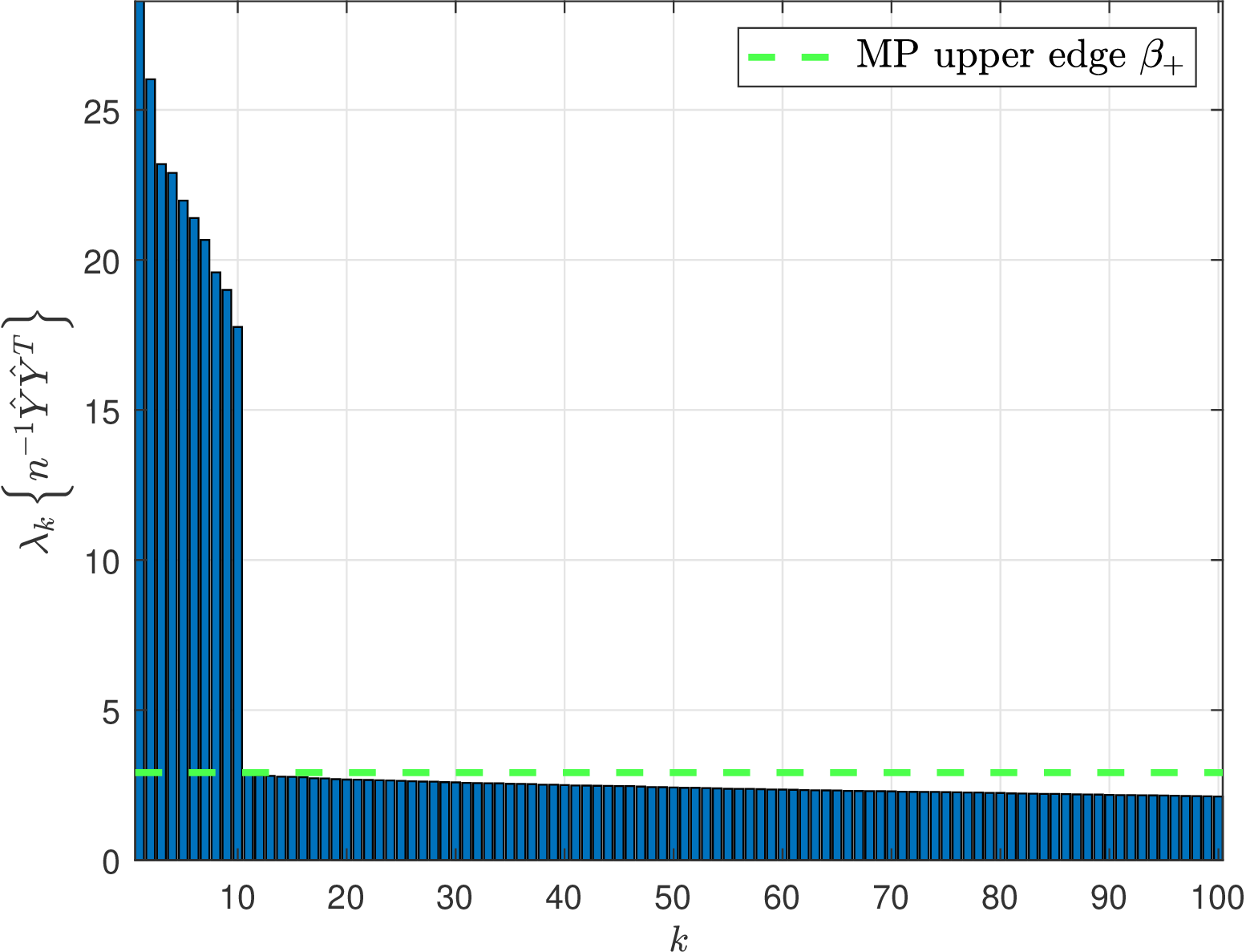} 
    }
    \\
    \subfloat[][Eigenvalue density, original data]  
  	{
    \includegraphics[width=0.35\textwidth]{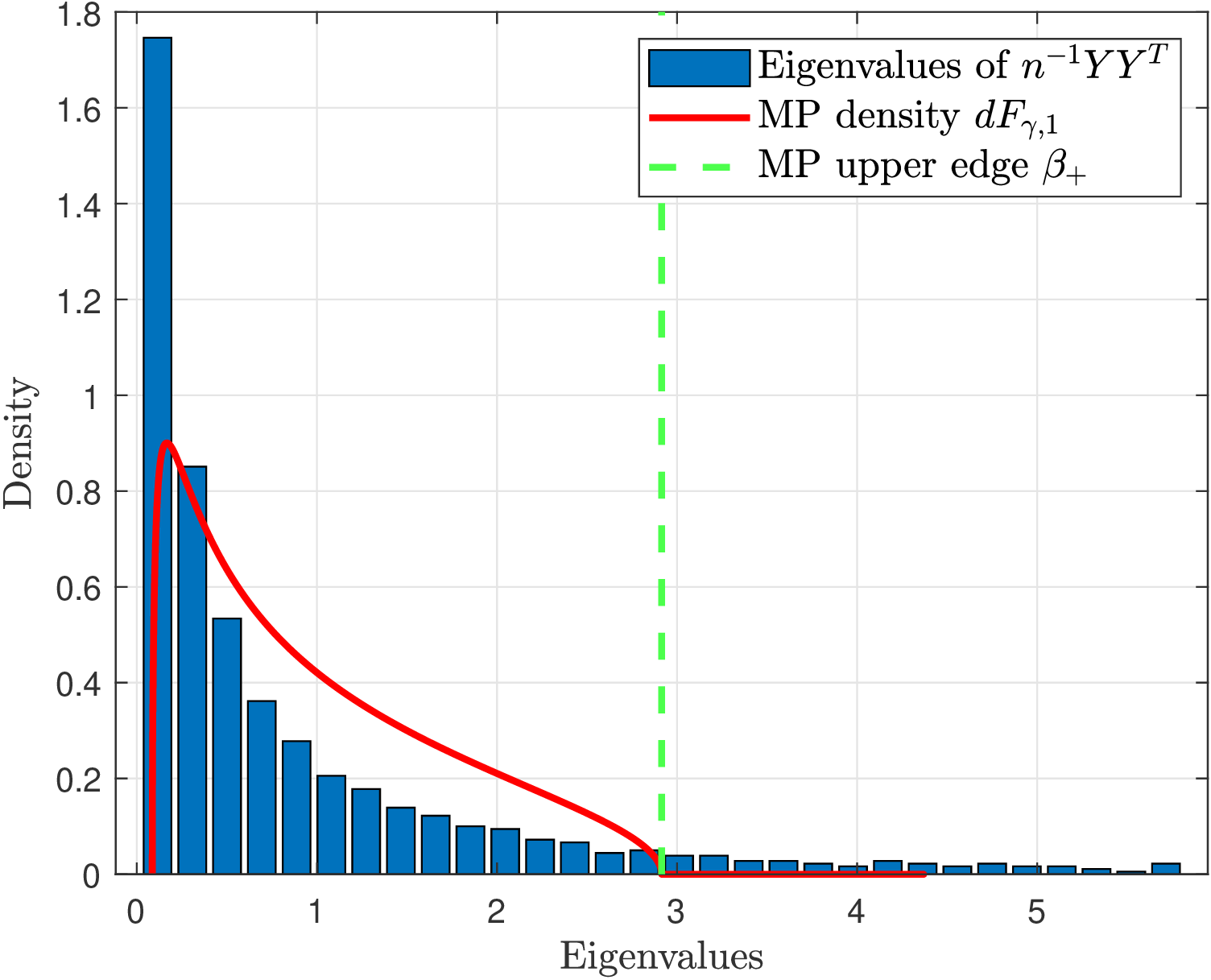}  
    }
    \subfloat[][Eigenvalue density, after biwhitening] 
  	{
    \includegraphics[width=0.35\textwidth]{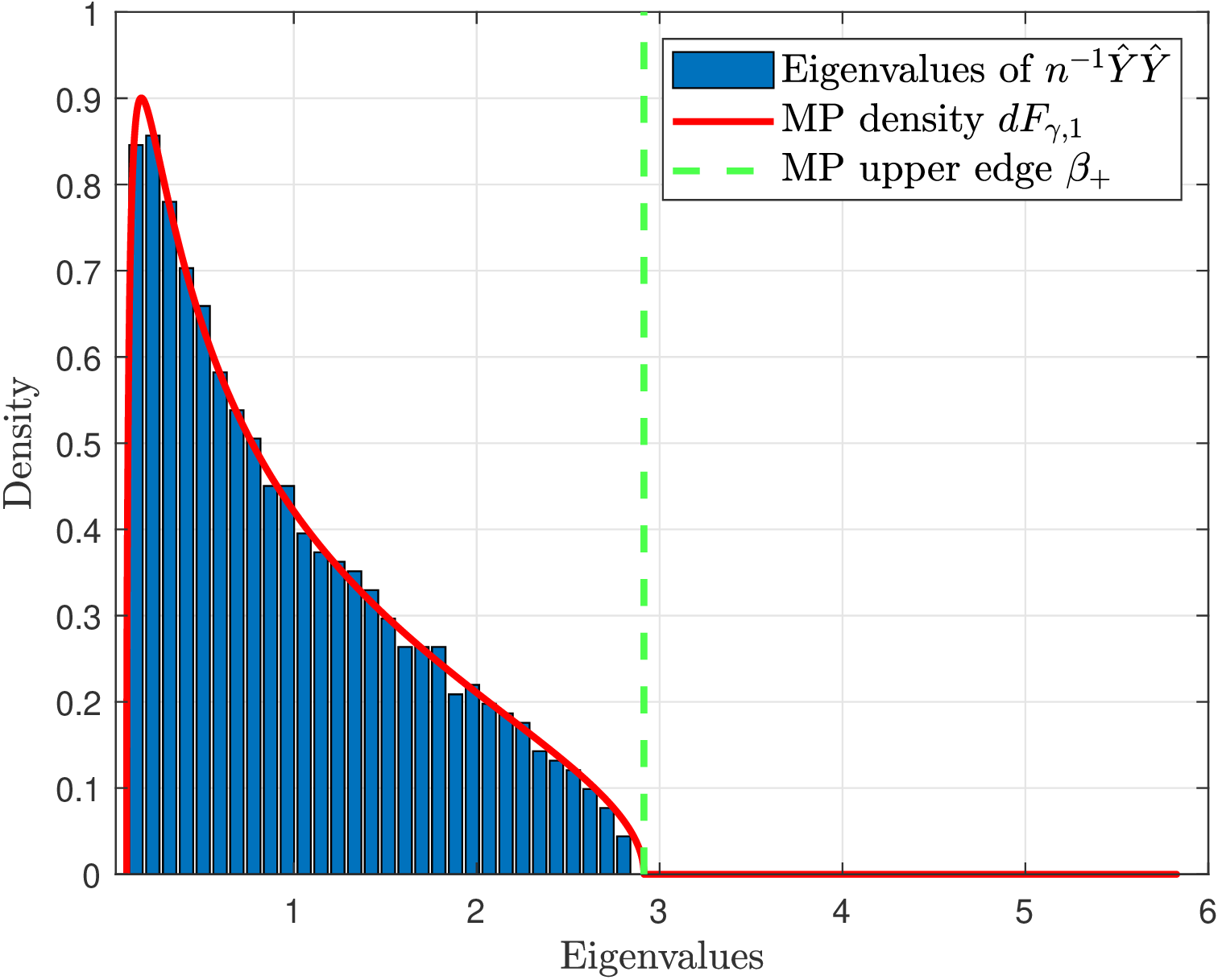} 
    } 

    }
    \caption
    {The spectrum of a simulated negative binomial matrix $Y$ with $m=1000$, $n=2000$, and $r=10$, versus the spectrum of its biwhitened version $D(\hat{\mathbf{u}}) Y D(\hat{\mathbf{v}})$. The number of negative binomial failures for each entry was sampled uniformly at random from $\{1,\ldots,10\}$. To find $\hat{\mathbf{u}}$ and $\hat{\mathbf{v}}$ we solved~\eqref{eq:scaling equations from estimated variances general model} with the variance estimator~\eqref{eq:variance estimator with alpha and beta}, where $\beta=0.3$ and $\alpha$ was chosen according to~\eqref{eq:estimator for alpha}. 
    } \label{fig:negative binomial example rand}
\end{figure}

\section{Proof of Theorem~\ref{thm:Marchenko-Pastur for biwhitened noise}} \label{appendix:proof of theorem on MP law using true variances}
Since the pair $(\mathbf{u}^{(n)},\mathbf{v}^{(n)})$ is defined up to a constant, namely $(a\mathbf{u}^{(n)},a^{-1}\mathbf{v}^{(n)})$ for any $a>0$, we take $a$ such that $\Vert \mathbf{u}^{(n)} \Vert_2 = \Vert \mathbf{v}^{(n)} \Vert_2$.
Then, since $X^{(n)}$ is a positive matrix, applying Lemma 2 in~\cite{landa2020scaling} using $A = X^{(n)}$, $r_i = n$,  $c_j = m$, $x_i = (u_i^{(n)})^2$ and $y_j = (v_j^{(n)})^2$, implies that
\begin{align}
\begin{aligned}
        \sqrt{\frac{n \min_{i,j} X_{i,j}^{(n)}}{m_n (\max_{i,j} X_{i,j}^{(n)} )^2 }} 
        &\leq {(u_i^{(n)})^2} 
        \leq \sqrt{\frac{m_n \max_{i,j} X_{i,j}^{(n)}}{n (\min_{i,j} X_{i,j}^{(n)} )^2 }}, \\
        \sqrt{\frac{n \min_{i,j} X_{i,j}^{(n)}}{m_n (\max_{i,j} X_{i,j}^{(n)} )^2 }} 
        &\leq {(v_j^{(n)})^2} 
        \leq \sqrt{\frac{m_n \max_{i,j} X_{i,j}^{(n)}}{n (\min_{i,j} X_{i,j}^{(n)} )^2 }},
        \end{aligned}
        \label{eq:x_i and y_i bounds}
    \end{align}
for all $i\in [m_n]$ and $j\in [n]$.
According to the assumptions in Theorem~\ref{thm:Marchenko-Pastur for biwhitened noise} we have $\max_{i,j} X_{i,j}^{(n)} \leq C \min_{i,j} X_{i,j}^{(n)}$, and therefore \eqref{eq:x_i and y_i bounds} asserts that 
\begin{equation}
     \frac{1}{C^2 \min_{i,j} X_{i,j}^{(n)}} \leq \frac{\min_{i,j} X_{i,j}^{(n)}}{\left(\max_{i,j} X_{i,j}^{(n)}\right)^2} \leq (u_i^{(n)})^2 (v_j^{(n)})^2 \leq \frac{\max_{i,j} X_{i,j}^{(n)}}{\left(\min_{i,j} X_{i,j}^{(n)}\right)^2} \leq \frac{C^2}{\max_{i,j} X_{i,j}^{(n)}}, \label{eq: u_i v_j squared upper bound}
\end{equation}
for all $i\in [m_n]$ and $j\in [n]$.

We now provide an upper bound on the moments of $\widetilde{\mathcal{E}}_{i,j}^{(n)}$. We can write
\begin{align}
    \mathbb{E} \left[ |\widetilde{\mathcal{E}}_{i,j}^{(n)}|^k \right] 
    &= \mathbb{E} \left[ u_i^k |{\mathcal{E}}_{i,j}^{(n)}|^k  v_j^k\right] 
    \leq  \left(\frac{C^2}{\max_{i,j} X_{i,j}^{(n)}}\right)^{k/2} \mathbb{E} \left[ |{\mathcal{E}}_{i,j}^{(n)}|^k \right].
\end{align}
Recall that $\mathbb{E}[\mathcal{E}_{i,j}^{k}]$ is the $k$'th central moment of the binomial variable $Y_{i,j}$. According to eq. (4.16) in~\cite{johnson2005univariate}, for all $k\geq 2$ we have the recurrence relation
\begin{equation}
    \mathbb{E}[(\mathcal{E}_{i,j}^{(n)})^{k}] = X_{i,j}^{(n)} \sum_{\ell=0}^{k-2} {k-1 \choose \ell} \mathbb{E}[(\mathcal{E}_{i,j}^{(n)})^{\ell}]. \label{eq:Poisson recursive moment formula}
\end{equation}
Therefore, using the fact that $\mathbb{E}[\mathcal{E}_{i,j}^{(n)}]=0$, it follows by induction that for all $k=2,3,\ldots,\infty$
\begin{equation}
    \mathbb{E}[(\mathcal{E}_{i,j}^{(n)})^{k}] = \sum_{\ell=1}^{\lfloor k/2 \rfloor} a_\ell^{(k)} (X_{i,j}^{(n)})^{\ell}, \label{eq:Poisson moment formula}
\end{equation}
for some constant coefficients $\{a_\ell^{(k)}\}_{\ell=1}^{\lfloor k/2 \rfloor}$. Hence, for even values of $k$ we can write
\begin{align}
    & \mathbb{E} \left[ |\widetilde{\mathcal{E}}_{i,j}^{(n)}|^k \right] 
    \leq 
    \left(\frac{C^2}{\max_{i,j} X_{i,j}^{(n)}}\right)^{k/2} \mathbb{E} \left[ {(\mathcal{E}}_{i,j}^{(n)})^k \right] \nonumber \\
    &\leq C^k \sum_{\ell=1}^{ k/2 } a_\ell^{(k)} \frac{(X_{i,j}^{(n)})^{\ell}}{\max_{i,j}(X_{i,j}^{(n)})^{k/2}} 
    \leq C^k \sum_{\ell=1}^{ k/2 } a_\ell^{(k)} c^{\ell - k/2},
\end{align}
where we used the fact that $\max_{i,j} X_{i,j}^{(n)} > c$ for all $n$.
Eventually, for all values of $k$ and $n$ we have
\begin{equation}
    \mathbb{E} \left[ |\widetilde{\mathcal{E}}^{(n)}_{i,j}|^k \right] \leq \sqrt{\mathbb{E}[|\widetilde{\mathcal{E}}^{(n)}_{i,j}|^{2k}]} \leq C^{k} \sqrt{\sum_{\ell=1}^{ k } a_\ell^{(k)} c^{\ell - k}} := \widetilde{C}_k. \label{eq:moment bound scaled centered Poisson}
\end{equation}

To prove the MP law, we apply Theorem 8.2 in~\cite{girko2001theory} to the matrix $\xi^{(n)} = n^{-1/2} \widetilde{\mathcal{E}}^{(n)}$. To that end, note that the entries of $\xi^{(n)} = (n^{-1/2} \widetilde{\mathcal{E}}^{(n)}_{i,j})_{i\in[m_n],\;j\in[n]}$ are independent, and $\xi^{(n)}$ satisfies $\sum_{i=1}^{m_n} \xi^{(n)}_{i,j} = m_n/n$, $\sum_{j=1}^{n} \xi^{(n)}_{i,j} = 1$ for all $i\in[m_n]$ and $j\in[n]$. We now establish the required Lindeberg condition. Observe that
\begin{align}
    &\mathbb{E}\left[(\xi_{i,j}^{(n)})^2 \mathbbm{1}(|\xi_{i,j}^{(n)}| > \tau) \right] = \frac{1}{n}\mathbb{E}\left[(\widetilde{\mathcal{E}}_{i,j}^{(n)})^2 \mathbbm{1}(|\widetilde{\mathcal{E}}_{i,j}^{(n)}| > \tau \sqrt{n}) \right] 
    = \frac{1}{n}\int_{|t|>\tau \sqrt{n}} t^2 d\mu_{\widetilde{\mathcal{E}}_{i,j}^{(n)}} (t) \nonumber \\
    &< \frac{1}{n} \int_{|t|>\tau \sqrt{n}} \frac{t^4}{\tau^2 n} d\mu_{\widetilde{\mathcal{E}}_{i,j}^{(n)}} (t) \leq \frac{1}{\tau^2 n^2} \mathbb{E}[(\widetilde{\mathcal{E}}_{i,j}^{(n)})^4] \leq \frac{\widetilde{C}_4}{\tau^2 n^2}, 
\end{align}
where $d\mu_{\widetilde{\mathcal{E}}_{i,j}^{(n)}}$ is the probability density of $\widetilde{\mathcal{E}}_{i,j}^{(n)}$, and we used~\eqref{eq:moment bound scaled centered Poisson} to get the last inequality.
We then have
\begin{align}
    &\lim_{n\rightarrow\infty} \max_{i\in [m_n]\;, j\in [n]} \left[ \sum_{j=1}^n \mathbb{E}\left[(\xi_{i,j}^{(n)})^2 \mathbbm{1}(|\xi_{i,j}^{(n)}| > \tau) \right] + \sum_{i=1}^{m_n} \mathbb{E}\left[(\xi_{i,j}^{(n)})^2 \mathbbm{1}(|\xi_{i,j}^{(n)}| > \tau) \right] \right] \nonumber \\
    &< \lim_{n\rightarrow\infty} \max_{i\in [m_n]\;, j\in [n]} \left[ \frac{\widetilde{C}_4}{\tau^2} \left(\frac{1}{n} + \frac{m_n}{n^2}\right)\right] = 0,
\end{align}
for every $\tau>0$. Then, according to Theorem 8.2 in~\cite{girko2001theory}, we have almost surely for all $\tau$ that
\begin{equation}
    \lim_{n\rightarrow\infty} \left| F_{\widetilde{\Sigma}^{(n)}}(\tau) - F_{\frac{m_n}{n},1}(\tau) \right| = \lim_{n\rightarrow\infty} \left| F_{\widetilde{\Sigma}^{(n)}}(\tau) - F_{\gamma,1}(\tau) \right| = 0.
\end{equation}
We note that Theorem 8.2 in~\cite{girko2001theory} is actually stated in terms of a distribution associated with the solution to a certain Dyson equation. To see that this distribution is in fact the MP distribution, we refer the reader to~\cite{marvcenko1967distribution} where the same equation is analyzed and is shown to provide the distribution whose density has the explicit form~\eqref{eq:MP density}.

Next, to prove the convergence of the largest eigenvalue to $\beta_+$, we apply Theorem 2.4 part II in~\cite{alt2017local} to the matrix $\xi^{(n)} = n^{-1/2} \widetilde{\mathcal{E}}^{(n)}$ (as $X$ in~\cite{alt2017local}). To that end, we need to show that the Conditions (A) -- (D) in~\cite{alt2017local} hold, which we consider next. To show Condition (A) in~\cite{alt2017local}, we have
\begin{equation}
    \mathbb{E}[(\xi^{(n)}_{i,j})^2] = \frac{\mathbb{E}[(\widetilde{\mathcal{E}}^{(n)}_{i,j})^2]}{n} \leq \frac{\widetilde{C}_2}{n} \leq \left(\frac{\widetilde{C}_2}{n + m_n}\right) (1 + \frac{m_n}{n}) \leq \frac{\widetilde{C}_2 (1+2\gamma) }{n + m_n},
\end{equation}
for all sufficiently large $n$, where we also used~\eqref{eq:moment bound scaled centered Poisson}. 
To show Condition (B) in~\cite{alt2017local}, observe that
\begin{equation}
    \mathbb{E}[(\xi^{(n)}_{i,j})^2] = \frac{\mathbb{E}[(\widetilde{\mathcal{E}}^{(n)}_{i,j})^2]}{n} = \frac{(u_i^{(n)})^2 (v_j^{(n)})^2 X_{i,j}^{(n)}}{n} \geq \frac{X_{i,j}^{(n)}}{n C^2 \min_{i,j} X_{i,j}^{(n)}} \geq \frac{1}{n C^2} \geq \frac{C^{-2}(1+0.5 \gamma)}{n + m_n}, \label{eq:condition B proof}
\end{equation}
for all sufficiently large $n$, where we also used~\eqref{eq: u_i v_j squared upper bound}. Note that~\eqref{eq:condition B proof} immediately establishes Condition (B) as explained in Remark 2.8 in~\cite{alt2017local}.
Last, Condition (C) in~\cite{alt2017local} follows by combining~\eqref{eq:condition B proof} with~\eqref{eq:moment bound scaled centered Poisson}, and Condition (D) in~\cite{alt2017local} follows from our asymptotic setting where $m_n/n\rightarrow\gamma$.

Then, according to Theorem 2.4 part II in~\cite{alt2017local}, 
\begin{equation}
    \operatorname{Pr}  \left\{ \lambda_1\{n^{-1} \widetilde{\mathcal{E}}^{(n)} (\widetilde{\mathcal{E}}^{(n)})^T \} > \beta_+ + \varepsilon_\star \right\} \rightarrow 0, \label{eq:no eigenvalues above beta_+}
\end{equation}
for any $\varepsilon_{*}>0$, where $\beta_{+}$ is the upper edge of the support of the MP density~\eqref{eq:MP density}. 
We mention that Theorem 2.4 part II in~\cite{alt2017local} is actually stated in terms of the support of a density satisfying an appropriate Dyson equation. 
To see that this density is in fact the MP density, see the short proof of Theorem 8.2 in~\cite{girko2001theory}, where it is shown that the Dyson equation that governs the case of general variances reduces to the Dyson equation giving rise to the MP density if the average variance in each row and in each column of the random matrix is $1$.

Last, since the limiting spectral distribution of $n^{-1} \widetilde{\mathcal{E}}^{(n)} (\widetilde{\mathcal{E}}^{(n)})^T$, given by $F_{\widetilde{\Sigma}^{(n)}}(\tau)$, converges almost surely to the MP distribution $F_{\gamma,1}(\tau)$, which is strictly positive for any $\tau\in (\beta_-,\beta_+)$, then we also have that
\begin{equation}
    \operatorname{Pr}  \left\{ \lambda_1\{n^{-1} \widetilde{\mathcal{E}}^{(n)} (\widetilde{\mathcal{E}}^{(n)})^T \} < \beta_+ - \varepsilon_\star \right\} \rightarrow 0,
\end{equation}
for any $\varepsilon_{*}>0$, which together with~\eqref{eq:no eigenvalues above beta_+} establishes that  $\lambda_1\{\widetilde{\Sigma}^{(n)} \} \overset{p}{\longrightarrow} \beta_+ = (1 + \sqrt{\gamma})^2$.

\section{Proof of Theorem~\ref{thm:r_tilde no false detect}} \label{appendix:proof of theorem on rank estimator}
Suppose in negation that $\operatorname{Pr} \{ r^{(n)} < \widetilde{r}_\varepsilon^{(n)}\} $ does not converge to $0$ as $n\rightarrow\infty$. Then, there exists a sequence $\{n_k\}_{k\geq 1}$ with $n_k\underset{k\rightarrow\infty}{\longrightarrow} \infty$ and a constant $\beta > 0$, such that
\begin{equation}
    \operatorname{Pr} \{ r^{(n_k)}  < \widetilde{r}_\varepsilon^{(n_k)}\} \geq \beta,
\end{equation}
for all $k\geq 1$.
In addition, according to the definition of $\widetilde{r}_\varepsilon^{(n)}$ we have
\begin{equation}
    \lambda_{\widetilde{r}_\varepsilon^{(n_k)}} \{ n_k^{-1} \widetilde{Y}^{(n)} (\widetilde{Y}^{(n)})^T \} > \left(1 + \sqrt{\frac{m_{n_k}}{n_k}}\right)^2 + \varepsilon,
\end{equation}
for all $k\geq 1$.
We can now write
\begin{align}
    0 < \beta &\leq \operatorname{Pr} \{ r^{(n_k)} < \widetilde{r}_\varepsilon^{(n_k)}\} 
    = \operatorname{Pr} \{ r^{(n_k)} + 1 \leq \widetilde{r}_\varepsilon^{(n_k)}\} \nonumber \\
    &\leq \operatorname{Pr} \left\{ \lambda_{\widetilde{r}_\varepsilon^{(n_k)}} \{ n_k^{-1} \widetilde{Y}^{(n_k)} (\widetilde{Y}^{(n_k)})^T \} \leq \lambda_{r^{(n_k)}+1} \{ n_k^{-1} \widetilde{Y}^{(n_k)} (\widetilde{Y}^{(n_k)})^T \} \right\} \nonumber \\
    &\leq \operatorname{Pr} \left\{ \left(1 + \sqrt{\frac{m_{n_k}}{n_k}}\right)^2 +\varepsilon < \lambda_{r^{(n_k)}+1} \{ n_k^{-1} \widetilde{Y}^{(n_k)} (\widetilde{Y}^{(n_k)})^T \} \right\}. \label{eq:beta lower bound}
\end{align}
According to Theorem 3.3.16 in~\cite{horn1994topics} we have
\begin{equation}
    \sigma_{r^{(n_k)}+1}\{\widetilde{Y}^{(n)}\} =  \sigma_{r^{(n_k)}+1}\{\widetilde{X}^{(n_k)} + \widetilde{\mathcal{E}}^{(n_k)} \} \leq \sigma_{r^{(n_k)}+1}\{\widetilde{X}^{(n_k)}\} + \sigma_{1}\{\mathcal{E}^{(n_k)}\} = \sigma_{1}\{\mathcal{E}^{(n_k)}\},
\end{equation}
where $\sigma_i\{\widetilde{Y}\}$ is the $i$'th largest singular value of $\widetilde{Y}$, and we used the fact that $\operatorname{rank}\{\widetilde{X}^{(n_k)}\} = \operatorname{rank}\{{X}^{(n_k)} \} = r^{(n_k)}$. Therefore, we have
\begin{equation}
    \lambda_{r^{(n_k)}+1} \{ n_k^{-1} \widetilde{Y}^{(n_k)} (\widetilde{Y}^{(n_k)})^T\} = n_k^{-1} (\sigma_{r^{(n_k)}+1}\{\widetilde{Y}^{(n_k)}\})^2 \leq n_k^{-1} (\sigma_{1}\{\mathcal{E}^{(n_k)}\})^2 = \Vert \widetilde{\Sigma}^{(n_k)} \Vert_2.
\end{equation}
Combining the above with~\eqref{eq:beta lower bound} we obtain
\begin{equation}
    0 < \beta \leq \operatorname{Pr} \left\{ \left(1 + \sqrt{\frac{m_{n_k}}{n_k}}\right)^2 +\varepsilon < \Vert \widetilde{\Sigma}^{(n_k)} \Vert_2 \right\} \leq \operatorname{Pr} \left\{ \left(1 + \sqrt{\gamma}\right)^2 + \frac{\varepsilon}{2} < \Vert \widetilde{\Sigma}^{(n_k)} \Vert_2 \right\}, \label{eq:beta lowe bound 2}
\end{equation}
for all sufficiently large $k$, where we used the fact that $m_{n_k}/n_k \rightarrow \gamma$. However, from Theorem~\ref{thm:Marchenko-Pastur for biwhitened noise} we know that $\operatorname{Pr}\{\Vert \widetilde{\Sigma}^{(n_k)} \Vert_2 > \varepsilon/2 + \left(1 + \sqrt{\gamma}\right)^2\} \rightarrow 0$ as $n\rightarrow\infty$ for any $\varepsilon > 0$, which is a contradiction to~\eqref{eq:beta lowe bound 2}.

\section{Proof of Lemma~\ref{lem:convergence of scaling factors for estimated variances}} \label{appendix:proof of lemma on convergence of scaling factors}
To prove Lemma~\ref{lem:convergence of scaling factors for estimated variances}, we rely on Theorem 3 in~\cite{landa2020scaling}. However, since this theorem requires a random matrix with bounded variables (from above and from below away from zero), we define a truncated version of $Y_{i,j}\sim \operatorname{Poisson}(X_{i,j})$ that retains its mean as follows. For any $m$ and $n$, let $\alpha > 0$, and define the random variable $\{\overline{Y}_{i,j}\}_{i\in [m],\;j\in[n]}$ according to
\begin{equation}
    (\overline{Y}_{i,j}|Y_{i,j} = y) = 
    \begin{dcases}
    y, &  X_{i,j} / \alpha \leq y \leq \alpha X_{i,j} , \\
    X_{i,j} \frac{\operatorname{Pr}\{Y_{i,j} < X_{i,j}/\alpha  - 1\}}{\operatorname{Pr}\{ Y_{i,j} < X_{i,j}/\alpha \}}, & y < X_{i,j} / \alpha \\
    X_{i,j} \frac{\operatorname{Pr}\{Y_{i,j} > \alpha X_{i,j} - 1\}}{\operatorname{Pr}\{ Y_{i,j} > \alpha X_{i,j} \}}, & y > \alpha X_{i,j}.
    \end{dcases} \label{eq:Y_overline def}
\end{equation}
We then have the following lemma.
\begin{lem}
$\{\overline{Y}_{i,j}\}_{i\in[m],\;j\in[n]}$ are independent random variables with $\mathbb{E}[\overline{Y}_{i,j}] = X_{i,j}$, and
\begin{equation}
     \min \left\{\frac{1}{\alpha} , 1 - \frac{X_{i,j}}{X_{i,j} + \lceil X_{i,j}/\alpha \rceil - 1} \right\} \leq \frac{\overline{Y}_{i,j}}{X_{i,j}} \leq  1 + \alpha +  X_{i,j}^{-1} . \label{eq:Y_overline bounds}
\end{equation}
\end{lem}
\begin{proof}
The fact that $\overline{Y}_{i,j}$ are independent follows from their definition. We now prove that $\mathbb{E}[\overline{Y}_{i,j}] = X_{i,j}$. We can write
\begin{align}
    \mathbb{E}[\overline{Y}_{i,j}] &= \sum_{k = \lceil X_{i,j}/\alpha \rceil}^{\lfloor \alpha X_{i,j} \rfloor} k \operatorname{Pr}\{Y_{i,j} = k\} + X_{i,j} {\operatorname{Pr}\{Y_{i,j} < X_{i,j}/\alpha  - 1\}} + X_{i,j} {\operatorname{Pr}\{Y_{i,j} > \alpha X_{i,j} - 1\}} \nonumber \\
    &= \sum_{k = 0}^{\infty} k \operatorname{Pr}\{Y_{i,j} = k\} - \sum_{k = 0}^{\lceil X_{i,j}/\alpha \rceil - 1} k \operatorname{Pr}\{Y_{i,j} = k\} - \sum_{k = \lfloor \alpha X_{i,j} \rfloor + 1}^{\infty} k \operatorname{Pr}\{Y = k\} \nonumber \\ 
    &+ X {\operatorname{Pr}\{Y_{i,j} < X_{i,j}/\alpha  - 1\}} + X {\operatorname{Pr}\{Y_{i,j} > \alpha X_{i,j} - 1\}} \nonumber \\
    &= X_{i,j} - \sum_{k = 1}^{\lceil X_{i,j}/\alpha \rceil - 1} \frac{X_{i,j}^k e^{-X_{i,j}}}{(k-1)!} - \sum_{k = \lfloor \alpha X_{i,j} \rfloor + 1}^{\infty} \frac{X_{i,j}^k e^{-X_{i,j}}}{(k-1)!} \nonumber \\ 
    &+ X_{i,j} {\operatorname{Pr}\{Y_{i,j} < X_{i,j}/\alpha  - 1\}} + X_{i,j} {\operatorname{Pr}\{Y_{i,j} > \alpha X_{i,j} - 1\}} \nonumber \\
    &= X_{i,j} - X_{i,j} \sum_{k = 0}^{\lceil X_{i,j}/\alpha \rceil - 2} \operatorname{Pr}\{Y_{i,j} = k\} - X_{i,j} \sum_{k = \lfloor \alpha X_{i,j} \rfloor }^{\infty} \operatorname{Pr}\{Y_{i,j} = k\} \nonumber \\ 
    &+ X_{i,j} {\operatorname{Pr}\{Y_{i,j} < X_{i,j}/\alpha  - 1\}} + X_{i,j} {\operatorname{Pr}\{Y_{i,j} > \alpha X_{i,j} - 1\}} \nonumber \\
    & = X_{i,j}.
\end{align}

To show~\eqref{eq:Y_overline bounds}, observe that
\begin{align}
    &X_{i,j} \geq X_{i,j} \frac{\operatorname{Pr}\{Y < X_{i,j}/\alpha  - 1\}}{\operatorname{Pr}\{ Y < X_{i,j}/\alpha \}}
    = X_{i,j} \frac{\operatorname{Pr}\{Y < X_{i,j}/\alpha\} - \operatorname{Pr}\{Y = \lceil X_{i,j}/\alpha \rceil - 1\}}{\operatorname{Pr}\{ Y < X_{i,j}/\alpha \}} \nonumber \\
    &= X_{i,j} \left( 1 - \frac{\operatorname{Pr}\{Y = \lceil X_{i,j}/\alpha \rceil - 1\}}{\operatorname{Pr}\{ Y < X_{i,j}/\alpha \}} \right) 
    \geq X_{i,j} \left( 1 - \frac{\operatorname{Pr}\{Y = \lceil X_{i,j}/\alpha \rceil - 1\}}{\operatorname{Pr}\{Y = \lceil X_{i,j}/\alpha \rceil - 1\} + \operatorname{Pr}\{Y = \lceil X_{i,j}/\alpha \rceil - 2\}} \right) \nonumber \\
    &= X_{i,j} \left( 1 - \frac{\operatorname{Pr}\{Y = \lceil X_{i,j}/\alpha \rceil - 1\}}{\operatorname{Pr}\{Y = \lceil X_{i,j}/\alpha \rceil - 1\} \left( 1 + ( \lceil X_{i,j}/\alpha \rceil - 1)/X_{i,j}\right)} \right)
    = X_{i,j} \left(1 - \frac{X_{i,j}}{X_{i,j} + \lceil X_{i,j}/\alpha \rceil - 1}\right), \label{eq:Y_overline upper bound proof}
    \end{align}
where we used the property of the Poisson distribution $\operatorname{Pr}\{Y_{i,j} = k\} = \operatorname{Pr}\{Y_{i,j} = k+1\} (k+1)/X_{i,j}$ (which follows immediately from its probability-mass function). Similarly, we have
\begin{align}
    &X_{i,j} \leq X_{i,j} \frac{\operatorname{Pr}\{Y_{i,j} > \alpha X_{i,j}  - 1\}}{\operatorname{Pr}\{ Y_{i,j} > \alpha X_{i,j} \}}
    = X_{i,j} \frac{\operatorname{Pr}\{Y_{i,j} > \alpha X_{i,j}\} + \operatorname{Pr}\{Y_{i,j} = \lfloor \alpha X_{i,j} \rfloor \}}{\operatorname{Pr}\{ Y_{i,j} > \alpha X_{i,j} \}} \nonumber \\
    &= X_{i,j} \left( 1 + \frac{\operatorname{Pr}\{Y_{i,j} = \lfloor \alpha X_{i,j} \rfloor\}}{\operatorname{Pr}\{ Y_{i,j} > \alpha X_{i,j} \}} \right) 
    \leq X_{i,j} \left( 1 + \frac{\operatorname{Pr}\{Y_{i,j} = \lfloor \alpha X_{i,j} \rfloor + 1\} (\lfloor \alpha X_{i,j} \rfloor + 1) / X_{i,j}}{\operatorname{Pr}\{ Y_{i,j} > \alpha X_{i,j} \}} \right)  \nonumber \\
    &< X_{i,j} \left( 1 + \frac{\lfloor \alpha X_{i,j} \rfloor + 1}{ X_{i,j}} \right)
    \leq X_{i,j} \left( 1 + \alpha +  X_{i,j}^{-1} \right), \label{eq:Y_overline lower bound proof}
\end{align}
where we used the fact that $\operatorname{Pr}\{Y = \lfloor \alpha X_{i,j} \rfloor + 1\} < \operatorname{Pr}\{ Y > \alpha X_{i,j} \}$, and again the property $\operatorname{Pr}\{Y_{i,j} = k\} = \operatorname{Pr}\{Y_{i,j} = k+1\} (k+1)/X_{i,j}$.
Combining~\eqref{eq:Y_overline lower bound proof} and~\eqref{eq:Y_overline upper bound proof} together with the definition of $\overline{Y}_{i,j}$ proves~\eqref{eq:Y_overline bounds}.
\end{proof}

Now, let $\overline{Y}^{(n)}_{i,j}$ be as in~\eqref{eq:Y_overline def} when replacing $Y_{i,j}$ and $X_{i,j}$ with $Y_{i,j}^{(n)}$ and $X_{i,j}^{(n)}$, respectively.
Observe that since $X_{i,j}^{(n)} \rightarrow \infty$ as $n\rightarrow \infty$ (by the conditions in Lemma~\ref{lem:convergence of scaling factors for estimated variances}) and according to~\eqref{eq:Y_overline bounds}, for any $\alpha>0$ there exist constants $0 < \alpha_1 \leq \alpha_2$ such that for all sufficiently large $n$:
\begin{equation}
     \alpha_1 X_{i,j}^{(n)}\leq \overline{Y}_{i,j}^{(n)} \leq \alpha_2 X_{i,j}^{(n)}. \label{eq:Y_overline upper and lower asymptotic bounds}
\end{equation}
Next we will show that $\overline{Y}_{i,j}^{(n)} = Y_{i,j}^{(n)}$ for all $i\in [m]$ and $j\in [n]$ with probability tending $1$ as $n\rightarrow \infty$. To that end, according to Proposition 11.15 in~\cite{goldreich2017introduction}, each variable $Y_{i,j}\sim \operatorname{Poisson}(X_{i,j})$ admits the following sub-exponential tail bound:
\begin{equation}
    \operatorname{Pr}\{|Y_{i,j} - X_{i,j}|\geq t\} \leq \operatorname{exp} \left( - c^{'} \frac{ t^2 }{{X_{i,j}} + t} \right),
\end{equation}
for all $t > 0$ and some universal constant $c^{'}>0$. For any index $n$ , taking $t = (\alpha - 1) X_{i,j}^{(n)}$ for $\alpha > 1$ gives
\begin{equation}
    \operatorname{Pr}\{Y_{i,j}^{(n)} \geq \alpha X_{i,j}^{(n)}\} \leq \operatorname{exp} \left( -  \frac{c^{'}(\alpha-1)^2 X_{i,j}^{(n)}}{\alpha} \right) 
    \leq \operatorname{exp} \left( - \frac{c c^{'}(\alpha-1)^2}{C \alpha} (\log n)^{1 + \epsilon} \right).
\end{equation}
where we used the fact that $c (\log n)^{1+\epsilon} \leq \max_{i,j} X_{i,j}^{(n)} \leq C \min_{i,j} X_{i,j}^{(n)}$.
Analogously, taking $t = (1 - \alpha^{-1}) X_{i,j}^{(n)}$ for $\alpha>1$ gives
\begin{equation}
    \operatorname{Pr}\{Y_{i,j}^{(n)} \leq  \frac{X_{i,j}^{(n)}}{\alpha}\} \leq  \operatorname{exp} \left( - \frac{c^{'} (1-\alpha^{-1})^2}{2-\alpha^{-1}} X_{i,j}^{(n)} \right) \leq \operatorname{exp} \left( - \frac{c^{'} (1-\alpha^{-1})^2}{C(2-\alpha^{-1})} (\log n)^{1+\epsilon} \right).
\end{equation}
Therefore, using the union bound and the fact that $m_n/n \rightarrow \gamma$ we have that
\begin{align}
    &\operatorname{Pr}\left\{ \overline{Y}^{(n)} \neq Y^{(n)} \right\} 
    = \operatorname{Pr}\left\{ \cup_{i,j} \{ \overline{Y}_{i,j}^{(n)} \neq Y_{i,j}^{(n)}\} \right\}
    \nonumber \\
    &\leq \sum_{i\in[m_n],\; j\in[n]} \left[  \operatorname{Pr}\{Y_{i,j}^{(n)} \geq \alpha X_{i,j}^{(n)}\} +  \operatorname{Pr}\{Y_{i,j}^{(n)} \leq  \frac{X_{i,j}^{(n)}}{\alpha}\} \right]
    \\
    &\leq m_n n \left[ \operatorname{exp} \left( - \frac{c c^{'}(\alpha-1)^2}{C \alpha} (\log n)^{1 + \epsilon} \right) + \operatorname{exp} \left( - \frac{c^{'} (1-\alpha^{-1})^2}{C(2-\alpha^{-1})} (\log n)^{1+\epsilon} \right) \right] \underset{n \rightarrow\infty}{\longrightarrow} 0. \label{eq:probability that Y_overline is not Y}
\end{align}

We can now apply Theorem 3 in~\cite{landa2020scaling} to the random matrix $\overline{Y}^{(n)} = (\overline{Y}_{i,j}^{(n)})_{i\in [m_n],\; j\in [n]}$ for sufficiently large $n$ using the bounds in~\eqref{eq:Y_overline upper and lower asymptotic bounds}. Then, together with~\eqref{eq:probability that Y_overline is not Y} it follows that for any $\delta \in (0, 1]$, with probability at least
\begin{equation}
    1 - 2 m_n \operatorname{exp}\left( -\frac{\delta^2 n}{ C_p^2 } \right) - 2 n \operatorname{exp}\left(-\frac{\delta^2 m_n}{ C_p^2}\right) - \operatorname{Pr}\left\{ \overline{Y}^{(n)} \neq Y^{(n)} \right\} , \label{eq:probability of tail bounds on x_tilde and y_tilde}
\end{equation}
there exists a pair of positive random vectors $(\widetilde{\mathbf{x}}^{(n)},\widetilde{\mathbf{y}}^{(n)})$ that scales $Y^{(n)}$ to row sums $n$ and column sums $m_n$, such that for all $i,j$:
    \begin{equation}
        \frac{\vert \widetilde{x}_i^{(n)} - (u_i^{(n)})^2 \vert}{(u_i^{(n)})^2} \leq  
      { C_e \delta } , \qquad \qquad
     \frac{\vert \widetilde{y}^{(n)}_j - (v_j^{(n)})^2 \vert}{(v_j^{(n)})^2} \leq{ C_e \delta }, \label{eq:x_tilde_i and y_tilde_j relative error bounds}
    \end{equation}
where
\begin{equation}
    C_p = \sqrt{2}\left(\frac{\alpha_2}{\alpha_1}\right)^2, \qquad \qquad C_e = 1 + 2\left(\frac{\alpha_2}{\alpha_1} \right)^{7/2}. \label{eq:C_p and C_e expressions}
\end{equation}
Taking any $\alpha > 1$ and $\delta = C_p \sqrt{2 \max\{1,\gamma^{-1}\} \log n /n}$, we obtain that there exists $C^{'}>0$ such that
\begin{equation}
    \frac{\vert \widetilde{x}_i^{(n)} - (u_i^{(n)})^2 \vert}{(u_i^{(n)})^2} \leq  
      C^{'} \sqrt{\frac{\log n}{n} }, \qquad \qquad
     \frac{\vert \widetilde{y}_j^{(n)} - (v_j^{(n)})^2 \vert}{(v_j^{(n)})^2} \leq C^{'} \sqrt{\frac{\log n}{n} },
\end{equation}
for all $i,j$, with probability that tends to $1$ as $n\rightarrow\infty$. Note that we can always find a constant $a_{n}>0$ such that $\widetilde{x}_i^{(n)} = a_{n}^2 (\hat{u}_i^{(n)})^2$ and $\widetilde{y}_j^{(n)} = a_n^{-2} (\hat{v}_j^{(n)})^2$ for all $i,j$. Therefore, we have
\begin{equation}
    {\left\vert \left(\frac{a_n \hat{u}_i^{(n)}}{u_i^{(n)}}\right)^2 - 1 \right\vert} \leq  
      C^{'} \sqrt{\frac{\log n}{n} } , \qquad \qquad
     {\left\vert \left(\frac{a_n^{-1} \hat{v}_j^{(n)}}{v_j^{(n)}}\right)^2 - 1 \right\vert} \leq C^{'} \sqrt{\frac{\log n}{n} },
\end{equation}
for all $i,j$, with probability that tends to $1$ as $n\rightarrow\infty$. Using the Taylor expansion of the function $\sqrt{1 + \varepsilon}$ around $\varepsilon=0$, we get
\begin{equation}
    {\left\vert \frac{a_n \hat{u}_i^{(n)}}{u_i^{(n)}} - 1 \right\vert} \leq  
      \widetilde{C} \sqrt{\frac{\log n}{n} } , \qquad \qquad
     {\left\vert\frac{a_n^{-1} \hat{v}_j^{(n)}}{v_j^{(n)}} - 1 \right\vert} \leq \widetilde{C} \sqrt{\frac{\log n}{n} },
\end{equation}
for all $i,j$ and some constant $\widetilde{C}>0$, with probability that tends to $1$ as $n\rightarrow\infty$.

\section{Proof of Theorem~\ref{thm:Marchenko-Pastur for biwhitened noise from estimated variances}} \label{appendix:proof of theorem on MP law using estimated variances}
Let us denote $\mathbf{e}_1^{(n)} = a_n\hat{\mathbf{u}}^{(n)} - \mathbf{u^{(n)}}$ and $\mathbf{e}_2^{(n)} = a_n^{-1} \hat{ \mathbf{v}}^{(n)} - \mathbf{v}^{(n)}$. We can now write
\begin{align}
   &\frac{1}{\sqrt{n}}\Vert \widetilde{\mathcal{E}}^{(n)} - \hat{ \mathcal{E}}^{(n)} \Vert_2 = \frac{1}{\sqrt{n}}\Vert D(\mathbf{u}^{(n)}) \mathcal{E}^{(n)} D(\mathbf{v}^{(n)}) - D(\hat{\mathbf{u}}^{(n)}) \mathcal{E}^{(n)} D(\hat{\mathbf{v}}^{(n)}) \Vert_2 \nonumber \\
   &= \frac{1}{\sqrt{n}}\Vert D(\mathbf{u}^{(n)}) \mathcal{E}^{(n)} D(\mathbf{v}^{(n)}) - D(a_n \hat{\mathbf{u}}^{(n)}) \mathcal{E}^{(n)} D(a_n^{-1} \hat{\mathbf{v}}^{(n)}) \Vert_2 \nonumber \\
   &\leq \frac{1}{\sqrt{n}}\Vert D(\mathbf{u}^{(n)}) \mathcal{E}^{(n)} D(\mathbf{e}_2^{(n)}) \Vert_2 + \frac{1}{\sqrt{n}}\Vert D(\mathbf{e}_1^{(n)}) \mathcal{E}^{(n)} D(\mathbf{v}^{(n)}) \Vert_2 + \frac{1}{\sqrt{n}}\Vert D(\mathbf{e}_1^{(n)}) \mathcal{E}^{(n)} D(\mathbf{e}_2^{(n)}) \Vert_2 \nonumber \\
   &\leq \frac{\Vert \mathcal{E}^{(n)} \Vert_2}{\sqrt{n}}  \left(\Vert D(\mathbf{u}^{(n)}) \Vert_2 \cdot  \Vert D(\mathbf{e}_2^{(n)}) \Vert_2 + \Vert D(\mathbf{v}^{(n)}) \Vert_2 \cdot \Vert D(\mathbf{e}_1^{(n)}) \Vert_2 + \Vert D(\mathbf{e}_1^{(n)}) \Vert_2 \cdot \Vert D(\mathbf{e}_2^{(n)}) \Vert_2\right). \label{eq:spectral error initial bound}
\end{align}
According to Lemma~\ref{lem:convergence of scaling factors for estimated variances} we have
\begin{align}
    &\Vert D(\mathbf{e}_1^{(n)}) \Vert_2 = \max_j |a_n^{-1} \hat{v}_j^{(n)} - v_j^{(n)} | \leq \widetilde{C} \sqrt{\frac{\log n}{n}} \max_j v_j^{(n)} \label{eq:e_1 bound}\\
    &\Vert D(\mathbf{e}_2^{(n)}) \Vert_2 = \max_i |a_n \hat{u}_i^{(n)} - u_i^{(n)} | \leq \widetilde{C} \sqrt{\frac{\log n}{n}} \max_i u_i^{(n)} \label{eq:e_2 bound},
\end{align}
with probability that tends to $1$ as $n\rightarrow\infty$.
Therefore, substituting~\eqref{eq:e_1 bound} and~\eqref{eq:e_2 bound} into~\eqref{eq:spectral error initial bound} while utilizing~\eqref{eq: u_i v_j squared upper bound} gives
\begin{equation}
    \frac{1}{\sqrt{n}}\Vert \widetilde{\mathcal{E}}^{(n)} - \hat{ \mathcal{E}}^{(n)} \Vert_2  \leq \frac{\Vert \mathcal{E}^{(n)} \Vert_2}{\sqrt{n}} \left(  \frac{2 \widetilde{C} C^2 \sqrt{\log n}}{\sqrt{n \max_{i,j} X_{i,j}^{(n)}}} + \frac{\widetilde{C}^2 C^2 \log n}{ n \sqrt{\max_{i,j} X_{i,j}^{(n)}}}\right), \label{eq:eta_tilde - eta_hat norm bound}
\end{equation}
with probability that tends to $1$ as $n\rightarrow\infty$.
Since $\{\mathcal{E}_{i,j}^{(n)}\}_{i\in[m_n],\;j\in[n]}$ are independent, $\mathbb{E} [\mathcal{E}_{i,j}^{(n)}] = 0$, $\mathbb{E} (\mathcal{E}_{i,j}^{(n)})^2 = X_{i,j}^{(n)}$, and $\mathbb{E} (\mathcal{E}_{i,j}^{(n)})^4 = X_{i,j}^{(n)} + 3(X_{i,j}^{(n)})^2$ (see~\eqref{eq:Poisson recursive moment formula}),~\cite{latala2005some} asserts that
\begin{align}
    \mathbb{E}\Vert \mathcal{E}^{(n)} \Vert_2 
    &\leq \widetilde{C}_2 \left[{ \sqrt{n \max_{i,j} X_{i,j}^{(n)}} + \sqrt{m_n \max_{i,j} X_{i,j}^{(n)}} + \left( n m_n \max_{i,j} \left\{ X_{i,j}^{(n)} + 3 (X_{i,j}^{(n)})^2 \right\} \right)^{1/4} }\right] \nonumber \\
    &\leq \widetilde{C}_3 \sqrt{ n \max_{i,j} X_{i,j}^{(n)}},
\end{align}
for all sufficiently large $n$ and some universal constants $\widetilde{C}_2,\widetilde{C}_3$, where we used the fact that $X_{i,j}^{(n)} = (X_{i,j}^{(n)})^2/X_{i,j}^{(n)} \leq C (X_{i,j}^{(n)})^{2} / (c (\log n)^{1+\epsilon})$ and that $m_n/n \rightarrow \gamma$ as $n\rightarrow \infty$. Consequently, applying Markov's inequality gives
\begin{equation}
    \Vert \mathcal{E}^{(n)} \Vert_2 \leq \log n \sqrt{ n \max_{i,j} X_{i,j}^{(n)}}, \label{eq:eta norm bound}
\end{equation}
with probability that tends to $1$ as $n\rightarrow\infty$. Combining~\eqref{eq:eta norm bound} with~\eqref{eq:eta_tilde - eta_hat norm bound}, we get
\begin{equation}
    \frac{1}{\sqrt{n}}\Vert \widetilde{\mathcal{E}}^{(n)} - \hat{ \mathcal{E}}^{(n)}  \Vert_2\leq \widetilde{C}_3 \frac{(\log n)^{3/2}}{\sqrt{n} }, \label{eq:E_tilde - E_hat bound}
\end{equation}
for some constant $\widetilde{C}_3$, with probability that tends to $1$ as $n\rightarrow\infty$. Hence, it follows that (see Theorem 3.3.16 in~\cite{horn1994topics})
\begin{equation}
    \max_{i} \left\vert {s}_i\{\frac{1}{\sqrt{n}} \widetilde{\mathcal{E}}^{(n)} \} - s_i\{\frac{1}{\sqrt{n}} \hat{\mathcal{E}}^{(n)} \} \right\vert \leq \widetilde{C}_3 \frac{(\log n)^{3/2}}{\sqrt{n} },
\end{equation}
with probability that tends to $1$ as $n\rightarrow\infty$, where $s_i\{A\}$ is the $i$'th largest singular value of $A$. We then have that
\begin{align}
    &\max_{i} \left\vert {\lambda}_i\{\frac{1}{{n}} \widetilde{\mathcal{E}}^{(n)} (\widetilde{\mathcal{E}}^{(n)})^T \} - \lambda_i\{\frac{1}{{n}} \hat{\mathcal{E}}^{(n)} (\hat{\mathcal{E}}^{(n)})^T \} \right\vert 
    = \max_{i} \left\vert {s}^2_i\{\frac{1}{\sqrt{n}} \widetilde{\mathcal{E}}^{(n)} \} - s_i^2\{\frac{1}{\sqrt{n}} \hat{\mathcal{E}}^{(n)} \} \right\vert \nonumber \\
    &= \max_{i} \left\{ \left({s}_i\{\frac{1}{\sqrt{n}} \widetilde{\mathcal{E}}^{(n)}  \} + s_i\{\frac{1}{\sqrt{n}} \hat{\mathcal{E}}^{(n)} \}\right)\cdot \left\vert {s}_i\{\frac{1}{\sqrt{n}} \widetilde{\mathcal{E}}^{(n)}  \} - s_i\{\frac{1}{\sqrt{N}} \hat{\mathcal{E}}^{(n)} \} \right\vert \right\} \nonumber \\ 
    &\leq \widetilde{C}_3 \frac{(\log n)^{3/2}}{\sqrt{n} }  \max_{i} \left\{{s}_i\{\frac{1}{\sqrt{n}} \widetilde{\mathcal{E}}^{(n)}  \} + s_i\{\frac{1}{\sqrt{n}} \hat{\mathcal{E}}^{(n)} \}\right\} = \widetilde{C}_3 \frac{(\log n)^{3/2}}{{n} }  \left( \Vert \widetilde{\mathcal{E}}^{(n)} \Vert_2 + \Vert \hat{\mathcal{E}}^{(n)} \Vert_2 \right),
\end{align}
with probability that tends to $1$ as $n\rightarrow\infty$.
Using~\eqref{eq:eta norm bound} and~\eqref{eq: u_i v_j squared upper bound} we have
\begin{align}
    \Vert \widetilde{\mathcal{E}}^{(n)} \Vert_2 
    &= \Vert D(\mathbf{u}^{(n)}){\mathcal{E}^{(n)}} D(\mathbf{v}^{(n)}) \Vert_2 
    \leq \Vert D(\mathbf{u}^{(n)}) \Vert_2 \cdot \Vert {\mathcal{E}^{(n)}} \Vert_2 \cdot \Vert D(\mathbf{v}^{(n)}) \Vert_2 \nonumber \\
    &\leq \log n \sqrt{ n \max_{i,j} X_{i,j}^{(n)}} \max_i u_i^{(n)} \max_j v_j^{(n)} 
    \leq C \sqrt{n} \log n,
\end{align}
with probability that tends to $1$ as $n\rightarrow\infty$. Employing the above together with~\eqref{eq:E_tilde - E_hat bound} we obtain 
\begin{align}
    \Vert \hat{\mathcal{E}} \Vert_2 
    &\leq \Vert \widetilde{\mathcal{E}}^{(n)} \Vert_2 + \Vert \hat{\mathcal{E}}^{(n)} - \widetilde{\mathcal{E}}^{(n)} \Vert_2 
    \leq C \sqrt{n} \log n + \widetilde{C}_3 (\log n)^{3/2},
\end{align}
with probability that tends to $1$ as $n\rightarrow\infty$.
Overall, it follows that
\begin{equation}
    \max_{i} \left\vert \lambda_i\{\frac{1}{{n}} \widetilde{\mathcal{E}}^{(n)} (\widetilde{\mathcal{E}}^{(n)})^T \} - \lambda_i\{\frac{1}{{n}} \hat{\mathcal{E}}^{(n)} (\hat{\mathcal{E}}^{(n)})^T \} \right\vert \overset{p}{\underset{n\rightarrow \infty}{\longrightarrow}} 0,
\end{equation}
where $\overset{p}{{\longrightarrow}}$ refers to convergence in probability.

\section{Proof of Proposition~\ref{prop:simple condtion for existence and uniquness for A}} \label{appendix:proof of simple condition for existence and uniqueness for A}
\subsection{Proof that $A$ is not completely decomposable}
We begin by showing that $A$ is not completely decomposable under the conditions in Proposition~\ref{prop:existcne and uniquness for A}.
To that end, assume in negation that $A$ is completely decomposable. Then, since $A$ does not have any zero rows and columns, there must exist proper nonempty subsets $\mathcal{I}_1\subset [m]$ and $\mathcal{I}_2\subset [n]$ such that $[A]_{i\in \mathcal{I}_1, \; j \in \mathcal{I}_2}$ and $[A]_{i\in \mathcal{I}_1^c, \; j \in \mathcal{I}_2^c}$ are both zero matrices. We can assume without loss of generality that $|\mathcal{I}_2| \geq n/2$, as otherwise we simply replace $(\mathcal{I}_1,\mathcal{I}_2)$ with $(\mathcal{I}_1^c,\mathcal{I}_2^c)$. Let us define $k = n - |\mathcal{I}_2|$, and since $|\mathcal{I}_2| \geq n/2$, we have that $n \geq 2k$. Now, if $|\mathcal{I}_1| \geq m/2$ we immediately get a contradiction to condition~\ref{Cond: cond 2 in prop} in Proposition~\ref{prop:existcne and uniquness for A}, since $[A]_{i\in \mathcal{I}_1, \; j \in \mathcal{I}_2}$ has $|\mathcal{I}_1|$ rows with $|\mathcal{I}_2| = n-k$ zeros each, where
\begin{equation}
    |\mathcal{I}_1| \geq \frac{m}{2} = \frac{m k}{n} \frac{n}{2k} \geq \frac{m k}{n},
\end{equation}
and we used the fact that $n \geq 2k$.
We next consider the alternative possibility that $|\mathcal{I}_1| < m/2$. In this case, we must have that 
\begin{equation}
    |\mathcal{I}_1| < \frac{m k}{n}, \label{eq:I_1 cardinality bound}
\end{equation}
as otherwise we get a contradiction to condition~\ref{Cond: cond 2 in prop} in Proposition~\ref{prop:existcne and uniquness for A} (using the fact that $|\mathcal{I}_1|$ is an integer).
In addition, the matrix $[A]_{i\in \mathcal{I}_1^c, \; j \in \mathcal{I}_2^c}$ has $|\mathcal{I}_2^c|=n - |\mathcal{I}_2|$ columns that have $|\mathcal{I}_1^c| = m - | \mathcal{I}_1|$ zeros each. We define $\ell = | \mathcal{I}_1|$, which satisfies $\ell < m/2$. Therefore, we also must have that
\begin{equation}
    n - |\mathcal{I}_2| < \frac{n \ell}{m} = \frac{n | \mathcal{I}_1|}{m} < k, \label{eq:I_2 cardinality bound}
\end{equation}
as otherwise we get a contradiction to condition~\ref{Cond: cond 3 in prop} in Proposition~\ref{prop:existcne and uniquness for A}, where we used~\eqref{eq:I_1 cardinality bound} and the fact that $n - |\mathcal{I}_2|$ is an integer. Overall, recall that $k = n - |\mathcal{I}_2|$, which together with~\eqref{eq:I_2 cardinality bound} gives $k < k$, a contradiction to our initial assumption that $A$ is completely decomposable. 

\subsection{Proof that Condition~\ref{cond:existence of x and y for nonegative A} holds}
We next prove that that Condition~\ref{cond:existence of x and y for nonegative A} holds. 
Let us assume in negation that Condition~\ref{cond:existence of x and y for nonegative A} does not hold. Then, since $A$ does not have any zero rows and columns, we must have that
\begin{equation}
    |\mathcal{I}_1| n + | \mathcal{I}_2 | m \geq m n,
\end{equation} 
for some proper nonempty subsets $\mathcal{I}_1\subset [m]$ and $\mathcal{I}_2\subset [n]$ for which $[A]_{i\in \mathcal{I}_1, \; j \in \mathcal{I}_2}$ is a zero matrix. It follows that either $|\mathcal{I}_2| \geq n/2$ or $|\mathcal{I}_1|\geq m/2$. Suppose that the former holds, that is $|\mathcal{I}_2|\geq n/2$, and define $k = n - |\mathcal{I}_2|$ which satisfies $0 < k \leq n/2$. Observe that the matrix $[A]_{i\in \mathcal{I}_1, \; j \in \mathcal{I}_2}$ has $|\mathcal{I}_1|$ rows with $|\mathcal{I}_2| = n - k$ zeros each. Using the fact that $|\mathcal{I}_1| n + | \mathcal{I}_2 | m \geq m n$ we have
\begin{equation}
    |\mathcal{I}_1| \geq \frac{m (n - | \mathcal{I}_2 |)}{n} = \frac{m k}{n},
\end{equation}
which is a contradiction to condition~\ref{Cond: cond 2 in prop} in Proposition~\ref{prop:existcne and uniquness for A} (using the fact that $|\mathcal{I}_1|$ is an integer). We next assume that the other possibility holds, namely that $|\mathcal{I}_1|\geq m/2$, and define $\ell= m - |\mathcal{I}_1|$ which satisfies $0 < \ell \leq m/2$. Observe that the matrix $[A]_{i\in \mathcal{I}_1, \; j \in \mathcal{I}_2}$ has $|\mathcal{I}_2|$ columns with $|\mathcal{I}_1| = m - \ell$ zeros each. Using the fact that $|\mathcal{I}_1| n + | \mathcal{I}_2 | m \geq m n$ we have
\begin{equation}
    |\mathcal{I}_2| \geq \frac{n (m - | \mathcal{I}_1 |)}{m} = \frac{n \ell}{m},
\end{equation}
which is a contradiction to condition~\ref{Cond: cond 3 in prop} in Proposition~\ref{prop:existcne and uniquness for A} (using the fact that $|\mathcal{I}_2|$ is an integer). Therefore, we have a contradiction to our assumption that Condition~\ref{cond:existence of x and y for nonegative A} does not hold, thereby concluding the proof.

\section{Proof of Proposition~\ref{prop:existence of variance estimator}} \label{appendix:proof of prop on existence of variance estimator}
The fact that~\eqref{eq:quadratic noise variance estimator} is an unbiased estimator for $\operatorname{Var}[Y_{i,j}]$ follows from direct calculation, as
\begin{align}
    \mathbb{E}[\widehat{\operatorname{Var}}[Y_{i,j}]] &= \frac{a + b X_{i,j} + c \mathbb{E}[Y_{i,j}^2]}{1 + c} \nonumber \\
    &= \frac{a + b X_{i,j} + c \left( a + b X_{i,j} + (c + 1)X_{i,j}^2 \right)}{1 + c} =  \operatorname{Var}[Y_{i,j}], \label{eq:quadratic variance derivation}
\end{align}
where we used $\mathbb{E}[Y_{i,j}^2] = \operatorname{Var}[Y_{i,j}] + X_{i,j}^2$ together with the QVF property~\eqref{eq: quadratic variance}.
Next, if $Y_{i,j}\sim \operatorname{Bernoulli}(p_{i,j})$ then $\operatorname{Var}[Y_{i,j}] = p_{i,j}-p_{i,j}^2$, hence $c=-1$. Among the six fundamental NEF-QVFs, the binomial is the only family with $c<0$, and according to the formulas in~\cite{morris1982natural}, 
while the value of $c$ is invariant to linear transformations, it must change under a (non-null) convolution or division. Hence, the case of $c=-1$ corresponds uniquely to a linear transformation of a Bernoulli. Lastly, for any $f:\{0,1\} \rightarrow\{c_0,c_1\}$ we have $\mathbb{E}[f(Y_{i,j})] = c_0 p_{i,j} + c_1 (1-p_{i,j})$, which is a first degree polynomial in $p_{i,j}$ that cannot possibly match $\operatorname{Var}[Y_{i,j}]$ for all values of $p_{i,j}\in  [0,1]$.

\end{appendices}

\bibliographystyle{plain}
\bibliography{mybib}

\end{document}